\documentclass[11pt]{amsbook}
\usepackage{amsfonts, amsthm, amssymb, amsmath, stmaryrd, color, enumerate, relsize}
\usepackage[pdftex]{graphicx}
\usepackage{mathrsfs,array}
\usepackage{xy}
\usepackage{hyperref}
\usepackage{tikz}
\input xy
\xyoption{all}

\DeclareMathOperator{\Spa}{Spa}

\DeclareMathOperator{\Spv}{Spv}

\DeclareMathOperator{\Cont}{Cont}

\DeclareMathOperator{\Hom}{Hom}
\DeclareMathOperator{\Ext}{Ext}
\newcommand{\intHom}{\underline{\Hom}}

\DeclareMathOperator{\Mod}{Mod}

\DeclareMathOperator{\Spec}{Spec}

\DeclareMathOperator{\Sets}{Sets}

\newcommand{\AnRing}{{\mathrm{AnRing}}}
\DeclareMathOperator{\AnSpec}{AnSpec}

\DeclareMathOperator{\op}{op}

\DeclareMathOperator{\Cond}{Cond}
\DeclareMathOperator{\Ab}{Ab}

\DeclareMathOperator{\Fun}{Fun}

\DeclareMathOperator{\ProFin}{ProFin}
\DeclareMathOperator{\Pro}{Pro}

\DeclareMathOperator{\Fin}{Fin}

\DeclareMathOperator{\colim}{colim}

\renewcommand{\op}{{\mathrm{op}}}

\newcommand{\proet}{{\mathrm{pro\acute{e}t}}}

\newcommand{\heart}{\heartsuit}

\def\A{\mathbf{A}}

\newcommand{\solid}{{\mathsmaller{\square}}}
\DeclareMathOperator{\Solid}{Solid}

\renewcommand*{\hat}{\widehat}
\renewcommand*{\tilde}{\widetilde}

\numberwithin{equation}{section}
\newtheorem{theorem}{Theorem}
\numberwithin{theorem}{section}
\newtheorem{lemma}[theorem]{Lemma}
\newtheorem{corollary}[theorem]{Corollary}
\newtheorem{proposition}[theorem]{Proposition}

\newtheorem{assumption}[theorem]{Assumption}

\theoremstyle{definition}
\newtheorem{remark}[theorem]{Remark}
\newtheorem{exercise}[theorem]{Exercise}
\newtheorem{warning}[theorem]{Warning}
\newtheorem{definition}[theorem]{Definition}
\newtheorem{question}[theorem]{Question}
\newtheorem{example}[theorem]{Example}

\date{\today}

\title{Lectures on Analytic Geometry}

\author{Peter Scholze}

\begin{document}

\maketitle

\tableofcontents

\chapter*{Lectures on Analytic Geometry}

\section*{Preface}

This is a slightly updated version of lectures notes for a course on analytic geometry taught in the winter term 2019/20 at the University of Bonn. The material presented is part of joint work with Dustin Clausen.\\

This is intended as a stable citable version of the material. In the first half of this course, we develop the basic theory of liquid real vector spaces, which is used in \cite{Complex} to give a new approach to complex-analytic geometry. In the second half, we gave a tentative definition of a category of analytic spaces that contains (for example) adic spaces and complex-analytic spaces. While the precise definition of analytic spaces represents an abandoned stepping stone on our way to define analytic stacks in \cite{AnalyticStacks} and hence should be seen as a historical artifact, much of the surrounding discussion stays very relevant.\\

\hfill{May 2026, Peter Scholze}

\newpage

\section{Lecture I: Introduction}

Mumford writes in \emph{Curves and their Jacobians}: ``[Algebraic geometry] seems to have acquired the reputation of being esoteric, exclusive, and very abstract, with adherents who are secretly plotting to take over all the rest of mathematics. In one respect this last point is accurate.''

For some reason, this secret plot has so far stopped short of taking over analysis. The goal of this course is to launch a new attack, turning functional analysis into a branch of commutative algebra, and various types of analytic geometry (like manifolds) into algebraic geometry. Whether this will make these subjects equally esoteric will be left to the reader's judgement.

What do we mean by analytic geometry? We mean a version of algebraic geometry that
\begin{enumerate}
\item instead of merely allowing polynomial rings as its basic building blocks, allows rings of convergent power series as basic building blocks;
\item instead of being able to define open subsets only by the nonvanishing of functions, one can define open subsets by asking that a function is small, say less than $1$;
\item strictly generalizes algebraic geometry in the sense that the category of schemes, the theory of quasicoherent sheaves over them, etc., all embed fully faithfully into the corresponding analytic category.
\end{enumerate}

The author always had the impression that the highly categorical techniques of algebraic geometry could not possibly be applied in analytic situations; and certainly not over the real numbers.\footnote{In the nonarchimedean case, the theory of adic spaces goes a long way towards fulfilling these goals, but it has its shortcomings: Notably, it lacks a theory of quasicoherent sheaves, and in the nonnoetherian case the general formalism does not work well (e.g.~the structure presheaf fails to be sheaf). Moreover, the language of adic spaces cannot easily be modified to cover complex manifolds. This is possible by using Berkovich spaces, but again there is no theory of quasicoherent sheaves etc.} The goal of this course is to correct this impression.

How can one build algebraic geometry? One perspective is that one starts with the abelian category $\mathrm{Ab}$ of abelian groups, with its symmetric monoidal tensor product. Then one can consider rings $R$ in this category (which are just usual rings), and over any ring $R$ one can consider modules $M$ in that category (which are just usual $R$-modules). Now to any $R$, one associates the space $\Spec R$, essentially by declaring that (basic) open subsets of $\Spec R$ correspond to localizations $R[f^{-1}]$ of $R$. One can then glue these $\Spec R$'s together along open subsets to form schemes, and accordingly the category of $R$-modules glues to form the category of quasicoherent sheaves.

Now how to build analytic geometry? We are basically stuck with the first step: We want an abelian category of some kind of topological abelian groups (together with a symmetric monoidal tensor product that behaves reasonably). Unfortunately, topological abelian groups do not form an abelian category: A map of topological abelian groups that is an isomorphism of underlying abelian groups but merely changes the topology, say
\[
(\mathbb R,\mathrm{discrete\ topology})\to (\mathbb R,\mathrm{natural\ topology}),
\]
has trivial kernel and cokernel, but is not an isomorphism.

This problem was solved in the course on condensed mathematics last semester, \cite{Condensed}, by replacing the category of topological spaces with the much more algebraic category of condensed sets, and accordingly topological abelian groups with abelian group objects in condensed sets, i.e.~condensed abelian groups. Let us recall the definition.

\begin{definition} Consider the pro-\'etale site $\ast_\proet$ of the point $\ast$, i.e.~the category $\ProFin\cong \Pro(\Fin)$ of profinite sets $S$ with covers given by finite families of jointly surjective maps. A condensed set is a sheaf on $\ast_\proet$; similarly, a condensed abelian group, ring, etc.~is a sheaf of abelian groups, rings, etc.~on $\ast_\proet$.\footnote{As discussed last semester, this definition has minor set-theoretic issues. We explained then how to resolve them; we will mostly ignore the issues in these lectures.}
\end{definition}

As for sheaves on any site, it is then true that a condensed abelian group, ring, etc.~is the same as an abelian group object/ring object/etc.~in the category of condensed sets: for example, a condensed abelian group is a condensed set $M$ together with a map $M\times M\to M$ (the addition map) of condensed sets making certain diagrams commute that codify commutativity and associativity, and admits an inverse map.

More concretely, a condensed set is a functor
\[
X: \ProFin^\op\to \Sets
\]
such that
\begin{enumerate}
\item one has $X(\emptyset)=\ast$, and for all profinite sets $S_1$, $S_2$, the map $X(S_1\sqcup S_2)\to X(S_1)\times X(S_2)$ is bijective;
\item for any surjection $S^\prime\to S$ of profinite sets, the map
\[
X(S)\to \{x\in X(S')\mid p_1^\ast(x) = p_2^\ast(x)\in X(S'\times_S S')\}
\]
is bijective, where $p_1,p_2: S'\times_S S'\to S'$ are the two projections.
\end{enumerate}

How to think about a condensed set? The value $X(\ast)$ should be thought of as the underlying set, and intuitively $X(S)$ is the space of continuous maps from $S$ into $X$. For example, if $T$ is a topological space, one can define a condensed $X=\underline{T}$ via
\[
\underline{T}(S) = \Cont(S,T),
\]
the set of continuous maps from $S$ into $T$. One can verify that this defines a condensed set.\footnote{There are some set-theoretic subtleties with this assertion if $T$ does not satisfy the separation axiom $T1$, i.e.~its points are not closed; we will only apply this functor under this assumption.} Part (1) is clear, and for part (2) the key point is that any surjective map of profinite sets is a quotient map.

For example, in analysis a central notion is that of a sequence $x_0,x_1,\ldots$ converging to $x_\infty$. This is codified by maps from the profinite set $S=\{0,1,\ldots,\infty\}$ (the one-point compactification $\mathbb N\cup \{\infty\}$ of $\mathbb N$) into $X$. Allowing more general profinite sets makes it possible to capture more subtle convergence behaviour. For example, if $T$ is a compact Hausdorff space and you have any sequence of points $x_0,x_1,\ldots\in T$, then it is not necessarily the case that it converges in $T$; but one can always find convergent subsequences. More precisely, for each ultrafilter $\mathcal U$ on $\mathbb N$, one can take the limit along the ultrafilter. In fact, this gives a map $\beta \mathbb N\to T$ from the space $\beta \mathbb N$ of ultrafilters on $\mathbb N$. The space $\beta \mathbb N$ is a profinite set, known as the Stone--\v{C}ech compactification of $\mathbb N$; it is the initial compact Hausdorff space to which $\mathbb N$ maps (by an argument along these lines).

We have the following result about the relation between topological spaces and condensed sets, proved last semester except for the last assertion in part (4), which can be found for example in \cite[Lemma 4.3.7]{BhattScholzeProetale}.\footnote{Again, some of these assertions run into minor set-theoretic problems; we refer to the notes from last semester. Importantly, everything is valid on the nose when restricted to topological spaces with closed points, and quasiseparated condensed sets; in particular, points (3) and (4) are true as stated.}

\begin{proposition}\label{prop:condtop} Consider the functor $T\mapsto \underline{T}$ from topological spaces to condensed sets.
\begin{enumerate}
\item The functor has a left adjoint $X\mapsto X(\ast)_{\mathrm{top}}$ sending any condensed set $X$ to its underlying set $X(\ast)$ equipped with the quotient topology arising from the map
\[
\bigsqcup_{S,a\in X(S)} S\to X(\ast).
\]
\item Restricted to compactly generated (e.g., first-countable, e.g., metrizable) topological spaces, the functor is fully faithful.
\item The functor induces an equivalence between the category of compact Hausdorff spaces, and quasicompact quasiseparated condensed sets.
\item The functor induces a fully faithful functor from the category of compactly generated weak Hausdorff spaces (the standard ``convenient category of topological spaces'' in algebraic topology), to quasiseparated condensed sets. The category of quasiseparated condensed sets is equivalent to the category of ind-compact Hausdorff spaces $``\varinjlim_i" T_i$ where all transition maps $T_i\to T_j$ are closed immersions. If $X_0\hookrightarrow X_1\hookrightarrow \ldots$ is a sequence of compact Hausdorff spaces with closed immersions and $X=\varinjlim_n X_n$ as a topological space, the map
\[
\varinjlim_n \underline{X_n}\to \underline{X}
\]
is an isomorphism of condensed sets. In particular, $\varinjlim_n \underline{X_n}$ comes from a topological space.
\end{enumerate}
\end{proposition}

Here the notions of quasicompactness/quasiseparatedness are general notions applying to sheaves on any (coherent) site. In our case, a condensed set $X$ is quasicompact if there is some profinite $S$ with a surjective map $S\to X$. A condensed set $X$ is quasiseparated if for any two profinite sets $S_1,S_2$ with maps to $X$, the fibre product $S_1\times_X S_2$ is quasicompact.

The functor in (4) is close to an equivalence, and in any case (4) asserts that quasiseparated condensed sets can be described in very classical terms. Let us also mention the following related result.

\begin{lemma}\label{lem:compactgenproduct} Let $X_0\hookrightarrow X_1\hookrightarrow \ldots$ and $Y_0\hookrightarrow Y_1\hookrightarrow\ldots$ be two sequences of compact Hausdorff spaces with closed immersions. Then, inside the category of topological spaces, the natural map
\[
\bigcup_n X_n\times Y_n\to (\bigcup_n X_n)\times (\bigcup_n Y_n)
\]
is a homeomorphism; i.e.~the product on the right is equipped with its compactly generated topology.
\end{lemma}

\begin{proof} The map is clearly a continuous bijection. In general, for a union like $\bigcup_n X_n$, open subsets $U$ are the subsets of the form $\bigcup_n U_n$ where each $U_n\subset X_n$ is open. Thus, let $U\subset \bigcup_n X_n\times Y_n$ be any open subset, written as a union of open subset $U_n\subset X_n\times Y_n$, and pick any point $(x,y)\in U$. Then for any large enough $n$ (so that $(x,y)\in X_n\times Y_n$), we can find open neighborhoods $V_n\subset X_n$ of $x$ in $X_n$ and $W_n\subset Y_n$ of $y$ in $Y_n$, such that $V_n\times W_n\subset U_n$. In fact, we can ensure that even $\overline{V}_n\times \overline{W}_n\subset U_n$ by shrinking $V_n$ and $W_n$. Constructing the $V_n$ and $W_n$ inductively, we may then moreover ensure $\overline{V}_n\subset V_{n+1}$ and $\overline{W}_n\subset W_{n+1}$. Then $V=\bigcup_n V_n\subset \bigcup_n X_n$ and $W=\bigcup_n W_n\subset \bigcup_n Y_n$ are open, and $V\times W = \bigcup_n V_n\times W_n\subset U$ contains $(x,y)$, showing that $U$ is open in the product topology.
\end{proof}

We asserted above that one should think of $X(\ast)$ as the underlying set of $X$, and about $X(S)$ as the continous maps from $S$ into $X$. Note however that in general the map
\[
X(S)\to \prod_{s\in S} X(\{s\}) = \prod_{s\in S} X(\ast) = \mathrm{Map}(S,X(\ast))
\]
may not be injective, and in fact it may and does happen that $X(\ast)=\ast$ but $X(S)\neq \ast$ for large $S$. However, for quasiseparated condensed sets, this map is always injective (as follows from (4) above), and in particular a quasiseparated condensed set $X$ is trivial as soon as $X(\ast)$ is trivial.

The critical property of condensed sets is however exactly that they can also handle non-Hausdorff situations (i.e., non-quasiseparated situations in the present technical jargon) well. For example, condensed abelian groups, like sheaves of abelian groups on any site, form an abelian category. Considering the example from above and passing to condensed abelian groups, we get a short exact sequence
\[
0\to \underline{(\mathbb R, \mathrm{discrete\ topology})}\to \underline{(\mathbb R,\mathrm{natural\ topology})}\to Q\to 0
\]
of condensed abelian groups, for a condensed abelian group $Q$ satisfying
\[
Q(S) = \Cont(S,\mathbb R)/\{\mathrm{locally\ constant\ maps\ } S\to \mathbb R\}
\]
for any profinite set $S$.\footnote{It is nontrivial that this formula for $Q$ is correct: A priori, this describes the quotient on the level of presheaves, and one might have to sheafify. However, using $H^1(S,M)=0$ for any profinite set $S$ and discrete abelian group $M$, as proved last semester, one gets the result by using the long exact cohomology sequence.} In particular $Q(\ast)=0$ while $Q(S)\neq 0$ for general $S$: There are plenty of non-locally constant maps $S\to \mathbb R$ for profinite sets $S$, say any convergent sequence that is not eventually constant. In particular, $Q$ is not quasiseparated. We see that enlarging topological abelian groups into an abelian category precisely forces us to include non-quasiseparated objects, in such a way that a quotient $M_1/M_2$ still essentially remembers the topology on both $M_1$ and $M_2$.

At this point, we have our desired abelian category, the category $\Cond(\Ab)$ of condensed abelian groups. Again, like for sheaves of abelian groups on any site, it has a symmetric monoidal tensor product (representing bilinear maps in the obvious way). We could then follow the same steps as for schemes. However, the resulting theory does not yet achieve our stated goals:

\begin{enumerate}
\item The basic building blocks in algebraic geometry, the polynomial rings, are exactly the free rings on some set $I$. As sets are generated by finite sets, really the case of a finite set is relevant, giving rise to the polynomial algebras $\mathbb Z[X_1,\ldots,X_n]$. Similarly, the basic building blocks are now the free rings on a condensed set; as these are generated by profinite sets, one could also just take the free rings on a profinite set $S$. But for a profinite set $S$, the corresponding free condensed ring generated by $S$ is not anything like a ring of (convergent) power series. In fact, the underlying ring is simply the free ring on the underlying set of $S$, i.e.~an infinite polynomial algebra.

To further illustrate this point, let us instead consider the free condensed ring $A$ equipped with an element $T\in A$ with the sequence $T,T^2,\ldots,T^n,\ldots$ converging $0$: This is given by $A=\mathbb Z[S]/([\infty]=0)$ for the profinite set $S=\mathbb N\cup \{\infty\}$. The underlying ring of $A$ is then still the polynomial algebra $\mathbb Z[T]$; it is merely equipped with a nonstandard condensed ring structure.
\item One cannot define open subsets by ``putting bounds on functions''. If say over $\mathbb Q_p$ we want to make the locus $\{|T|\leq 1\}$ into an open subset of the affine line, and agree that this closed unit disc should correspond to the algebra of convergent power series
\[
\mathbb Q_p\langle T\rangle = \{\sum_{n\geq 0} a_n T^n\mid a_n\to 0\}
\]
while the affine line corresponds to $\mathbb Q_p[T]$, then being a localization should mean that the multiplication map
\[
\mathbb Q_p\langle T\rangle\otimes_{\mathbb Q_p[T]} \mathbb Q_p\langle T\rangle\to \mathbb Q_p\langle T\rangle
\]
should be an isomorphism. However, this is not an isomorphism of condensed abelian groups: On underlying abelian groups, the tensor product is just the usual algebraic tensor product.
\end{enumerate}

These failures are not unexpected: We did not yet put in any nontrivial analysis. Somewhere we have to specify which kinds of convergent power series we want to use. Concretely, in a condensed ring $R$ with a sequence $T,T^2,\ldots,T^n,\ldots$ converging to $0$, we want to allow certain power series $\sum_{n\geq 0} r_n T^n$ where not almost all of the coefficients $r_n$ are zero. Hopefully, this specification maintains a nice abelian category that acquires its own tensor product, in such a way that now the tensor product computation in point (2) above works out.

Last semester, we developed such a formalism that works very well in nonarchimedean geometry: This is the formalism of solid abelian groups that we will recall in the next lecture. In particular, the solidification of the ring $A$ considered in (1) is exactly the power series ring $\mathbb Z[[T]]$ (with its condensed ring structure coming from the usual topology on the power series ring), and after solidification the tensor product equation in (2) becomes true. Unfortunately, the real numbers $\mathbb R$ are not solid: Concretely, $T=\frac 12\in \mathbb R$ has the property that its powers go to $0$, but not any sum $\sum r_n (\frac 12)^n$ with coefficients $r_n\in \mathbb Z$ converges in $\mathbb R$. Thus, the solid formalism breaks down completely over $\mathbb R$. Now maybe that was expected? After all, the whole formalism is based on the paradigm of resolving nice compact Hausdorff spaces like the interval $[0,1]$ by profinite sets, and this does not seem like a clever thing to do over the reals; over totally disconnected rings like $\mathbb Q_p$ it of course seems perfectly sensible.

However, last semester we promised a theorem on how the formalism might be adapted to cover the reals. The first main goal of this course will be to prove this theorem.

\newpage

\section{Lecture II: Solid modules}

In this lecture, we recall the theory of solid modules that was developed last semester.

Recall from the last lecture that we want to put some ``completeness'' condition on modules in such a way that the free objects behave like some kind of power series. To understand the situation, let us first analyze the structure of free condensed abelian groups.

\begin{proposition}\label{prop:freecondensedgroup} Let $S=\varprojlim_i S_i$ be a profinite set, written as an inverse limit of finite sets $S_i$. For any $n$, let $\mathbb Z[S_i]_{\leq n}\subset \mathbb Z[S_i]$ be the subset of formal sums $\sum_{s\in S_i} n_s [s]$ such that $\sum |n_s|\leq n$; this is a finite set, and the natural transition maps $\mathbb Z[S_j]\to \mathbb Z[S_i]$ preserve these subsets. There is a natural isomorphism of condensed abelian groups
\[
\mathbb Z[S]\cong \bigcup_n \varprojlim_i \mathbb Z[S_i]_{\leq n}\subset \varprojlim_i \mathbb Z[S_i] .
\]
In particular, $\mathbb Z[S]$ is a countable union of the profinite sets $\mathbb Z[S]_{\leq n} := \varprojlim_i \mathbb Z[S_i]_{\leq n}$.
\end{proposition}

Note that the right-hand side indeed defines a subgroup: The addition on $\varprojlim_i \mathbb Z[S_i]$ takes
\[
\mathbb Z[S_i]_{\leq n}\times \mathbb Z[S_i]_{\leq n'}
\]
into $\mathbb Z[S_i]_{\leq n+n'}$. We also remark that the bound imposed is as an $\ell^1$-bound, but in fact it is equivalent to an $\ell^0$-bound, as only finitely many $n_s$ (in fact, $n$ of them) can be nonzero.

\begin{proof} By definition, $\mathbb Z[S]$ is the free condensed abelian group on $S$, and this is formally given by the sheafification of the functor $T\mapsto \mathbb Z[\Cont(T,S)]$. First, we check that the map
\[
\mathbb Z[S]\to \varprojlim_i \mathbb Z[S_i]
\]
is an injection. First, we observe that the map of underlying abelian groups is injective. This means that given any finite formal sum $\sum_{j=1}^k n_j [s_j]$ where the $s_j\in S$ are distinct elements and $n_j\neq 0$ are integers, one can find some projection $S\to S_i$ such that the image in $\mathbb Z[S_i]$ is nonzero. But we can arrange that the images of the $s_j$ in $S_i$ are all distinct, giving the result.

Now, assume that $f\in \mathbb Z[\Cont(T,S)]$ maps to $0$ in $\varprojlim_i \mathbb Z[S_i](T)$. In particular, for all $t\in T$, the specialization $f(t)\in \mathbb Z[S]$ is zero by the injectivity on underlying abelian groups. We have to see that there is some finite cover of $T$ by profinite sets $T_m\to T$ such that the preimage of $f$ in each $\mathbb Z[\Cont(T_m,S)]$ is zero. Write $f=\sum_{j=1}^k n_j [g_j]$ where $g_j: T\to S$ are distinct continuous functions and $n_j\neq 0$. We argue by induction on $k$. For each pair $1\leq j<j'\leq k$, let $T_{jj'}\subset T$ be the closed subset where $g_j=g_{j'}$. Then the $T_{jj'}$ cover $T$: Indeed, if $t\in T$ does not lie in any $T_{jj'}$, then all $g_j(t)\in S$ are pairwise distinct, and then $\sum_{j=1}^k n_j [g_j(t)]\in \mathbb Z[S]$ is nontrivial. Thus, we may pass to the cover by the $T_{jj'}$ and thereby assume that $g_j=g_{j'}$ for some $j\neq j'$. But this reduces $k$, so we win by induction.

As observed before, $\bigcup_n \varprojlim_i \mathbb Z[S_i]_{\leq n}$ defines a condensed abelian group, and it admits a map from $S=\varprojlim_i S_i$; in particular, as $\mathbb Z[S]$ is the free condensed abelian group on $S$, the map $\mathbb Z[S]\to \varprojlim_i \mathbb Z[S_i]$ factors over $\bigcup_n \varprojlim_i \mathbb Z[S_i]_{\leq n}$. It remains to see that the induced map
\[
\mathbb Z[S]\to \bigcup_n \varprojlim_i\mathbb Z[S_i]_{\leq n}
\]
is surjective. For this, consider the map
\[
S^n\times \{-1,0,1\}^n=\varprojlim_i S_i^n\times \{-1,0,1\}^n\to \varprojlim_i \mathbb Z[S_i]_{\leq n}
\]
given by $(x_1,\ldots,x_n,a_1,\ldots,a_n)\mapsto a_1[x_1]+\ldots+a_n[x_n]$. This is a cofiltered limit of surjections of finite sets, and so a surjection of profinite sets. This map sits in a commutative diagram
\[\xymatrix{
S^n\times \{-1,0,1\}^n\ar[r]\ar[d] & \varprojlim_i \mathbb Z[S_i]_{\leq n}\ar[d]\\
\mathbb Z[S]\ar[r]& \bigcup_{m} \varprojlim_i \mathbb Z[S_i]_{\leq m}
}\]
and so implies that the lower map contains $\mathbb Z[S_i]_{\leq n}$ in its image. As this works for any $n$, we get the desired result.
\end{proof}

\begin{remark} In particular, we see that the condensed set $\mathbb Z[S]$ is quasiseparated, and in fact comes from a compactly generated weak Hausdorff topological space $\mathbb Z[S]_{\mathrm{top}}$, by Proposition~\ref{prop:condtop}~(4). Moreover, by Lemma~\ref{lem:compactgenproduct}, the addition
\[
\mathbb Z[S]_{\mathrm{top}}\times \mathbb Z[S]_{\mathrm{top}}\to \mathbb Z[S]_{\mathrm{top}}
\]
is continuous (not only when the source is equipped with its compactly generated topology), so $\mathbb Z[S]$ really comes from a topological abelian group $\mathbb Z[S]_{\mathrm{top}}$. This is in fact the free topological abelian group on $S$, cf.~\cite{ArhangelskiiTkachenko}.
\end{remark}

\begin{exercise} Prove that for any compact Hausdorff space $S$, the condensed abelian group $\mathbb Z[S]$ can naturally be written as a countable union $\bigcup_n \mathbb Z[S]_{\leq n}$ of compact Hausdorff spaces $\mathbb Z[S]_{\leq n}$, and comes from a topological abelian group $\mathbb Z[S]_{\mathrm{top}}$.
\end{exercise}

The idea of solid modules is that we would want to enlarge $\mathbb Z[S]$, allowing more sums deemed ``convergent'', and an obvious possibility presents itself:
\[
\mathbb Z[S]^\solid := \varprojlim_i \mathbb Z[S_i] .
\]
In particular, this ensures that
\[
\mathbb Z[\mathbb N\cup\{\infty\}]^\solid/([\infty]=0) = \varprojlim_n \mathbb Z[\{0,1,\ldots,n-1,\infty\}]/([\infty]=0) = \varprojlim_n \mathbb Z[T]/T^n = \mathbb Z[[T]]
\]
is the power series algebra.\footnote{Here, $\mathbb Z[T]$ denotes the polynomial algebra in a variable $T$, not the free condensed module on a profinite set $T$.}

In other words, we want to pass to a subcategory $\Solid\subset \Cond(\Ab)$ with the property that the free solid abelian group on a profinite set $S$ is $\mathbb Z[S]^\solid$. This is codified in the following definition.

\begin{definition}\label{def:solid} A condensed abelian group $M$ is solid if for any profinite set $S$ with a map $f: S\to M$, there is a unique map of condensed abelian groups $\tilde{f}: \mathbb Z[S]^\solid\to M$ such that the composite $S\to \mathbb Z[S]^\solid\to M$ is the given map $f$.
\end{definition}

Let us immediately state the main theorem on solid abelian groups.

\begin{theorem}[{\cite[Theorem 5.8 (i), Theorem 6.2 (i)]{Condensed}}]\label{thm:solid} The category $\Solid$ of solid abelian groups is, as a subcategory of the category $\Cond(\Ab)$ of condensed abelian groups, closed under all limits, colimits, and extensions. The inclusion functor admits a left adjoint $M\mapsto M^\solid$ (called solidification), which, as an endofunctor of $\Cond(\Ab)$, commutes with all colimits and takes $\mathbb Z[S]$ to $\mathbb Z[S]^\solid$. The category of solid abelian groups admits compact projective generators, which are exactly the condensed abelian groups of the form $\prod_I \mathbb Z$ for some set $I$. There is a unique symmetric monoidal tensor product $\otimes^\solid$ on $\Solid$ making $\Cond(\Ab)\to \Solid: M\mapsto M^\solid$ symmetric monoidal.
\end{theorem}

The theorem is rather nontrivial. Indeed, a small fraction of it asserts that $\mathbb Z$ is solid. What does this mean? Given a profinite set $S=\varprojlim_i S_i$ and a continuous map $f: S\to \mathbb Z$, there should be a unique map $\tilde{f}: \mathbb Z[S]^\solid\to \mathbb Z$ with given restriction to $S$. Note that
\[
\Cont(S,\mathbb Z)=\varinjlim_i \Cont(S_i,\mathbb Z)=\varinjlim_i\Hom(\mathbb Z[S_i],\mathbb Z)\to \Hom(\varprojlim_i \mathbb Z[S_i],\mathbb Z)=\Hom(\mathbb Z[S]^\solid,\mathbb Z).
\]
In particular, the existence of $f$ is clear, but the uniqueness is not. We need to see that any map $\mathbb Z[S]^\solid\to \mathbb Z$ of condensed abelian groups factors over $\mathbb Z[S_i]$ for some $i$. This seems to be hard to prove by a direct attack (if $S$ is a countable limit of finite sets, it is however possible). We note in particular that we do not know whether the similar result holds with $\mathbb Z$ replaced by any ring $A$ (and in fact the naive translation of the theorem above fails for general rings).

Let us analyze the structure of $\mathbb Z[S]^\solid$. It turns out that it can be regarded as a space of measures. More precisely,
\[
\mathbb Z[S]^\solid = \varprojlim_i \mathbb Z[S_i] = \varprojlim_i \intHom(C(S_i,\mathbb Z),\mathbb Z)=\intHom(\varinjlim_i C(S_i,\mathbb Z),\mathbb Z)=\intHom(C(S,\mathbb Z),\mathbb Z)
\]
is the space dual to the (discrete) abelian group $C(S,\mathbb Z)$ of continuous functions $S\to \mathbb Z$. Accordingly, we will often denote elements of $\mathbb Z[S]^\solid$ as $\mu\in \mathbb Z[S]^\solid$ and refer to them as measures.

Thus, in a solid abelian group, it holds true that whenever $f: S\to M$ is a continuous map, and $\mu\in \mathbb Z[S]^\solid$ is a measure, one can form the integral $\int_S f \mu := \tilde{f}(\mu)\in M$. Again, an important special case is when $S=\mathbb N\cup\{\infty\}$. In that case, $f$ can be thought of as a convergent sequence $m_0,m_1,\ldots,m_n,\ldots,m_\infty$ in $M$, and a measure $\mu$ on $S$ can be characterized by the masses $a_0,a_1,\ldots,a_n,\ldots$ it gives to the finite points (which are isolated in $S$), as well as the mass $a$ it gives to all of whose $S$. Then formally we have
\[
\int_S f \mu = a_0(m_0-m_\infty) + a_1(m_1-m_\infty) + \ldots + a_n(m_n-m_\infty) +\ldots + a m_\infty\ ;
\]
in other words, the infinite sums on the right are defined in $M$. In particular, if $m_\infty=0$, then any sum $\sum a_im_i$ with coefficients $a_i\in \mathbb Z$ is defined in $M$. Note that the sense in which it is defined is a tricky one: It is not directly as any kind of limit of the finite sums; rather, it is characterized by the uniqueness of the map $\mathbb Z[S]^\solid\to M$ with given restriction to $S$. Roughly, this says that there is only one way to consistently define such sums for all measures $\mu$ on $S$ simultaneously; and then one evaluates for any given $\mu$.

Now we want to explain the proof of Theorem~\ref{thm:solid} \emph{when restricted to $\mathbb F_p$-modules} for some prime $p$. (The arguments below could be adapted to $\mathbb Z/n\mathbb Z$ or $\mathbb Z_p$ with minor modifications.) This leads to some important simplifications. Most critically, the underlying condensed set of $\mathbb F_p[S]^\solid$ is profinite.

In other words, we set
\[
\mathbb F_p[S]^\solid := \varprojlim_i \mathbb F_p[S_i] = \mathbb Z[S]^\solid/p
\]
and define the category of solid $\mathbb F_p$-modules $\Solid(\mathbb F_p)\subset \Cond(\mathbb F_p)$ as the full subcategory of all condensed $\mathbb F_p$-modules such that for all profinite sets $S$ with a map $f: S\to M$, there is a unique extension to a map $\mathbb F_p[S]^\solid\to M$ of condensed $\mathbb F_p$-modules. We note that equivalently we are simply considering the $p$-torsion full subcategories inside $\Solid$ resp.~$\Cond(\Ab)$.

In the following, we use two simple facts about condensed sets:
\begin{enumerate}
\item If $M$ is a discrete set, considered as a condensed set, then for any profinite set $S=\varprojlim_i S_i$,
\[
\underline{M}(S) = C(S,M) = \varinjlim_i C(S_i,M) = \varinjlim_i \underline{M}(S_i).
\]
\item If $X_i$, $i\in I$, is any filtered system of condensed sets, then for any profinite set $S$,
\[
(\varinjlim_i X_i)(S) = \varinjlim_i X_i(S).
\]
For the latter, one has to see that the filtered colimit is still a sheaf on $\ast_\proet$, which follows from the commutation of equalizers with filtered colimits.
\end{enumerate}

\begin{proposition} The discrete $\mathbb F_p$-module $\mathbb F_p$ is solid.
\end{proposition}

\begin{proof} We have to prove that
\[
\Hom(\varprojlim_i \mathbb F_p[S_i],\mathbb F_p)=\varinjlim_i C(S_i,\mathbb F_p).
\]
But any map $\varprojlim_i \mathbb F_p[S_i]\to \mathbb F_p$ of condensed abelian groups can be regarded as a map of condensed sets. As the source is profinite and the target is discrete, it follows that the map factors over $\mathbb F_p[S_i]$ for some $i$. A priori this factorization is only as condensed sets, but if we assume that the transition maps are surjective, it is automatically a factorization as condensed abelian groups. This gives the desired result.
\end{proof}

\begin{corollary} For any set $I$, the profinite $\mathbb F_p$-vector space $\prod_I \mathbb F_p$ is solid.
\end{corollary}

\begin{proof} It follows directly from the definition that the class of solid $\mathbb F_p$-modules is stable under all limits.
\end{proof}

Let us say that a condensed $\mathbb F_p$-vector space $V$ is profinite if the underlying condensed set is profinite. One checks easily that any such $V$ can be written as a cofiltered limit of finite-dimensional $\mathbb F_p$-vector spaces, and then that $V\mapsto V^\ast=\Hom(V,\mathbb F_p)$ defines an anti-equivalence between profinite $\mathbb F_p$-vector spaces and discrete $\mathbb F_p$-vector spaces. In particular, any profinite $\mathbb F_p$-vector space is isomorphic to $\prod_I \mathbb F_p$ for some set $I$. From here, it is not hard to verify the following proposition:

\begin{proposition} The class of profinite $\mathbb F_p$-vector spaces forms an abelian subcategory of $\Cond(\mathbb F_p)$ stable under all limits, cokernels, and extensions.
\end{proposition}

\begin{proof} Let $f: V\to W$ be a map of profinite $\mathbb F_p$-vector spaces. This can be written as a cofiltered limit of maps $f_i: V_i\to W_i$ of finite-dimensional $\mathbb F_p$-vector spaces. One can take kernels and cokernels of each $f_i$, and then pass to the limit to get the kernel and cokernel of $f$, which will then again be profinite $\mathbb F_p$-vector spaces. For stability under extensions, we only have to check that the middle term is profinite as a condensed set, which is clear.
\end{proof}

\begin{theorem} A condensed $\mathbb F_p$-module $M$ is solid if and only if it is a filtered colimit of profinite $\mathbb F_p$-vector spaces. These form an abelian category stable under all limits and colimits.
\end{theorem}

One could also prove the rest of Theorem~\ref{thm:solid} in this setting.

\begin{proof} We know that profinite $\mathbb F_p$-vector spaces are solid. As $\Hom(\mathbb F_p[S]^\solid,-)$ (as a functor from $\Cond(\mathbb F_p)$ to sets) commutes with filtered colimits as $\mathbb F_p[S]^\solid$ is profinite, the class of solid $\mathbb F_p$-vector spaces is stable under all filtered colimits, and as observed before it is stable under all limits. Next, we observe that the class of filtered colimits of profinite $\mathbb F_p$-vector spaces forms an abelian subcategory of $\Cond(\mathbb F_p)$. For this, note that any map $f: V\to W$ of such is a filtered colimit of maps $f_i: V_i\to W_i$ of profinite $\mathbb F_p$-vector spaces, and one can form the kernel and cokernel of each $f_i$ and then pass to the filtered colimit again. In particular, they are stable under all colimits (which are generated by cokernels, finite direct sums, and filtered colimits).

It remains to see that any solid $M$ is a filtered colimit of profinite $\mathbb F_p$-vector spaces. Any $M$ admits a surjection $\bigoplus_j \mathbb F_p[S_j]\to M$ for certain profinite sets $S_j\to M$. As $M$ is solid, this gives a surjection $V\to M$ where $V=\bigoplus_j \mathbb F_p[S_j]^\solid$ is a filtered colimit of profinite $\mathbb F_p$-vector spaces. As solid modules are stable under limits, in particular kernels, the kernel of this map is again solid, and so by repeating the argument we find a presentation $W\to V\to M\to 0$ where $W$ and $V$ are filtered colimits of profinite $\mathbb F_p$-vector spaces. As this class is stable under quotients, we see that also $M$ is such a filtered colimit, as desired.
\end{proof}

\newpage

\section{Lecture III: Condensed $\mathbb R$-vector spaces}

The first goal of this course is to define an analogue of solid modules over the real numbers. Roughly, we want to define a notion of liquid $\mathbb R$-vector spaces that is close to the notion of complete locally convex $\mathbb R$-vector spaces, but is a nice abelian category.

To get started, we will translate some of the classical theory of topological vector spaces into the condensed setup. The most familiar kind of topological $\mathbb R$-vector space are the Banach spaces.

\begin{definition} A Banach space is a topological $\mathbb R$-vector space $V$ that admits a norm $||\cdot||$, i.e.~a continuous function
\[
||\cdot||: V\to \mathbb R_{\geq 0}
\]
with the following properties:
\begin{enumerate}
\item For any $v\in V$, the norm $||v||=0$ if and only if $v=0$;
\item For all $v\in V$ and $a\in \mathbb R$, one has $||av|| = |a| ||v||$;
\item For all $v,w\in V$, one has $||v+w||\leq ||v||+||w||$;
\item The sets $\{v\in V\mid ||v||<\epsilon\}$ for varying $\epsilon\in \mathbb R_{>0}$ define a basis of open neighborhoods of $0$;
\item For any sequence $v_0,v_1,\ldots\in V$ with $||v_i-v_j||\to 0$ as $i,j\to\infty$, there exists a (necessarily unique) $v\in V$ with $||v-v_i||\to 0$.
\end{enumerate}
\end{definition}

We note that this notion is very natural from the topological point of view in the sense that it is easy to say what the open subsets of $V$ are -- they are the unions of open balls with respect to the given norm. On the other hand, $V$ is clearly a metrizable topological space (with distance $d(x,y) = ||y-x||$), so in particular first-countable, so in particular compactly generated, so the passage to the condensed vector space $\underline{V}$ does not lose information. Let us try to understand $\underline{V}$; in other words, we must understand how profinite sets map into $V$.

\begin{proposition} Let $V$ be a Banach space, or more generally a complete locally convex topological vector space. Let $S$ be a profinite set, or more generally a compact Hausdorff space, and let $f: S\to V$ be a continuous map. Then $f$ factors over a compact absolutely convex subset $K\subset V$, i.e.~a compact Hausdorff subspace of $V$ such that for all $x,y\in K$ also $ax+by\in K$ whenever $|a|+|b|\leq 1$.
\end{proposition}

In other words, as a condensed set $\underline{V}$ is the union of its compact absolutely convex subsets.

Recall that a topological vector space $V$ is locally convex if it has a basis of neighborhoods $U$ of $0$ such that for all $x,y\in U$ also $ax+by\in U$ whenever $|a|+|b|\leq 1$. It is complete if every Cauchy net converges, i.e.~for any directed index set $I$ and any map $I\to V: i\mapsto x_i$, if $x_i-x_j$ converges to $0$, then there is a unique $x\in V$ so that $x-x_i$ converges to $0$. (Note that $V$ may fail to be metrizable, so one needs to pass to nets.)

To prove the proposition, we need to find $K$. It should definitely contain all convex combinations of images of points in $S$; and all limit points of such. This quickly leads to the idea of integrating a (suitably bounded) measure on $S$ against $f$. The intuition is that if a profinite set $S$ maps into $V$, then one can also integrate any measure on $S$ against this map, to produce a map from the space of measures on $S$ towards $V$.

More precisely, in case $S=\varprojlim_i S_i$ is profinite, consider the space of (``signed Radon'') measures of norm $\leq 1$,
\[
\mathcal M(S)_{\leq 1} := \varprojlim_i \mathcal M(S_i)_{\leq 1},
\]
where $\mathcal M(S_i)_{\leq 1}\subset \mathbb R[S_i]$ is the subspace of $\ell^1$-norm $\leq 1$. More generally, increasing the norm, we define
\[
\mathcal M(S) = \bigcup_{c>0} \mathcal M(S)_{\leq c},\ \mathcal M(S)_{\leq c} = \varprojlim_i \mathcal M(S_i)_{\leq c}.
\]
This can be defined either in the topological or condensed setting, and is known as the space of ``signed Radon'' measures. We note that we do not regard $\mathcal M(S)$ as a Banach space with the norm given by $c$. We also stress the strong similarity with the description of $\mathbb Z[S]$.

\begin{exercise} For any compact Hausdorff space $S$, there is a notion of signed Radon measure on $S$ --- look up the official definition. Show that the corresponding topological vector space $\mathcal M(S)$, with the weak topology, is a union of compact Hausdorff subspaces $\mathcal M(S)_{\leq c}$ (but beware that the weak topology is not the induced colimit topology). Moreover, show that if $S$ is a profinite set, then a signed Radon measure on $S$ is equivalent to a map $\mu$ assigning to each open and closed subset $U\subset S$ of $S$ a real number $\mu(U)\in \mathbb R$, so that $\mu(U\sqcup V)=\mu(U)+\mu(V)$ for two disjoint $U,V$, and there is some constant $C=C(\mu)$ such that for all disjoint decompositions $S=U_1\sqcup\ldots\sqcup U_n$,
\[
\sum_{i=1}^n |\mu(U_i)|\leq C.
\]

Moreover, for a general compact Hausdorff space $S$, choose a surjection $\tilde{S}\to S$ from a profinite set, so that $S=\tilde{S}/R$ for the equivalence relation $R=\tilde{S}\times_S \tilde{S}\subset \tilde{S}\times \tilde{S}$. Show that $\mathcal M(S)$ is the coequalizer of $\mathcal M(R)\rightrightarrows \mathcal M(\tilde{S})$. Thus, the rather complicated notion of signed Radon measures on general compact Hausdorff spaces is simply a consequence of descent from profinite sets.
\end{exercise}

\begin{proposition}\label{prop:completelocallyconvex} Let $V$ be a complete locally convex $\mathbb R$-vector space. Then any continuous map $f: S\to V$ from a profinite set $S$ extends uniquely to a map of topological $\mathbb R$-vector spaces $\mathcal M(S)\to V: \mu\mapsto \int_S f\mu$. The image of $\mathcal M(S)_{\leq 1}$ is a compact absolutely convex subset of $V$ containing $S$.
\end{proposition}

\begin{remark} The exercise implies that the same holds true for any compact Hausdorff $S$. Also note that maps $\mathcal M(S)\to V$ of topological vector spaces are equivalent to maps between the corresponding condensed vector spaces, as the source is compactly generated.
\end{remark}

\begin{proof} The final sentence is clear as $V$ is Hausdorff and so the image of any compact set is compact; and clearly $\mathcal M(S)_{\leq 1}$ is absolutely convex (hence so is its image), and the image contains $S$ (as $S\subset \mathcal M(S)_{\leq 1}$ as Dirac measures).

We have to construct the map $\mathcal M(S)\to V$, so take $\mu\in \mathcal M(S)$, and by rescaling assume that $\mu\in \mathcal M(S)_{\leq 1}$. Write $S=\varprojlim_i S_i$ as a limit of finite sets $S_i$ with surjective transition maps, and pick any lift $t_i: S_i\to S$ of the projection $\pi: S\to S_i$ (we ask for no compatibilities between different $t_i$). We define a net in $V$ as follows: For each $i$, let
\[
v_i = \sum_{s\in S_i} f(t_i(s)) \mu(\pi_i^{-1}(s)) \in V.
\]
We have to see that this is a Cauchy net, so pick some absolutely convex neighborhood $U$ of $0$. Choosing $i$ large enough, we can (by continuity of $f$) ensure that for any two choices $t_i,t_i': S_i\to S$, one has $f(t_i(s))-f(t_i'(s))\in U$ (in other words, $f$ varies at most within a translate of $U$ on the preimages of $S\to S_i$). As $\sum_{s\in S_i} |\mu(\pi_i^{-1}(s))|\leq 1$, we get a convex combination of such differences, which still lies in $U$. This verifies that each $v_i$ does not depend on the choice of the $t_i$ up to translation by an element in $U$, and a slight refinement then proves that one gets a Cauchy net. By completeness of $V$, we get a unique limit $v\in V$ of the $v_i$. The proof essentially also gives continuity: Note that the choice of $i$ depending on $U$ only depended on $f$, not on $\mu$.
\end{proof}

Thus, from the condensed point of view, the following concepts seem fundamental: Compact absolutely convex sets, and the spaces of measures $\mathcal M(S)$ with their compactly generated topology. Note that the latter are ill-behaved from the topological point of view: It is hard to say what a general open subset of $\mathcal M(S)$ looks like; one cannot do better than naively taking the increasing union $\bigcup_{c>0} \mathcal M(S)_{\leq c}$ of compact Hausdorff spaces.

These concepts lead one to the definition of a Smith space:

\begin{definition} A Smith space is a complete\footnote{One can show that completeness follows from the other assumptionss, basically as compact Hausdorff spaces are always complete.} locally convex topological $\mathbb R$-vector space $V$ that admits a compact absolutely convex subset $K\subset V$ such that $V=\bigcup_{c>0} cK$ with the induced compactly generated topology on $V$.
\end{definition}

One can verify that $\mathcal M(S)$ is a Smith space: One needs to check that it is complete and locally convex. We will redefine Smith spaces in the next lecture, and check that it defines a Smith space in the latter sense.

\begin{corollary} Let $V$ be a complete locally convex topological $\mathbb R$-vector space, and consider the category of Smith spaces $W\subset V$. Then this category is filtered, and as condensed $\mathbb R$-vector spaces
\[
\underline{V} = \varinjlim_{W\subset V} \underline{W}.
\]
\end{corollary}

\begin{proof} To see that the category is filtered, let $W_1,W_2\subset V$ be two Smith subspaces, with compact generating subsets $K_1,K_2$. Then $K_1\sqcup K_2\to V$ factors over a compact absolutely convex subset $K\subset V$ by the proposition above, and $W=\bigcup_{c>0} cK\subset V$ (with the inductive limit topology) is a Smith space containing both $W_1$ and $W_2$. By the same token, if $f: S\to V$ is any map from a profinite set, then $f$ factors over a compact absolutely convex subset $K\subset V$, and hence over the Smith space $W=\bigcup_{c>0} cK\subset V$, proving that
\[
\underline{V} = \varinjlim_{W\subset V} \underline{W}.\qedhere
\]
\end{proof}

This realizes the idea that from the condensed point of view, Smith spaces are the basic building blocks. On the other hand, it turns out that Smith spaces are closely related to Banach spaces:

\begin{theorem}[Smith, \cite{Smith}] The categories of Smith spaces and Banach spaces are anti-equivalent. More precisely, if $V$ is a Banach space, then $\Hom(V,\mathbb R)$ is a Smith space; and if $W$ is a Smith space, then $\Hom(W,\mathbb R)$ is a Banach space, where in both cases we endow the dual space with the compact-open topology. The corresponding biduality maps are isomorphisms.
\end{theorem}

This will be proved in the next lecture.

\begin{remark} This gives one sense in which Banach spaces are always reflexive, i.e.~isomorphic to their bidual. Note, however, that this does not coincide with the usual notion of reflexivity: It is customary to make the dual of a Banach space itself into a Banach space (by using the norm $||f||=\sup_{v\in V, ||v||\leq 1} |f(v)|$), and then reflexivity asks whether a Banach space is isomorphic to its corresponding bidual. Even if that happens, the two notions of dual are still different: If a topological vector space is both Banach and Smith, it is finite-dimensional.
\end{remark}

\newpage

\section{Lecture IV: $\mathcal M$-complete condensed $\mathbb R$-vector spaces}

The discussion of the last lecture motivates the following definition.

\begin{definition} Let $V$ be a condensed $\mathbb R$-vector space. Then $V$ is $\mathcal M$-complete if (the underlying condensed set of) $V$ is quasiseparated, and for all maps $f: S\to V$ from a profinite set, there is an extension to a map
\[
\tilde{f}: \mathcal M(S)\to V
\]
of condensed $\mathbb R$-vector spaces.
\end{definition}

\begin{definition}[slight return] A Smith space is an $\mathcal M$-complete condensed $\mathbb R$-vector space $V$ such that there exists some compact Hausdorff $K\subset V$ with $V=\bigcup_{c>0} cK$.
\end{definition}

\begin{exercise} Prove that this definition of Smith spaces is equivalent to the previous definition.
\end{exercise}

We note that it is clear that Smith spaces defined in this way form a full subcategory of topological $\mathbb R$-vector spaces (using Proposition~\ref{prop:condtop} and Lemma~\ref{lem:compactgenproduct}). Moreover, Proposition~\ref{prop:completelocallyconvex} implies that any Smith space in the sense of the last lecture is a Smith space in the current sense. It remains to see that for any Smith space in the current sense, the corresponding topological vector space is complete and locally convex. Note that it follows easily that in the current definition one can take $K$ to be absolutely convex.\\

Let us check first that for any profinite set $T$, the condensed $\mathbb R$-vector space $\mathcal M(T)$ is a Smith space. We only need to see that it is $\mathcal M$-complete. Any map $f: S\to \mathcal M(T)$ factors over $\mathcal M(T)_{\leq c}$ for some $c$. Then writing $T=\varprojlim_i T_i$ as an inverse limit of finite sets $T_i$, $f$ is an inverse limit of maps $f_i: S\to \mathcal M(T_i)_{\leq c}$. This reduces us to the case that $T$ is finite (provided we can bound the norm of the extension); but then we can apply Proposition~\ref{prop:completelocallyconvex} (and observing that the extension maps $\mathcal M(S)_{\leq 1}$ into $\mathcal M(T_i)_{\leq c}$).

Our initial hope was that $\mathcal M$-complete condensed $\mathbb R$-vector spaces, without the quasiseparation condition, would behave as well as solid $\mathbb Z$-modules. We will see in the next lecture that this is not so. In this lecture, we however want to show that with the quasiseparation condition, the category behaves as well as it can be hoped for: It is not an abelian category (because cokernels are problematic under quasiseparatedness), but otherwise nice.

First, we check that the extension $\tilde{f}$ is necessarily unique.

\begin{proposition} Let $V$ be a quasiseparated condensed $\mathbb R$-vector space and let $g: \mathcal M(S)\to V$ be a map of condensed $\mathbb R$-vector space for some profinite set $S$. If the restriction of $g$ to $S$ vanishes, then $g=0$.
\end{proposition}

\begin{proof} The preimage $g^{-1}(0)\subset \mathcal M(S)$ is a quasicompact injection of condensed sets. This means that it is a closed subset in the topological sense, see the appendix to this lecture.

On the other hand, $g^{-1}(0)$ contains the $\mathbb R$-vector space spanned by $S$. This contains a dense subset of $\mathcal M(S)_{\leq c}$ for all $c$, and thus its closure by the preceding.
\end{proof}

\begin{exercise} Show that in the definition of $\mathcal M$-completeness, one can restrict to extremally disconnected $S$.
\end{exercise}

\begin{proposition} Any $\mathcal M$-complete condensed $\mathbb R$-vector space $V$ is the filtered colimit of the Smith spaces $W\subset V$; conversely, any filtered colimit of Smith spaces along injections is quasiseparated and $\mathcal M$-complete. For any map $f: V\to V'$ between $\mathcal M$-complete condensed $\mathbb R$-vector spaces, the kernel and image of $f$ (taken in condensed $\mathbb R$-vector spaces) are $\mathcal M$-complete condensed $\mathbb R$-vector spaces. Moreover, $f$ admits a cokernel in the category of $\mathcal M$-complete condensed $\mathbb R$-vector spaces.
\end{proposition}

\begin{proof} Let $V$ be an $\mathcal M$-complete condensed $\mathbb R$-vector space and let $f: S\to V$ be a map from a profinite set. Then $f$ extends uniquely to $\tilde{f}: \mathcal M(S)\to V$. The subset $\mathcal M(S)_{\leq 1}\subset\mathcal M(S)$ is compact Hausdorff (i.e., a quasicompact quasiseparated condensed set); thus, its image in the quasiseparated $V$ is still quasicompact and quasiseparated, i.e.~a compact Hausdorff $K\subset V$. Let $W=\bigcup_{c>0} cK\subset V$. This is a condensed $\mathbb R$-vector space, and it is itself $\mathcal M$-complete. Indeed, assume that $T$ is some profinite set with a map $T\to W$; by rescaling, we can assume that $T$ maps into $K$. We get a unique extension to a map $\mathcal M(T)\to V$, and we need to see that it factors over $W$. We may assume that $T$ is extremally disconnected (as any surjection $T'\to T$ from an extremally disconnected $T'$ induces a surjection $\mathcal M(T')\to \mathcal M(T)$). Then the map $T\to K$ lifts to a map $T\to \mathcal M(S)_{\leq 1}$. Thus, we get a map $\mathcal M(T)\to \mathcal M(S)$ whose composite with $\mathcal M(S)\to V$ must agree with the original map $\mathcal M(T)\to V$ by uniqueness. This shows that the image of $\mathcal M(T)\to V$ is contained in $W$, as desired. It follows that $W$ is a Smith space.

It is now easy to see that the set of such $W$ is filtered, and then $V$ is the filtered colimit, as we have seen that any map from a profinite set $S$ factors over one such $W$. It is clear that any filtered colimit $\varinjlim_i W_i$ of Smith spaces $W_i$ along injections is quasiseparated (as quasiseparatedness is preserved under filtered colimits of injections). It is also $\mathcal M$-complete, as any map $f: S\to \varinjlim_i W_i$ factors over one $W_i$, and then extends to a map $\tilde{f}: \mathcal M(S)\to W_i\to \varinjlim_i W_i$.

For the assertions about a map $f$, the claim about the kernel is clear. For the image, we may assume that $V$ is a Smith space, by writing it as a filtered colimit of such; and then that $V=\mathcal M(S)$ by picking a surjection $S\to K\subset V$ where $K$ is a generating compact Hausdorff subset of $V$. But then the image of $f$ is exactly the Smith space $W\subset V$ constructed in the first paragraph.

For the cokernel of $f$, we may first take the cokernel $Q$ in the category of condensed $\mathbb R$-vector spaces. Now we may pass to the maximal quasiseparated quotient $Q^{\mathrm{qs}}$ of $Q$, which is still a condensed $\mathbb R$-vector space, by the appendix to this lecture. It is then formal that this satisfies the conditon of being $\mathcal M$-complete when tested against extremally disconnected $S$, as any map $S\to Q^{\mathrm{qs}}$ lifts to $S\to W$, and so extends to $\mathcal M(S)\to W\to Q^{\mathrm{qs}}$. The automatic uniqueness of this extension then implies the result for general profinite $S$ (by covering $S$ by an extremally disconnected $\tilde{S}\to S$).
\end{proof}

We will now prove the anti-equivalence between Smith spaces in the current sense, with Banach spaces. (In particular, using the original form of Smith's theorem, this also solves the exercise above.)

\begin{theorem} For any Smith space $W$, the internal dual $\intHom_{\mathbb R}(W,\mathbb R)$ is isomorphic to $\underline{V}$ for a Banach space $V$. Conversely, for any Banach space $V$, the internal dual $\intHom_{\mathbb R}(\underline{V},\mathbb R)$ is a Smith space. The corresponding biduality maps are isomorphisms.
\end{theorem}

\begin{proof} Assume first that $W=\mathcal M(S,\mathbb R)$ for some profinite set $S$. Then $\Hom_{\mathbb R}(\mathcal M(S,\mathbb R),\mathbb R)=C(S,\mathbb R)$ as we have seen that any continuous map $S\to \mathbb R$ extends uniquely to a map of condensed $\mathbb R$-vector spaces $\mathcal M(S,\mathbb R)\to \mathbb R$. More generally, for any profinite set $T$,
\[
\intHom_{\mathbb R}(\mathcal M(S,\mathbb R),\mathbb R)(T) = \Hom_{\mathbb R}(\mathcal M(S,\mathbb R),\intHom_{\mathbb R}(\mathbb R[T],\mathbb R)),
\]
and $\intHom_{\mathbb R}(\mathbb R[T],\mathbb R)=\underline{C(T,\mathbb R)}$, for the Banach space $C(T,\mathbb R)$ with the sup-norm. Indeed, this Banach space has the property that continuous maps $T'\to C(T,\mathbb R)$ are equivalent to continuous maps $T'\times T\to \mathbb R$. Thus, using that Banach spaces are complete locally convex, we see
\[
\Hom_{\mathbb R}(\mathcal M(S,\mathbb R),\intHom_{\mathbb R}(\mathbb R[T],\mathbb R)) = \Hom_{\mathbb R}(\mathcal M(S,\mathbb R),\underline{C(T,\mathbb R)}) = C(T,\mathbb R)(S) = C(T\times S,\mathbb R),
\]
showing that indeed $\intHom_{\mathbb R}(\mathcal M(S,\mathbb R),\mathbb R)=\underline{C(S,\mathbb R)}$.

Now assume that $V=C(S,\mathbb R)$. We need to see that for any profinite set $T$,
\[
\Hom_{\mathbb R}(\underline{C(S,\mathbb R)},\underline{C(T,\mathbb R)}) = \mathcal M(S)(T).
\]
But the left-hand side can be computed in topological vector spaces, and any map of Banach spaces has bounded norm. It suffices to see that maps $C(S,\mathbb R)\to C(T,\mathbb R)$ of Banach spaces of norm $\leq 1$ are in bijection with maps $T\to \mathcal M(S)_{\leq 1}$. But any such map of Banach spaces is uniquely determined by its restriction to $\varinjlim_i C(S_i,\mathbb R)$, where $S=\varprojlim_i S_i$ is written as a limit of finite sets, where again we assume that the operator norm is bounded by $1$. Similarly, maps $T\to \mathcal M(S)_{\leq 1}$ are limits of maps $T\to \mathcal M(S_i)_{\leq 1}$. We can thus reduce to the case $S$ is finite, where the claim is clear. We see also that the biduality maps are isomorphisms in this case.

Now let $W$ be a general Smith space, with a generating subset $K\subset W$. Pick a profinite set $S$ with a surjection $S\to K$. We get a surjection $\mathcal M(S)\to W$. Its kernel $W'\subset \mathcal M(S)$ is automatically quasiseparated, and in fact a Smith space again: It is generated by $W'\cap \mathcal M(S)_{\leq 1}$, which is compact Hausdorff. Picking a further surjection $\mathcal M(S')\to W'$, we get a resolution
\[
\mathcal M(S')\to \mathcal M(S)\to W\to 0.
\]
Taking $\intHom_{\mathbb R}(-,\mathbb R)$, we get
\[
0\to \intHom_{\mathbb R}(W,\mathbb R)\to \underline{C(S,\mathbb R)}\to \underline{C(S',\mathbb R)} .
\]
The latter map corresponds to a map of Banach space $C(S,\mathbb R)\to C(S',\mathbb R)$, whose kernel is a closed subspace that is itself a Banach space $V$, and we find $\intHom_{\mathbb R}(W,\mathbb R)=\underline{V}$.

Conversely, if $V$ is any Banach space, equipped with a norm, the unit ball $B=\{f\in \Hom(V,\mathbb R)\mid ||f||\leq 1\}$ is compact Hausdorff when equipped with the weak topology.\footnote{This is known as the Banach-Alaoglu theorem, and follows from Tychonoff by using the closed embedding $B\hookrightarrow \prod_{v\in V} [-1,1]$.} Picking a surjection $S\to B$ from a profinite set $S$, we get a closed immersion $V\to C(S,\mathbb R)$ (to check that it is a closed immersion, use that it is an isometric embedding by the Hahn-Banach theorem), whose quotient will then be another Banach space; thus, any Banach space admits a resolution
\[
0\to V\to C(S,\mathbb R)\to C(S',\mathbb R)
\]
We claim that taking $\intHom_{\mathbb R}(\underline{-},\mathbb R)$ gives an exact sequence
\[
\mathcal M(S')\to \mathcal M(S)\to \intHom_{\mathbb R}(\underline{V},\mathbb R)\to 0.
\]
For this, using that we already know about spaces of measures, we have to see that for any extremally disconnected $T$, the sequence
\[
\Hom_{\mathbb R}(C(S',\mathbb R),C(T,\mathbb R))\to \Hom_{\mathbb R}(C(S,\mathbb R),C(T,\mathbb R))\to \Hom_{\mathbb R}(V,C(T,\mathbb R))\to 0
\]
is exact. This follows from the fact that the Banach space $C(T,\mathbb R)$ is injective in the category of Banach spaces, cf.~\cite[Section 1.3]{SepInjBanach}: If $V\subset V'$ is a closed embedding of Banach spaces, then any map $V\to C(T,\mathbb R)$ of Banach spaces extends to a map $V'\to C(T,\mathbb R)$.\footnote{The case $T=\ast$ is the Hahn-Banach theorem. Part of the Hahn-Banach theorem is that the extension can be bounded by the norm of the original map. This implies the injectivity of $l^\infty(S_0,\mathbb R)=C(\beta S_0,\mathbb R)$ for any set $S_0$, and then the result for general $T$ follows by writing $T$ as a retract of $\beta S_0$ for some $S_0$.}

Using these resolutions, it also follows that the biduality maps are isomorphisms in general.
\end{proof}

We remark that if $S$ is extremally disconnected, then $\mathcal M(S)$ is a projective object in the category of Smith spaces (in fact, in the category of $\mathcal M$-complete condensed $\mathbb R$-vector spaces), as $S$ is a projective object in the category of condensed sets. Dually, this means that $C(S,\mathbb R)$ is an injective object in the category of Banach spaces. For solid abelian groups, it turned out that $\mathbb Z[S]^\solid$ is projective in the category of solid abelian groups, for all profinite sets $S$. Could a similar thing happen here? It turns out, that no:

\begin{proposition}\label{prop:notprojective} For ``most'' profinite sets $S$, for example $S=\mathbb N\cup \{\infty\}$, or a product $S=S_1\times S_2$ of two infinite profinite sets $S_1, S_2$, the Banach space $C(S,\mathbb R)$ is not injective; equivalently, the Smith space $\mathcal M(S)$ is not projective.
\end{proposition}

\begin{proof} Note that $C(\mathbb N\cup\{\infty\},\mathbb R)$ is the product of the Banach space $c_0$ of null-sequences, with $\mathbb R$. One can show that $C(S_1\times S_2,\mathbb R)$ contains $c_0$ as a direct factor, cf.~\cite{Cembranos}, so it is enough to show that $c_0$ is not injective, which is \cite[Theorem 1.25]{SepInjBanach}.
\end{proof}

In fact, cf.~\cite[Section 1.6.1]{SepInjBanach}, there is no known example of an injective Banach space that is not isomorphic to $C(S,\mathbb R)$ where $S$ is extremally disconnected!

Next, let us discuss tensor products. Before making the connection with the existing notions for Banach spaces, let us follow our nose and define a tensor product on the category of $\mathcal M$-complete condensed $\mathbb R$-vector spaces.

\begin{proposition} Let $V$ and $W$ be $\mathcal M$-complete condensed $\mathbb R$-vector spaces. Then there is an $\mathcal M$-complete condensed $\mathbb R$-vector space $V\otimes_\pi W$ equipped with a bilinear map
\[
V\times W\to V\otimes_\pi W
\]
of condensed $\mathbb R$-vector spaces, which is universal for bilinear maps; i.e.~any bilinear map $V\times W\to L$ to a $\mathcal M$-complete condensed $\mathbb R$-vector space extends uniquely to a map $V\otimes_\pi W\to L$ of condensed $\mathbb R$-vector spaces.

The functor $(V,W)\mapsto V\otimes_\pi W$ from pairs of $\mathcal M$-complete condensed $\mathbb R$-vector spaces to $\mathcal M$-complete condensed $\mathbb R$-vector spaces commutes with colimits in each variable, and satisfies
\[
\mathcal M(S)\otimes_\pi \mathcal M(S')\cong \mathcal M(S\times S')
\]
for any profinite sets $S,S'$.
\end{proposition}

Note that the case of $S_1\times S_2$ in Proposition~\ref{prop:notprojective} means that a tensor product $\mathcal M(S_1\times S_2)$ of two projectives $\mathcal M(S_1)$, $\mathcal M(S_2)$ is not projective anymore. This is a subtlety that will persist, and that we have to live with.

\begin{proof} It is enough to show that $\mathcal M(S\otimes S')$ represents bilinear maps $\mathcal M(S)\times \mathcal M(S')\to L$; indeed, the other assertions will then follow via extending everything by colimits to the general case.

In other words, we have to show that bilinear maps $\mathcal M(S)\times \mathcal M(S')\to L$ are in bijection with maps $S\times S'\to L$. Any map $S\times S'\to L$ extends to $\mathcal M(S\times S')\to L$, and thus gives a bilinear map $\mathcal M(S)\times \mathcal M(S')\to L$, as one can directly define a natural bilinear map $\mathcal M(S)\times \mathcal M(S')\to \mathcal M(S\times S')$. Thus, it is enough to show that any bilinear map $\mathcal M(S)\times \mathcal M(S')\to L$ that vanishes on $S\times S'$ is actually zero. For this, we again observe that it vanishes on a dense subset of $\mathcal M(S)_{\leq c}\times \mathcal M(S')_{\leq c}$ for any $c$, and thus on the whole of it, as the kernel is a closed subspace when $L$ is quasiseparated.
\end{proof}

Tensor products of Banach spaces were defined by Grothendieck, \cite{GrothendieckTensor}, and interestingly he has defined several tensor products. There are two main examples. One is the projective tensor product $V_1\otimes_\pi V_2$; this is a Banach space that represents bilinear maps $V_1\times V_2\to W$.

\begin{proposition} Let $V_1$ and $V_2$ be two Banach spaces. Then $\underline{V_1}\otimes_\pi \underline{V_2}\cong \underline{V_1\otimes_\pi V_2}$.
\end{proposition}

\begin{remark} The proposition implies that any compact absolutely convex subset $K\subset V_1\otimes_\pi V_2$ is contained in the closed convex hull of $K_1\times K_2$ for compact absolutely convex subsets $K_i\subset V_i$.
\end{remark}

\begin{proof} Any Banach space admits a projective resolution by spaces of the form $\ell^1(I)$ for some set $I$; this reduces us formally to that case. One can even reduce to the case that $I$ is countable, by writing $\ell^1(I)$ as the $\omega_1$-filtered colimit of $\ell^1(J)$ over all countable $J\subset I$. Thus, we may assume that $V_1=V_2=\ell^1(\mathbb N)$ is the space of $\ell^1$-sequences. In that case $V_1\otimes_\pi V_2 = \ell^1(\mathbb N\times \mathbb N)$.

In that case,
\[
\underline{\ell^1(\mathbb N)} = \varinjlim_{(\lambda_n)_n} \mathcal M(\mathbb N\cup \{\infty\})/(\mathbb R\cdot [\infty])
\]
where the filtered colimit is over all null-sequences of positive real numbers $0<\lambda_n\leq 1$. Now the result comes down to the observation that any null-sequence $\lambda_{n,m}$ of positive real numbers in $(0,1]$ parametrized by pairs $n,m\in \mathbb N$ can be bounded from above by a sequence of the form $\lambda_n \lambda_m'$ where $\lambda_n$ and $\lambda_m'$ are null-sequences of real numbers in $(0,1]$. Indeed, one can take
\[
\lambda_n=\lambda_n' = \sqrt{\max_{\mathrm{max}(n',m')\geq n} \lambda_{n',m'}}.\qedhere
\]
\end{proof}

On the other hand, there is the injective tensor product $V_1\otimes_\epsilon V_2$; this satisfies $C(S_1,\mathbb R)\otimes_\epsilon C(S_2,\mathbb R)\cong C(S_1\times S_2,\mathbb R)$. There is a map $V_1\otimes_\pi V_2\to V_1\otimes_\epsilon V_2$, that is however far from an isomorphism.

\begin{proposition} Let $V_i$, $i=1,2$, be Banach spaces with dual Smith spaces $W_i$. Let $W=W_1\otimes_\pi W_2$, which is itself a Smith space, and let $V$ be the Banach space dual to $W$. Then there is a natural map
\[
V_1\otimes_\epsilon V_2\to V
\]
that is a closed immersion of Banach spaces. It is an isomorphism if $V_1$ or $V_2$ satisfies the approximation property.
\end{proposition}

We recall that most natural Banach spaces have the approximation property; in fact, it had been a long-standing open problem whether all Banach spaces have the approximation property, until a counterexample was found by Enflo, \cite{Enflo} (for which he was awarded a live goose by Mazur).

\begin{proof} The injective tensor product has the property that if $V_1\hookrightarrow V_1'$ is a closed immersion of Banach spaces, then $V_1\otimes_\epsilon V_2\hookrightarrow V_1'\otimes_\epsilon V_2$ is a closed immersion. However, it is not in general true that if
\[
0\to V_1\to V_1'\to V_1''
\]
is a resolution, then
\[
0\to V_1\otimes_\epsilon V_2\to V_1'\otimes_\epsilon V_2\to V_1''\otimes_\epsilon V_2
\]
is a resolution: Exactness in the middle may fail. The statement is true when $V_2=C(S_2,\mathbb R)$ is the space of continuous functions on some profinite set $S_2$, as in that case $W\otimes_\epsilon V_2 = C(S_2,W)$ for any Banach space $W$, and $C(S_2,-)$ preserves exact sequences of Banach spaces for profinite sets $S_2$.

In any case, such resolutions prove the desired statement: Note that if $V_i=C(S_i,\mathbb R)$, one has $W_i=\mathcal M(S_i)$, and then $W=\mathcal M(S_1\times S_2)$ and so $V=C(S_1\times S_2,\mathbb R)$, which is indeed $V_1\otimes_\epsilon V_2$; and these identifications are functorial. Now the resolution above gives the same statement as long as $V_2=C(S_2,\mathbb R)$; and in general one sees that $V$ is the kernel of the map $V_1'\otimes_\epsilon V_2\to V_1''\otimes_\epsilon V_2$.

It remains to see that the map $V_1\otimes_\epsilon V_2\to V$ is an isomorphism if one of $V_1$ and $V_2$ has the approximation property. Assume that $V_2$ has the approximation property. Note that
\[
\underline{V}=\intHom_{\mathbb R}(W_1\otimes_{\mathbb R} W_2,\mathbb R) = \intHom_{\mathbb R}(W_1,\intHom_{\mathbb R}(W_2,\mathbb R)) = \intHom_{\mathbb R}(W_1,\underline{V_2}).
\]
In the classical literature, this is known as the weak-$\ast$-to-weak compact operators from $V_1^\ast$ to $V_2$. On the other hand, $V_1\otimes_\epsilon V_2\subset V$ is the closed subspace generated by algebraic tensors $V_1\otimes V_2$; equivalently, this is the closure of the space of maps $W_1\to \underline{V_2}$ of finite rank. The condition that $V_2$ has the approximation property precisely ensures that this is all of $\intHom_{\mathbb R}(W_1,\underline{V_2})$, cf.~\cite[Theorem 1.3.11]{GrothResumeRevisited}.
\end{proof}

In other words, when Banach spaces are covariantly embedded into $\mathcal M$-complete condensed $\mathbb R$-spaces, one gets the projective tensor product; while under the duality with Smith spaces, one (essentially) gets the injective tensor product. Thus, in the condensed setting, there is only one tensor product, but it recovers both tensor products on Banach spaces.

\newpage

\section*{Appendix to Lecture IV: Quasiseparated condensed sets}

In this appendix, we make some remarks about the category of quasiseparated condensed sets. We start with the following critical observation, giving an interpretation of the topological space $X(\ast)_{\mathrm{top}}$ purely in condensed terms.

\begin{proposition} Let $X$ be a quasiseparated condensed set. Then quasicompact injections $i:Z\hookrightarrow X$ are equivalent to closed subspaces $W\subset X(\ast)_{\mathrm{top}}$ via $Z\mapsto Z(\ast)_{\mathrm{top}}$, resp.~sending a closed subspace $W\subset X(\ast)_{\mathrm{top}}$ to the subspace $Z\subset X$ with
\[
Z(S) = X(S)\times_{\mathrm{Map}(S,X(\ast))} \mathrm{Map}(S,W).
\]
\end{proposition}

In the following, we will sometimes refer to quasicompact injections $i: Z\hookrightarrow X$ as closed subspaces of $X$, noting that by the proposition this agrees with the topological notion.

\begin{proof} The statement reduces formally to the case that $X$ is quasicompact by writing $X$ as the filtered union of its quasicompact subspaces. In that case, $X$ is equivalent to a compact Hausdorff space. If $i: Z\hookrightarrow X$ is a quasicompact injection, then $Z$ is again quasicompact and quasiseparated, so again (the condensed set associated to) a compact Hausdorff space. But injections of compact Hausdorff spaces are closed immersions. In other words, the statement reduces to the assertion that under the equivalence of quasicompact quasiseparated condensed sets with compact Hausdorff spaces, injections correspond to closed subspaces, which is clear.
\end{proof}

\begin{lemma} The inclusion of the category of quasiseparated condensed sets into all condensed sets admits a left adjoint $X\mapsto X^{\mathrm{qs}}$, with the unit $X\to X^{\mathrm{qs}}$ being a surjection of condensed sets. The functor $X\mapsto X^{\mathrm{qs}}$ preserves finite products. In particular, it defines a similar left adjoint for the inclusion of quasiseparated condensed $A$-modules into all condensed $A$-modules, for any quasiseparated condensed ring $A$.
\end{lemma}

\begin{proof} Choose a surjection $X'=\bigsqcup_i S_i\to X$ from a disjoint union of profinite sets $S_i$, giving an equivalence relation $R=X'\times_X X'\subset X'\times X'$. For any map $X\to Y$ with $Y$ quasiseparated, the induced map $X'\to Y$ has the property that $X'\times_Y X'\subset X'\times X'$ is a quasicompact injection (i.e., is a closed subspace) and is an equivalence relation that contains $R$; thus, it contains the minimal closed equivalence relation $\overline{R}\subset X'\times X'$ generated by $R$ (which exists, as any intersection of closed equivalence relations is again a closed equivalence relation). This shows that $X^{\mathrm{qs}} = X'/\overline{R}$ defines the desired adjoint.

To check that it preserves finite products, we need to check that if $R\subset X\times X$ and $R'\subset X'\times X'$ are two equivalence relations on quasiseparated condensed sets $X$, $X'$, then the minimal closed equivalence relation $\overline{R\times R'}$ on $X\times X'$ containing $R\times R'$ is given by $\overline{R}\times \overline{R'}$. To see this, note first that for fixed $x'\in X'$, it must contain $\overline{R}\times (x',x')\subset X\times X\times X'\times X'$. Similarly, for fixed $x\in X$, it must contain $(x,x)\times \overline{R'}$. But now if $(x_1,x_1')$ and $(x_2,x_2')$ are two elements of $X\times X'$ such that $x_1$ is $\overline{R}$-equivalent to $x_2$ and $x_1'$ is $\overline{R'}$-equivalent to $x_2'$, then $(x_1,x_1')$ is $\overline{R\times R'}$-equivalent to $(x_2,x_1')$, which is $\overline{R\times R'}$-equivalent to $(x_2,x_2')$. Thus, $\overline{R}\times \overline{R'}\subset \overline{R\times R'}$, and the reverse inclusion is clear.
\end{proof}

\begin{corollary} The inclusion of the category of $\mathcal M$-complete condensed $\mathbb R$-vector spaces into all condensed $\mathbb R$-vector spaces admits a left adjoint, the ``$\mathcal M$-completion''.
\end{corollary}

\begin{proof} Let $V$ be any condensed $\mathbb R$-vector space, and pick a resolution
\[
\bigoplus_j \mathbb R[S_j']\to \bigoplus_i \mathbb R[S_i]\to V\to 0.
\]
The left adjoint exists for $\mathbb R[S]$, with $S$ profinite, and takes the value $\mathcal M(S)$, essentially by definition of $\mathcal M$-completeness. It follows that the left adjoint for $V$ is given by the quasiseparation of the cokernel of
\[
\bigoplus_j \mathcal M(S_j')\to \bigoplus_i \mathcal M(S_i).\qedhere
\]
\end{proof}

\newpage

\section{Lecture V: Entropy and a real $B_{\mathrm{dR}}^+$}
\begin{center}
{\it in memoriam Jean-Marc Fontaine}
\end{center}

The discussion of the last lecture along with the definition of solid abelian groups raises the following question:

\begin{question} Does the category of condensed $\mathbb R$-vector spaces $V$ such that for all profinite sets $S$ with a map $f: S\to V$, there is a unique extension to a map $\tilde{f}: \mathcal M(S)\to V$ of condensed $\mathbb R$-vector spaces, form an abelian category stable under all kernels, cokernels, and extensions?
\end{question}

We note that stability under kernels is easy to see. We will see that stability under cokernels and stability under extensions fails, so the answer is no.

Note that the condition imposed is some form of local convexity. It is a known result in Banach space theory that there are extensions of Banach spaces that are not themselves Banach spaces, and in fact not locally convex. Let us recall the construction, due to Ribe, \cite{Ribe}.

Let $V=\ell^1(\mathbb N)$ be the Banach space of $\ell^1$-sequences of real numbers $x_0,x_1,\ldots$. We will construct a non-split extension
\[
0\to \mathbb R\to V'\to V\to 0.
\]
The construction is based on the following two lemmas.

\begin{lemma} Let $V$ be a Banach space, let $V_0\subset V$ be a dense subvectorspace and let $\phi: V_0\to \mathbb R$ be a function that is almost linear in the sense that for some constant $C$, we have for all $v,w\in V_0$
\[
|\phi(v+w)-\phi(v)-\phi(w)|\leq C(||v||+||w||);
\]
moreover, $\phi(av)=a\phi(v)$ for all $a\in \mathbb R$ and $v\in V_0$. Then one can turn the abstract $\mathbb R$-vector space $V_0'=V_0\times \mathbb R$ into a topological vector space by declaring a system of open neighborhoods of $0$ to be
\[
\{(v,r)\mid ||v||+|r-\phi(v)|<\epsilon\}.
\]
The completion $V'$ of $V_0'$ defines an extension
\[
0\to \mathbb R\to V'\to V\to 0
\]
of topological vector spaces (which stays exact as a sequence of condensed vector spaces).

This extension of topological vector spaces is split if and only if there is a linear function $f: V_0\to \mathbb R$ such that $|f(v)-\phi(v)|\leq C'||v||$ for all $v\in V_0$ and some constant $C'$.
\end{lemma}

One can in fact show that all extensions arise in this way.

\begin{proof} All statements are immediately verified. Note that the extension $V'\to V$ splits as topological spaces: Before completion we have the nonlinear splitting $V_0\to V_0': v\mapsto (v,\phi(v))$, and this extends to completions.
\end{proof}

We see that non-split extensions are related to functions that are almost linear locally, but not almost linear globally. An example is given by entropy. (We note that a relation between entropy and real analogues of $p$-adic Hodge-theoretic rings has been first proposed by Connes and Consani, cf.~e.g.~\cite{ConnesConsani1}, \cite{ConnesWitt}. There is some relation between the constructions of this and the next lecture, and their work.) Recall that if $p_1,\ldots,p_n$ are real numbers in $[0,1]$ with sum $1$ (considered as a probability distribution on the finite set $\{1,\ldots,n\}$), its entropy is
\[
H=-\sum_{i=1}^n p_i \log p_i.
\]
The required local almost linearity comes from the following lemma.

\begin{lemma}\label{lem:entropyalmostlinear} For all real numbers $s$ and $t$, one has
\[
|s\log|s|+t\log|t|-(s+t)\log|s+t||\leq 2\log 2(|s|+|t|).
\]
\end{lemma}

We note that $0\log 0 := 0$ extends the function $s\mapsto s\log|s|$ continuously to $0$.

\begin{proof} Rescaling both $s$ and $t$ by a positive scalar $\lambda$ multiplies the left-hand side by $\lambda$ (Check!). We may thus assume that $s,t\in [-1,1]$ and at least one of them has absolute value equal to $1$. Changing sign and permuting, we assume that $t=1$. It then suffices to see that the left-hand side is bounded by $2\log 2$ for all $s\in [-1,1]$. Note that some bound is now clear, as the left-hand side is continuous; to get $2\log 2$, note that $s\log|s|$ and $(s+1)\log|s+1|$ take opposite signs for $s\in[-1,1]$, so it suffices to bound both individually by $2\log 2$, which is easy.
\end{proof}

\begin{corollary} Let $V_0\subset V=\ell^1(\mathbb N)$ be the subspace spanned by sequences with finitely many nonzero terms. The function
\[
H: (x_0,x_1,\ldots)\in V_0\mapsto s\log|s| - \sum_{i\geq 0} x_i\log|x_i| ,\ \mathrm{where}\ s=\sum_{i\geq 0} x_i,
\]
is locally almost linear but not globally almost linear, and so defines a nonsplit extension
\[
0\to \mathbb R\to V'\to V\to 0.
\]
\end{corollary}

\begin{proof} Local almost linearity follows from the lemma (the scaling invariance $H(a x) = a H(x)$ uses the addition of $s\log|s|$). For global non-almost linearity, assume that $H$ was close to $\sum \lambda_i x_i$ for certain $\lambda_i\in \mathbb R$. Looking at the points $(0,\ldots,0,1,0,\ldots)$, one sees that the $\lambda_i$ are bounded (as $H=0$ on such points). On the other hand, looking at $(\frac 1n,\ldots,\frac 1n,0,\ldots)$ with $n$ occurences of $\frac 1n$, global almost linearity requires
\[
|H(n)-\frac 1n \sum_{i=0}^{n-1} \lambda_i|\leq C
\]
for some constant $C$. This would require $H(n)$ to be bounded (as the $\lambda_i$ are), but one computes $H(n)=\log n$.
\end{proof}

Now we translate this extension into the condensed picture. Ideally, we would like to show that there is an extension
\[
0\to \mathcal M(S)\to \widetilde{\mathcal M}(S)\to \mathcal M(S)\to 0
\]
of Smith spaces, functorial in the profinite set $S$. Taking $S=\mathbb N\cup \{\infty\}$ and writing $\mathcal M(S)=W_1\oplus \mathbb R\cdot[\infty]$ by splitting off $\infty$, the space $W_1$ is a Smith space containing $\ell^1(\mathbb N)$ as a subspace with the same underlying $\mathbb R$-vector space (a compact convex generating set of $W$ is the space of sequences $x_0,x_1,\ldots \in [-1,1]$ with $\sum |x_i|\leq 1$). Then we get an extension
\[
0\to W\to \widetilde{W}\to W\to 0
\]
and we can take the pullback along $\ell^1\to W$ and the pushout along $W\to \mathbb R$ (summing all $x_i$)\footnote{Warning: This map is not well-defined, but after pullback to $\ell^1$ the extension can be reduced to a self-extension of $\ell^1\subset W$ by itself, and there is a well-defined map $\ell^1\to \mathbb R$ along which one can take the pushout.} to get an extension
\[
0\to \mathbb R\to\ ?\to \ell^1\to 0.
\]
This will be the extension constructed above.

We would like to make the following definition; this erroneous definition was stated in the lecture.

\begin{definition}[does not work] For a finite set $S$, let
\[
\widetilde{\mathbb R[S]}_{\leq c} = \{(x_s,y_s)\in \mathbb R[S]\times \mathbb R[S]\mid \sum_{s\in S}(|x_s|+|y_s-x_s\log|x_s||)\leq c\}.
\]
For a profinite set $S=\varprojlim_i S_i$, let
\[
\widetilde{\mathcal M(S)} = \bigcup_{c>0} \varprojlim_i \widetilde{\mathbb R[S]}_{\leq c}.
\]
\end{definition}

The problem with this definition is that the transition maps in the limit over $i$ do not preserve the subspaces $_{\leq c}$. It would be enough if there was some universal constant $C$ such that for any map of finite set $S\to T$, the set $\widetilde{\mathbb R[S]}_{\leq c}$ maps into $\widetilde{\mathbb R[T]}_{\leq C c}$. Unfortunately, even this is not true. We will see in the next lecture that replacing the $\ell^1$-norm implicit in the definition of $\widetilde{\mathbb R[S]}_{\leq c}$ with the $\ell^p$-norm for some $p<1$, this problem disappears.

One can salvage the definition for the Smith space $W_1$ (the direct summand of $\mathcal M(\mathbb N\cup \{\infty\})$. In this case, one simply directly build the extension $\widetilde{W_1}$, as follows:
\[
\widetilde{W_1} = \bigcup_{c>0} \{(x_0,x_1,\ldots,y_0,y_1,\ldots)\in \prod_{\mathbb N} [-c,c]\times \prod_{\mathbb N} [-c,c]\mid \sum_n(|x_n|+|y_n-x_n\log|x_n||)\leq c\}.
\]

\begin{proposition} The condensed set $\widetilde{W_1}$ has a natural structure of a condensed $\mathbb R$-vector space, and sits in an exact sequence
\[
0\to W_1\to \widetilde{W_1}\to W_1\to 0.
\]
\end{proposition}

\begin{proof} Surjectivity of $\widetilde{W_1}\to W_1$ is clear by taking $y_n=x_n\log|x_n|$. To see that it is a condensed $\mathbb R$-vector space, note that stability under addition follows from Lemma~\ref{lem:entropyalmostlinear}, and one similarly checks stability under scalar multiplication.
\end{proof}

\begin{exercise} Show that the extension $0\to W_1\to \widetilde{W_1}\to W_1\to 0$ has a Banach analogue: An extension $0\to \ell^1\to \widetilde{\ell^1}\to \ell^1\to 0$, sitting inside the previous sequence; cf.~\cite{KaltonPeck}.
\end{exercise}

The proposition shows that the class of $\mathcal M$-complete condensed $\mathbb R$-vector spaces is not stable under extensions. Indeed, if $\widetilde{W_1}$ were $\mathcal M$-complete, then the map $\mathbb N\cup\{\infty\}\to \widetilde{W_1}$ given by sending $n$ to the element with $x_n=1$ and all other $x_i,y_i=0$ would extend to a map $\mathcal M(\mathbb N\cup\{\infty\})\to \widetilde{W_1}$ vanishing on $\infty$, thus giving a section $W_1\to \widetilde{W_1}$; but the extension is nonsplit.

We can now also show that cokernels are not well-behaved. Indeed, consider also the Smith space version $W_\infty$ of $\ell^\infty$, given by $\bigcup_{c>0}\prod_{\mathbb N} [-c,c]$. There is a natural inclusion $W_1\subset W_\infty$.

\begin{proposition} The map of condensed sets $f: W_1\to W_\infty$ given by
\[
(x_0,x_1,\ldots)\mapsto (x_0\log|x_0|,x_1\log|x_1|,\ldots)
\]
induces, via projection $W_\infty\to W_\infty/W_1$, a nonzero map of condensed $\mathbb R$-vector spaces
\[
W_1\to W_\infty/W_1.
\]
\end{proposition}

\begin{proof} Let us check first that it is a map of condensed abelian groups. For this, we need to see that two maps (of condensed sets) $W_1\times W_1\to W_\infty/W_1$ agree. Their difference is the projection along $W_\infty\to W_\infty/W_1$ of the map $W_1\times W_1\to W_\infty$ sending pairs of $(x_0,x_1,\ldots)$ and $(y_0,y_1,\ldots)$ to
\[
(x_0\log|x_0|+y_0\log|y_0|-(x_0+y_0)\log|x_0+y_0|,x_1\log|x_1|+y_1\log|y_1|-(x_1+y_1)\log|x_1+y_1|,\ldots),
\]
which lies in $W_1$ by Lemma~\ref{lem:entropyalmostlinear}. This gives additivity. One checks $\mathbb R$-linearity similarly.

To see that it is nonzero, it suffices to see that the image of $W_1$ under $f$ is not contained in the subspace $W_1\subset W_\infty$. But the sequence $z_n=(\tfrac 1n,\ldots,\tfrac 1n,0,\ldots)$ for varying $n$ (with $n$ occurences of $\tfrac 1n$) defines a map $\mathbb N\cup\{\infty\}\to W_1$, whose image under $f$ is the sequence $(-\frac{\log n}n,\ldots,-\frac{\log n}n,0,\ldots)$ which does not have bounded $\ell^1$-norm.
\end{proof}

Using this map, one can recover the extension $0\to W_1\to \widetilde{W_1}\to W_1\to 0$ as the pullback of the canonical extension $0\to W_1\to W_\infty\to W_\infty/W_1\to 0$.

Moreover, this finally answers the question at the beginning of this lecture: The map $\mathcal M(\mathbb N\cup \{\infty\})\to W_1\to W_\infty/W_1$ is zero when restricted to $\mathbb N\cup \{\infty\}$, so the cokernel $W_\infty/W_1$ of a map of Smith spaces does not satisfy the property that any map $S\to W_\infty/W_1$ extends uniquely to a map $\mathcal M(S)\to W_\infty/W_1$.

As a final topic of this lecture, we construct a kind of ``real analogue of $B_{\mathrm{dR}}^+$''. We note that the extension $\widetilde{W_1}$ is in fact a flat condensed $\mathbb R[\epsilon]$-module, where $\mathbb R[\epsilon] = \mathbb R[t]/t^2$, via letting $\epsilon$ map sequences $(x_0,x_1,\ldots,y_0,y_1,\ldots)$ to $(0,0,\ldots,x_0,x_1,\ldots)$.

One may wonder whether one can build an infinite self-extension of $W_1$'s, leading to a flat condensed $\mathbb R[[t]]$-module $W_{1,\mathbb R[[t]]}$ with $W_{1,\mathbb R[[t]]}\otimes_{\mathbb R[[t]]} \mathbb R=W_1$. Ideally, one would want to have a similar self-extension of all $\mathcal M(S)$, functorial in the profinite set $S$; as before, this does not actually work, but will work in an analogous setting as studied in the next lecture.

For the construction, consider the multiplicative map
\[
\mathbb R\to \mathbb R[[t]]
\]
sending $x\in \mathbb R$ to the series
\[
[x] := x|x|^t = x+x\log |x| t + \tfrac 12 x \log^2|x| t^2 + \ldots + \tfrac 1{n!} x \log^n |x| t^n + \ldots .
\]
Any element of $\mathbb R[t]/t^n$ can be written in a unique way in the form
\[
[x_0] + [x_1]t + \ldots + [x_{n-1}]t^{n-1}
\]
with $x_i\in \mathbb R$. For $n=2$, this writes $a+bt$ in the form $[a] + [b-a\log|a|]t$. For $n=3$, the formula is already rather complicated:
\[
a+bt+ct^2 = [a] + [b-a\log|a|]t + [c-(b-a\log|a|)\log|b-a\log|a||-\tfrac 12 a\log^2|a|]t^2.
\]
Thus much complexity is hidden in this isomorphism.

\begin{lemma} For any finite set $S$ and real number $c>0$, let
\[
(\mathbb R[t]/t^n)[S]_{\leq c} = \{(\sum_{i=0}^{n-1} [x_{i,s}]t^i)_s\in \mathbb R[t]/t^n[S]\mid \sum_{i,s} |x_{i,s}|\leq c\}.
\]
Then $(\mathbb R[t]/t^n)[S]_{\leq c}$ is stable under multiplication by $[-\epsilon,\epsilon]$ for some $\epsilon=\epsilon(n)$, and there is a constant $C=C(n)$ such that
\[
(\mathbb R[t]/t^n)[S]_{\leq c} + (\mathbb R[t]/t^n)[S]_{\leq c}\subset (\mathbb R[t]/t^n)[S]_{\leq C c}.
\]
\end{lemma}

\begin{proof} Let us explain the case of addition; we leave the case of scalar multiplication as an exercise. In the given coordinates, there are functions $S_0,\ldots,S_{n-1}$, where $S_i$ is a continuous function of $x_0,\ldots,x_i,y_0,\ldots,y_i$, such that
\[
[x_0]+[x_1]t+\ldots+[x_{n-1}]t^{n-1} + [y_0]+[y_1]t+\ldots+[y_{n-1}]t^{n-1} = [S_0]+[S_1]t+\ldots+[S_{n-1}]t^{n-1}\in \mathbb R[t]/t^n.
\]
Using that multiplication by $[\lambda]$ for any $\lambda\in \mathbb R$ commutes with all operations, one sees that
\[
S_i(\lambda x_0,\ldots,\lambda x_i,\lambda y_0,\ldots,\lambda y_i) = \lambda S_i(x_0,\ldots,x_i,y_0,\ldots,y_i).
\]
This implies that
\[
S_i(x_0,\ldots,x_i,y_0,\ldots,y_i)\leq C_i (|x_0|+\ldots+|x_i|+|y_0|+\ldots+|y_i|)
\]
by taking $C_i>0$ to be the maximum absolute value taken by $S_i$ on the compact set where all $x_j,y_j\in [-1,1]$ and at least one of them is $\pm 1$. This estimate easily implies the desired stability under addition.
\end{proof}

Ideally, one would like to use the lemma to see that for any profinite set $S=\varprojlim_i S_i$, one can define a condensed $\mathbb R[t]/t^n$-module
\[
\mathcal M(S,\mathbb R[t]/t^n) := \bigcup_{c>0} \varprojlim_i (\mathbb R[t]/t^n)[S_i]_{\leq c}.
\]
However, as before this construction does not work because the transition maps do not preserve the desired bounds. However, one can build the corresponding infinite self-extension of $W_1$. We believe this is related to the ``complex interpolation of Banach spaces'' as studied for example in \cite{ComplexInterpolationBanach}, \cite{ComplexInterpolationHilbert}.

\newpage

\section{Lecture VI: Statement of main result}

In the last lecture, we saw that a naive adaptation of solid abelian groups to the case of the real numbers, using the spaces $\mathcal M(S)$ of signed Radon measures, does not work; but its failure gives rise to interesting phenomena.

In this lecture, however, we want to find a variant that does work. We want to find a full subcategory of condensed $\mathbb R$-vector spaces that is abelian, stable under kernels, cokernels, and extensions, and contains all $\mathcal M$-complete objects (but is still enforcing a nontrivial condition). This forces us to include also all the extensions constructed in the previous lecture. These are not $\mathcal M$-complete; in the setting of extensions of Banach spaces, these extensions are not locally convex anymore.

The following weakening of convexity has been studied in the Banach space literature.

\begin{definition}\label{def:pbanach} For $0<p\leq 1$, a $p$-Banach space is a topological $\mathbb R$-vector space $V$ such that there exists a $p$-norm, i.e.~a continuous map
\[
||\cdot||: V\to \mathbb R_{\geq 0}
\]
with the following properties:
\begin{enumerate}
\item For any $v\in V$, the norm $||v||=0$ if and only if $v=0$;
\item For all $v\in V$ and $a\in \mathbb R$, one has $||av|| = |a|^p ||v||$;
\item For all $v,w\in V$, one has $||v+w||\leq ||v||+||w||$;
\item The sets $\{v\in V\mid ||v||<\epsilon\}$ for varying $\epsilon\in \mathbb R_{>0}$ define a basis of open neighborhoods of $0$;
\item For any sequence $v_0,v_1,\ldots\in V$ with $||v_i-v_j||\to 0$ as $i,j\to\infty$, there exists a (necessarily unique) $v\in V$ with $||v-v_i||\to 0$.
\end{enumerate}
\end{definition}

Thus, the only difference is in the scaling behaviour. We note that $V$ is a $p$-Banach then it is a $p'$-Banach for all $p'\leq p$, as if $||\cdot||$ is a $p$-norm, then $||\cdot||^{p'/p}$ is a $p'$-norm. A quasi-Banach space is a topological $\mathbb R$-vector space that is a $p$-Banach for some $p>0$.

A key result is the following.

\begin{theorem}[\cite{KaltonConvexityType}]\label{thm:pBanachExt} Any extension of $p$-Banach spaces is a $p'$-Banach for all $p'<p$.
\end{theorem}

This suggests that the extension problems we encountered disappear if instead of convexity, we ask for $p$-convexity\footnote{Note that $p$-convexity for $p<1$ is weaker than convexity, and might also be called concavity.} for all $p<1$. Here $p$-convexity is the assertion that for all $a_1,\ldots,a_n\in \mathbb R$ with $|a_1|^p+\ldots+|a_n|^p\leq 1$, the ball $\{v\in V\mid ||v||<\epsilon\}$ is stable under $(v_1,\ldots,v_n)\mapsto a_1v_1+\ldots+a_nv_n$.

Let us do the obvious translation into the condensed world.

\begin{definition} Let $0<p\leq 1$ be a real number. For any finite set $S$ and real number $c>0$, let
\[
\mathbb R[S]_{\ell^p\leq c} = \{(a_s)_s\in \mathbb R[S]\mid \sum_{s\in S} |a_s|^p\leq c\},
\]
and for a profinite set $S=\varprojlim_i S_i$, let
\[
\mathcal M_p(S) = \bigcup_{c>0} \varprojlim_i \mathbb R[S_i]_{\ell^p\leq c}.
\]
\end{definition}

For any fixed value of $p$, we run into the same problem with non-split self-extensions of $\mathcal M_p(S)$ as in the last lecture, except that for $p<1$ things are better-behaved, in that one can define these self-extensions functorially in the profinite set $S$: For any finite set $S$, and $n\geq 1$, one can define
\[
(\mathbb R[T]/T^n)[S]_{\ell^p\leq c} := \{(\sum_{i=0}^{n-1} [x_{i,s}]T^i)_s\mid \sum_{i,s} |x_{i,s}|^p\leq c\}.
\]
Then there is a constant $C=C(n,p)<\infty$ such that for any map of finite sets $S\to S'$, the image of $(\mathbb R[T]/T^n)[S]_{\ell^p\leq c}$ is contained in $(\mathbb R[T]/T^n)[S']_{\ell^p\leq Cc}$, by Proposition~\ref{prop:stableimage} below. In particular, the subset $(\mathbb R[T]/T^n)[S]_{\ell^p\leq c,\mathrm{stable}}\subset (\mathbb R[T]/T^n)[S]_{\ell^p\leq c}$ of all elements whose images under any map $S\to S'$ are still in the $\ell^p\leq c$-subspace, is cofinal with the full space, and now is compatible with transition maps. In other words, for a profinite set $S=\varprojlim_i S_i$, we can now functorially define
\[
\mathcal M_p(S,\mathbb R[T]/T^n) = \bigcup_{c>0} \varprojlim_i (\mathbb R[T]/T^n)[S_i]_{\ell^p\leq c,\mathrm{stable}}
\]
and thus, in the limit over $n$, define a $B_{\mathrm{dR}}^+$-theory for any $p<1$. As we will not have direct use for it (and, just like $\mathcal M_p$ itself, it does not define an abelian category of modules), we will not go into more details about this. We will just note that ``the'' real numbers are not well-defined anymore: To define a category of modules over $\mathbb R$, one needs to specify in addition $0<p<1$; and in some sense an infinitesimal deformation of $p$ gives rise to $B_{\mathrm{dR}}^+$, so there is in some sense a $1$-parameter family of versions of the real numbers. We will make this picture more precise later.

We used the following proposition:

\begin{proposition}\label{prop:stableimage} There is a constant $C=C(n,p)<\infty$ such that for all maps $f: S\to S'$ of finite sets,
\[
f((\mathbb R[T]/T^n)[S]_{\ell^p\leq c})\subset (\mathbb R[T]/T^n)[S']_{\ell^p\leq Cc}.
\]
\end{proposition}

\begin{proof} Decomposing into fibres over $S'$, we can assume that $S'$ is a point. Note that multiplication by $[\lambda]$ for $\lambda>0$ maps the $\ell^p\leq c$-subspace isomorphically to the $\ell^p\leq \lambda^p c$-subspace. Thus, by rescaling, we can assume that $c=1$. It remains to see that there is some bounded subset $B\subset \mathbb R[T]/T^n$ such that for all finite sets $S$ and all $(\sum_{i=0}^{n-1} [x_{i,s}]T^i)_s\in (\mathbb R[T]/T^n)[S]_{\ell^p\leq 1}$, one has
\[
\sum_{s\in S} \sum_{i=0}^{n-1} [x_{i,s}]T^i\in B\subset \mathbb R[T]/T^n.
\]
By induction, such a statement holds true for the sum over $i=1,\ldots,n-1$. Thus, it suffices to see that for all integers $m=|S|$ and all $x_1,\ldots,x_m$ with $\sum_{j=1}^m |x_j|^p\leq 1$, one has
\[
\sum_{j=1}^m [x_j]\in B'\subset \mathbb R[T]/T^n
\]
for some bounded subset $B'$. In other words, we need to see that for all $i=0,\ldots,n-1$, one has
\[
\sum_{j=1}^m x_j\log^i|x_j|\leq C_i
\]
for some constant $C_i$. But for all sufficiently small $x$, one has $\log^i |x|\leq |x|^{p-1}$, so neglecting large $x_j$ (which can only give a bounded contribution), one has $\sum_{j=1}^m x_j\log^i|x_j|\leq \sum_{j=1}^m |x_j|^p\leq 1$.
\end{proof}

Without further ado, let us now state the first main theorem of this course. Define
\[
\mathcal M_{<p}(S) = \varinjlim_{p'<p} \mathcal M_{p'}(S).
\]

\begin{theorem}\label{thm:pliquid} Fix any $0<p\leq 1$. For a condensed $\mathbb R$-vector space $V$, the following conditions are equivalent.
\begin{enumerate}
\item For all $p'<p$ and all maps $f: S\to V$ from a profinite set $S$, there is a unique extension to a map
\[
\tilde{f}_{p'}: \mathcal M_{p'}(S)\to V
\]
of condensed $\mathbb R$-vector spaces.
\item For all maps $f: S\to V$ from a profinite set $S$, there is a unique extension to a map
\[
\tilde{f}_{<p}: \mathcal M_{<p}(S)\to V
\]
of condensed $\mathbb R$-vector spaces.
\item One can write $V$ as the cokernel of a map
\[
\bigoplus_i \mathcal M_{<p}(S_i)\to \bigoplus_j \mathcal M_{<p}(S_j').
\]
\end{enumerate}

If $V$ satisfies these conditions, we say that $V$ is $p$-liquid. The class $\mathrm{Liq}_p(\mathbb R)$ of $p$-liquid $\mathbb R$-vector spaces is an abelian subcategory stable under all kernels, cokernels, and extensions. It is generated by the compact projective objects $\mathcal M_{<p}(S)$, $S$ extremally disconnected. Moreover:
\begin{enumerate}
\item The inclusion $\mathrm{Liq}_p(\mathbb R)\to \Cond(\mathbb R)$ has a left adjoint $V\mapsto V^{\mathrm{liq}}_p$, and there is a (necessarily unique) symmetric monoidal tensor product $\otimes^{\mathrm{liq}}_p$ on $\mathrm{Liq}_p(\mathbb R)$ making $V\mapsto V^{\mathrm{liq}}_p$ symmetric monoidal.
\item The functor $D(\mathrm{Liq}_p(\mathbb R))\to D(\Cond(\mathbb R))$ is fully faithful, and admits a left adjoint, which is the left derived functor $C\mapsto C^{L\mathrm{liq}}_p$ of $V\mapsto V^{\mathrm{liq}}_p$. There is a unique symmetric monoidal tensor product $\otimes^{L\mathrm{liq}}_p$ on $D(\mathrm{Liq}_p(\mathbb R))$ making $C\mapsto C^{L\mathrm{liq}}_p$ symmetric monoidal; it is the left derived functor of $\otimes^{\mathrm{liq}}_p$.
\item An object $C\in D(\Cond(\mathbb R))$ lies in $D(\mathrm{Liq}_p(\mathbb R))$ if and only if all $H^i(C)\in \mathrm{Liq}_p(\mathbb R)$; in that case, for all profinite sets $S$ and all $p'<p$, one has
\[
R\intHom_{\mathbb R}(\mathcal M_{p'}(S),C) = R\intHom_{\mathbb R}(\mathbb R[S],C).
\]
\end{enumerate}
\end{theorem}

In other words, one gets an extremely well-behaved theory of $p$-liquid $\mathbb R$-vector spaces. In the language of last semester, the present theory defines an analytic ring \cite[Definition 7.4, Proposition 7.5]{Condensed}. In fact, by \cite[Lemma 5.9, Remark 5.11]{Condensed}, all we need to prove is the following result.

\begin{theorem}\label{thm:keyreal} Let $0<p\leq 1$, let
\[
g: \bigoplus_i \mathcal M_{<p}(S_i)\to \bigoplus_j \mathcal M_{<p}(S_j')
\]
be a map of condensed $\mathbb R$-vector spaces, where all $S_i$ and $S_j'$ are extremally disconnected, and let $V$ be the kernel of $g$. Then for any $p'<p$ and any profinite set $S$, the map
\[
R\intHom_{\mathbb R}(\mathcal M_{p'}(S),V)\to R\intHom_{\mathbb R}(\mathbb R[S],V)
\]
is a quasi-isomorphism in $D(\Cond(\Ab))$.
\end{theorem}

Unfortunately, at this point, we have to pay a serious prize for going into the non-locally convex setting. More precisely, in order to compute $R\intHom_{\mathbb R}(\mathcal M_{p'}(S),V)$, i.e.~the complex of $\mathbb R$-linear maps $\mathcal M_{p'}(S)\to V$, we first have to understand maps of condensed sets $\mathcal M_{p'}(S)\to V$, i.e.~
\[
R\Gamma(\mathcal M_{p'}(S),V)
\]
for $V$ as above. We may restrict to the compact Hausdorff subspace $\mathcal M_{p'}(S)_{\leq 1}$. In order for this to be helpful, we need to be able to compute the cohomology
\[
H^i(S,V)
\]
for compact Hausdorff $S$ with coefficients in spaces $V$ like above; in particular, for $p$-Banach spaces.

The proof of \cite[Theorem 3.3]{Condensed} gives the following result.

\begin{theorem}\label{thm:pbanachcohom} Let $V$ be a Banach space and let $S$ be a compact Hausdorff space. Then $H^i(S,\underline{V})=0$ for $i>0$. Similarly, if $V$ is a $p$-Banach for any $p\leq 1$ and $S$ is a profinite set, then $H^i(S,\underline{V})=0$ for $i>0$.
\end{theorem}

However, a key step in the proof is the existence of a partition of unity, the application of which uses that $V$ is locally convex critically. In particular, the first part fails if $V$ is only a $p$-Banach for $p<1$.\footnote{We believe that if $S$ is finite-dimensional, one can prove the same result; but the relevant compact Hausdorff spaces are highly infinite-dimensional.}

This means that for the key computation, we have to further resolve $\mathcal M_{p'}(S)_{\leq 1}$ by profinite sets; and we should do so in a way that keeps the resulting resolution sufficiently explicit. In other words, the condensed formalism now forces us to resolve the real numbers, and vector spaces over it, in terms of profinite sets, explicitly!\footnote{This is a difficulty that is not seen in the classical quasi-Banach space literature; and so Theorem~\ref{thm:pBanachExt} and its techniques of proof cannot easily be extended to the present situation.}

We are now ready for the key turn. We will generalize the present question about the real numbers to an arithmetic ring that is a countable union of profinite subsets.

Let us define this ring. It depends on a real number $0<r<1$. We define the condensed ring $\mathbb Z((T))_r$ whose $S$-valued points are
\[
\mathbb Z((T))_r(S) = \{\sum_{n\in \mathbb Z} a_n T^n\mid a_n\in C(S,\mathbb Z), \sum_{n\in \mathbb Z} |a_n|r^n<\infty\}\subset \mathbb Z((T))(S).
\]
More precisely, this is the increasing union of the subsets
\[
\mathbb Z((T))_{r,\leq c}(S) = \{\sum_{n\in \mathbb Z} a_n T^n\mid a_n\in C(S,\mathbb Z), \sum_{n\in \mathbb Z} |a_n|r^n\leq c\}
\]
for varying reals $c>0$; and here $\sum_{n\in \mathbb Z} |a_n|r^n\leq c$ means that for all $s\in S$, $\sum_{n\in \mathbb Z} |a_n(s)|r^n\leq c$.

\begin{proposition} The condensed set $\mathbb Z((T))_{r,\leq c}$ is a profinite set.
\end{proposition}

\begin{proof} As multiplication by $T$ is an isomorphism between $\mathbb Z((T))_{r,\leq c}$ and $\mathbb Z((T))_{r,\leq rc}$, we can assume that $c\leq 1$. In that case $\mathbb Z((T))_{r,\leq c}\subset \mathbb Z[[T]]=\prod_{n\geq 0} \mathbb Z\cdot T^n$. Now $\mathbb Z((T))_{r,\leq c}$ can be written as the inverse limit of
\[
\{\sum_{n=0}^m a_n T^n\mid \sum_{n=0}^m |a_n|r^n\leq c\}\subset \prod_{n=0}^m \mathbb Z\cdot T^n,
\]
each of which is a finite set.
\end{proof}

For any finite set $S$, we now write the free module $\mathbb Z((T))_r[S]$ as the increasing union of
\[
\mathbb Z((T))_r[S]_{\leq c} := \{\sum_{n\in \mathbb Z,s\in S} a_{n,s} T^n [s]\mid a_{n,s}\in C(S,\mathbb Z), \sum_{n\in \mathbb Z, s\in S} |a_{n,s}|r^n\leq c\};
\]
each of these is again a profinite set. The addition defines maps
\[
\mathbb Z((T))_r[S]_{\leq c} \times \mathbb Z((T))_r[S]_{\leq c'}\to \mathbb Z((T))_r[S]_{\leq c+c'}
\]
and everything is covariantly functorial in $S$. For any profinite set $S=\varprojlim_i S_i$ we define
\[
\mathcal M(S,\mathbb Z((T))_r)_{\leq c} := \varprojlim_i \mathbb Z((T))_r[S_i]_{\leq c}
\]
which is a profinite set, functorial in $S$, equipped with an addition map
\[
\mathcal M(S,\mathbb Z((T))_r)_{\leq c} \times \mathcal M(S,\mathbb Z((T))_r)_{\leq c'}\to \mathcal M(S,\mathbb Z((T))_r)_{\leq c+c'}.
\]
In particular, the colimit
\[
\mathcal M(S,\mathbb Z((T))_r) = \bigcup_{c>0} \mathcal M(S,\mathbb Z((T))_r)_{\leq c}
\]
is a condensed abelian group, and in fact one easily sees that it is a $\mathbb Z((T))_r$-module.

In the next lecture, we will prove the following theorem, relating this theory to the various $\ell^p$-theories over $\mathbb R$. Part (1) is due to Harbater, \cite[Lemma 1.5]{HarbaterConvergentArithmetic}.

\begin{theorem}\label{thm:Ainflike} Let $0<r'<r$ and consider the map
\[
\theta_{r'}: \mathbb Z((T))_r\to \mathbb R: \sum a_n T^n\mapsto \sum a_n (r')^n.
\]
This map is surjective. Moreover:
\begin{enumerate}
\item The kernel of $\theta_{r'}$ is generated by a nonzerodivisor $f_{r'}\in \mathbb Z((T))_r$.
\item For any profinite set $S$, there is a canonical isomorphism
\[
\mathcal M(S,\mathbb Z((T))_r)/(f_{r'})\cong \mathcal M_p(S)
\]
of condensed $\mathbb Z((T))_r/(f_{r'})=\mathbb R$-modules, where $0<p<1$ is chosen so that $(r')^p=r$.
\item More generally, for any profinite set $S$ and any $n\geq 1$, there is a canonical isomorphism
\[
\mathcal M(S,\mathbb Z((T))_r)/(f_{r'})^n\cong \mathcal M_p(S,\mathbb R[X]/X^n)
\]
of condensed $\mathbb Z((T))_r/(f_{r'})^n\cong \mathbb R[X]/X^n$-modules.
\end{enumerate}
\end{theorem}

In other words, the theory over $\mathbb Z((T))_r$ specializes to all different $\mathcal M_p$-theories for $p<1$, and using infinitesimal variations one even recovers the $B_{\mathrm{dR}}^+$-theories for all $p$. In some sense, one could regard $\mathbb Z((T))_r$ as some version of Fontaine's $A_{\mathrm{inf}}$.

On the other hand, this again indicates that for a fixed $r$, the theory will not work, so we pass to a colimit again. Another twist, necessary for the proof, will be to isolate the desired theory within all condensed $\mathbb Z[T^{-1}]$-modules; this is related to the possibility to isolate $p$-liquid $\mathbb R$-vector spaces inside all condensed abelian groups (i.e., the forgetful functor is fully faithful).

\begin{theorem}\label{thm:liquidarithmetic} Fix any $0<r<1$. For a condensed $\mathbb Z[T^{-1}]$-module $M$, the following conditions are equivalent.
\begin{enumerate}
\item For all $1>r'>r$ and all maps $f: S\to M$ from a profinite set $S$, there is a unique extension to a map
\[
\tilde{f}_{r'}: \mathcal M(S,\mathbb Z((T))_{r'})\to M
\]
of condensed $\mathbb Z[T^{-1}]$-modules.
\item For all maps $f: S\to M$ from a profinite set $S$, there is a unique extension to a map
\[
\tilde{f}_{>r}: \mathcal M(S,\mathbb Z((T))_{>r})\to M
\]
of condensed $\mathbb Z[T^{-1}]$-modules.
\item One can write $M$ as the cokernel of a map
\[
\bigoplus_i \mathcal M(S_i,\mathbb Z((T))_{>r})\to \bigoplus_j \mathcal M(S_j',\mathbb Z((T))_{>r}).
\]
\end{enumerate}

The class $\mathrm{Liq}_r(\mathbb Z[T^{-1}])$ of such condensed $\mathbb Z[T^{-1}]$-modules is an abelian subcategory stable under all kernels, cokernels, and extensions. It is generated by the compact projective objects $\mathcal M(S,\mathbb Z((T))_{>r})$, $S$ extremally disconnected. Moreover:
\begin{enumerate}
\item The inclusion $\mathrm{Liq}_r(\mathbb Z[T^{-1}])\to \Cond(\mathbb Z[T^{-1}])$ has a left adjoint $M\mapsto M^{\mathrm{liq}}_r$, and there is a (necessarily unique) symmetric monoidal tensor product $\otimes^{\mathrm{liq}}_r$ on $\mathrm{Liq}_r(\mathbb Z[T^{-1}])$ making $M\mapsto M^{\mathrm{liq}}_r$ symmetric monoidal.
\item The functor $D(\mathrm{Liq}_r(\mathbb Z[T^{-1}]))\to D(\Cond(\mathbb Z[T^{-1}]))$ is fully faithful, and admits a left adjoint, which is the left derived functor $C\mapsto C^{L\mathrm{liq}}_r$ of $M\mapsto M^{\mathrm{liq}}_r$. There is a unique symmetric monoidal tensor product $\otimes^{L\mathrm{liq}}_r$ on $D(\mathrm{Liq}_r(\mathbb Z[T^{-1}]))$ making $C\mapsto C^{L\mathrm{liq}}_r$ symmetric monoidal; it is the left derived functor of $\otimes^{\mathrm{liq}}_r$.
\item An object $C\in D(\Cond(\mathbb Z[T^{-1}]))$ lies in $D(\mathrm{Liq}_r(\mathbb Z[T^{-1}]))$ if and only if all $H^i(C)\in \mathrm{Liq}_r(\mathbb Z[T^{-1}])$; in that case, for all profinite sets $S$ and all $1>r'>r$, one has
\[
R\intHom_{\mathbb Z[T^{-1}]}(\mathcal M(S,\mathbb Z((T))_{r'}),C) = R\intHom_{\mathbb Z[T^{-1}]}(\mathbb Z[T^{-1}][S],C).
\]
\end{enumerate}
\end{theorem}

It follows formally that $\mathrm{Liq}_r(\mathbb Z[T^{-1}])$ also sits fully faithfully in $\Cond(\mathbb Z((T))_{>r})$ (note that the tensor unit of $\mathrm{Liq}_r(\mathbb Z[T^{-1}])$ is $\mathbb Z((T))_{>r}$), and that everything holds true with $\mathbb Z[T^{-1}]$ replaced by $\mathbb Z((T))_{>r}$; we will then also write $\mathrm{Liq}_r(\mathbb Z((T))_{>r})$ or simply $\mathrm{Liq}(\mathbb Z((T))_{>r})$ for $\mathrm{Liq}_r(\mathbb Z[T^{-1}])$. However, the proof will require us to prove the finer statement with base ring $\mathbb Z[T^{-1}]$.

Again, it reduces formally to the following assertion.

\begin{theorem}\label{thm:key} Let $K$ be a condensed $\mathbb Z[T^{-1}]$-module that is the kernel of some map
\[
f: \bigoplus_{i\in I} \mathcal M(S_i,\mathbb Z((T))_{>r})\to \bigoplus_{j\in J} \mathcal M(S_j',\mathbb Z((T))_{>r})
\]
where all $S_i$ and $S_j'$ are extremally disconnected. Then for all $1>r'>r$ and all profinite sets $S$, the map
\[
R\intHom_{\mathbb Z[T^{-1}]}(\mathcal M(S,\mathbb Z((T))_{r'}),K)\to R\intHom_{\mathbb Z[T^{-1}]}(\mathbb Z[T^{-1}][S],K)
\]
in $D(\Cond(\Ab))$ is an isomorphism.
\end{theorem}

In the next lecture, we will show that $\mathbb Z((T))_{>r}$ is a principal ideal domain, and prove Theorem~\ref{thm:Ainflike}. In the appendix to this lecture, we explain how the rest of the assertions of this lecture reduce to Theorem~\ref{thm:key}.

\newpage

\section*{Appendix to Lecture VI: Recollections on analytic rings}

In \cite[Lecture VII]{Condensed}, we defined analytic rings, abstracting the examples we have seen. For simplicity, we will always work in the commutative case.

\begin{definition}[{\cite[Definition 7.1, 7.4]{Condensed}}] A theory of measures on a condensed ring $A$ is a functor
\[
\mathcal M: \{\mathrm{extremally\ disconnected\ sets}\}\to A\Mod: S\mapsto \mathcal M[S]
\]
to the category $A\Mod$ of $A$-modules in condensed abelian groups, taking finite disjoint unions to finite products, together with a natural transformation $S\to \mathcal M[S]$ of ``Dirac measures''.

An analytic ring is a theory of measures $(A,\mathcal M)$ such that for all index sets $I$, $J$ and extremally disconnected sets $S_i$, $S_j'$, and all maps
\[
f: \bigoplus_i \mathcal M[S_i]\to \bigoplus_j \mathcal M[S_j']
\]
in $A\Mod$, with kernel $K$, the map
\[
R\intHom_A(\mathcal M[S],K)\to R\intHom_A(A[S],K)
\]
is an isomorphism in $D(A\Mod)$, for all extremally disconnected sets $S$.
\end{definition}

\begin{remark} The definition of analytic rings was stated in a slightly different, but equivalent way: the current condition implies the one in \cite[Definition 7.4]{Condensed} by writing $C$ as the limit of its Postnikov truncations; conversely, use the argument of \cite[Lemma 5.10]{Condensed}.
\end{remark}

Thus, Theorem~\ref{thm:key} implies that $\mathbb Z[T^{-1}]$ with $\mathcal M[S]=\mathcal M(S,\mathbb Z((T))_{>r})$ is an analytic ring. (It says slightly more as we allow $S$ profinite there, and make a claim about all $r'<r$; we will explain what extra information this gives.) This immediately gives a lot of information:

\begin{proposition}[{\cite[Proposition 7.5]{Condensed}}] Let $(A,\mathcal M)$ be an analytic ring.
\begin{enumerate}
\item[{\rm (i)}] The full subcategory
\[
(A,\mathcal M)\Mod\subset A\Mod
\]
of all $A$-modules $M$ in condensed abelian groups such that for all extremally disconnected sets $S$, the map
\[
\Hom_A(\mathcal M[S],M)\to M(S)
\]
is an isomorphism, is an abelian category stable under all limits, colimits, and extensions. The objects $\mathcal M[S]$ for $S$ extremally disconnected form a family of compact projective generators. It admits a left adjoint
\[
A\Mod\to (A,\mathcal M)\Mod: M\mapsto M\otimes_A (A,\mathcal M)
\]
that is the unique colimit-preserving extension of $A[S]\mapsto \mathcal M[S]$. There is a unique symmetric monoidal tensor product $\otimes_{(A,\mathcal M)}$ on $(A,\mathcal M)\Mod$ making the functor
\[
A\Mod\to (A,\mathcal M)\Mod: M\mapsto M\otimes_A (A,\mathcal M)
\]
symmetric monoidal.
\item[{\rm (ii)}] The functor
\[
D((A,\mathcal M)\Mod)\to D(A\Mod)
\]
is fully faithful, and its essential image is stable under all limits and colimits and given by those $C\in D(A\Mod)$ such that for all extremally disconnected $S$, the map
\[
R\Hom_A(\mathcal M[S],C)\to R\Hom_A(A[S],C)
\]
is an isomorphism; in that case, also the $R\intHom$'s agree. An object $C\in D(A\Mod)$ lies in $D((A,\mathcal M)\Mod)$ if and only if all $H^i(C)$ are in $(A,\mathcal M)\Mod\subset A\Mod$. The inclusion $D((A,\mathcal M)\Mod)\subset D(A\Mod)$ admits a left adjoint
\[
D(A\Mod)\to D((A,\mathcal M)\Mod): C\mapsto C\otimes^L_A (A,\mathcal M)
\]
that is the left derived functor of $M\mapsto M\otimes_A (A,\mathcal M)$. There is a unique symmetric monoidal tensor product $\otimes^L_{(A, \mathcal M)}$ on $D((A,\mathcal M)\Mod)$ making the functor
\[
D(A\Mod)\to D((A,\mathcal M)\Mod): C\mapsto C\otimes^L_A (A,\mathcal M)
\]
symmetric monoidal.
\end{enumerate}
\end{proposition}

This shows that Theorem~\ref{thm:key} implies most of Theorem~\ref{thm:liquidarithmetic}. We note one omission in the statement of the proposition: It is not claimed that $\otimes^L_{(A, \mathcal M)}$ is the left derived functor of $\otimes_{(A, \mathcal M)}$. This is equivalent to the assertion that for all extremally disconnected $S,S'$, the complex
\[
\mathcal M[S]\otimes^L_{(A,\mathcal M)} \mathcal M[T] = A[S\times T]\otimes^L_A (A,\mathcal M)
\]
sits in degree $0$ (here, the displayed isomorphism comes from $\otimes^L_A (A,\mathcal M)$ being symmetric monoidal, and $A[S]\otimes^L_A A[T] = A[S\times T]$). But in the situation of Theorem~\ref{thm:liquidarithmetic}, we defined $\mathcal M[S]$ for all profinite $S$, and Theorem~\ref{thm:key} ensures that
\[
A[S]\otimes^L_A (A,\mathcal M) = \mathcal M[S]
\]
sits in degree $0$, for all profinite sets $S$. Indeed, to check the displayed isomorphism, note that the right hand side lies in $(A,\mathcal M)\Mod$, and the $R\Hom$ into any complex $C\in D((A,\mathcal M)\Mod)$ agrees, by representing $C$ by a complex of direct sums of projectives, taking the limit of Postnikov truncations, and using Theorem~\ref{thm:key}.

In Theorem~\ref{thm:liquidarithmetic}, we used the strange base ring $\mathbb Z[T^{-1}]$. We note that for analytic rings, one can always change the base ring to its natural choice, and maintain an analytic ring:

\begin{proposition}\label{prop:normalizeanalyticring} Let $(A,\mathcal M)$ be an analytic ring. Then $A' = \mathcal M[\ast]$ is naturally a condensed ring with a map of condensed rings $A\to A'$, and the $A$-module structure on $\mathcal M[S]$ naturally refines to an $A'$-module, for all extremally disconnected $S$, defining a theory of measures $(A',\mathcal M)$ with a map $(A,\mathcal M)\to (A',\mathcal M)$.

The theory of measures $(A',\mathcal M)$ is an analytic ring, the map $(A,\mathcal M)\to (A',\mathcal M)$ defines a map of analytic rings, and the forgetful functor $D((A',\mathcal M)\Mod)\to D((A,\mathcal M)\Mod)$ is an equivalence.
\end{proposition}

\begin{proof} Note that $A'=\mathcal M[\ast]$ is the tensor unit in $(A,\mathcal M)\Mod$ (as the image of the tensor unit $A[\ast]$ under the symmetric monoidal functor $\otimes_A (A,\mathcal M)$), and hence naturally a condensed $A$-algebra. Moreover, every object of $(A,\mathcal M)\Mod$, in particular $\mathcal M[S]$, is naturally a condensed module over it.

We show that for any $K\in (A,\mathcal M)\Mod$, the natural map
\[
R\intHom_{A'}(\mathcal M[S],K)\to R\intHom_A(\mathcal M[S],K)
\]
is an isomorphism; this will easily prove all assertions. Now we use the bar resolution
\[
\ldots \mathcal M[S]\otimes^L_A A'\otimes^L_A A'\rightrightarrows \mathcal M[S]\otimes^L_A A'\to \mathcal M[S].
\]
It is then sufficient to show that for any $n\geq 1$,
\[
R\intHom_{A'}(\mathcal M[S]\otimes^L_A A'\otimes^L_A\cdots\otimes^L_A A',K)\to R\intHom_A(\mathcal M[S],K)
\]
is an isomorphism, where there are $n$ tensor factors on the left. But the left-hand side agrees with
\[
R\intHom_A(\mathcal M[S]\otimes^L_A A'\otimes^L_A\cdots\otimes^L_A A',K)
\]
with $n-1$ tensor factors on the left. Moreover, as $K\in (A,\mathcal M)\Mod$, we can replace
\[
\mathcal M[S]\otimes^L_A A'\otimes^L_A\cdots\otimes^L_A A'
\]
by $\otimes^L_A (A,\mathcal M)$. But this turns this into
\[
\mathcal M[S]\otimes^L_{(A,\mathcal M)} \mathcal M[\ast]\otimes^L_{(A,\mathcal M)}\cdots\otimes^L_{(A,\mathcal M)} \mathcal M[\ast] = \mathcal M[S],
\]
yielding the desired statement.
\end{proof}

In particular, $A=\mathbb Z((T))_{>r}$ with $\mathcal M[S] = \mathcal M(S,\mathbb Z((T))_{>r})$ defines an analytic ring.

Finally, let us prove Theorem~\ref{thm:keyreal}. Given $0<p'<p\leq 1$, pick $x=\tfrac 12$, $r=x^p$ and $r'=x^{p'}$. Then we can take $f_x=2-T^{-1}$, and Theorem~\ref{thm:Ainflike}~(2) says that
\[
\mathcal M(S,\mathbb Z((T))_{r'})=\mathcal M_{p'}(S)
\]
while (by passage to a colimit over all $r'>r$)
\[
\mathcal M(S,\mathbb Z((T))_{>r}) = \mathcal M_{<p}(S).
\]
We regard $\mathbb R$ everywhere as $\mathbb Z[T^{-1}]$-algebra via $T^{-1}\mapsto 2$. Given $V$ as in Theorem~\ref{thm:keyreal}, it is automatic that $V\in \mathrm{Liq}_r(\mathbb Z[T^{-1}])$. Using Theorem~\ref{thm:key}, this implies that
\[
R\intHom_{\mathbb Z[T^{-1}]}(\mathcal M(S,\mathbb Z((T))_{r'}),V) = R\intHom_{\mathbb Z[T^{-1}]}(\mathbb Z[T^{-1}][S],V)
\]
for all profinite sets $S$. The right-hand side agrees with $R\intHom_{\mathbb Z}(\mathbb Z[S],V)$, while the left-hand side agrees with $R\intHom_{\mathbb Z}(\mathcal M_{p'}(S),V)$ by taking the quotient by $f_x=2-T^{-1}$. Thus, this proves a version of Theorem~\ref{thm:keyreal} where we use the base ring $\mathbb Z$ in place of $\mathbb R$.

From this discussion, by taking the colimit over $p'>p$, one deduces that $\mathbb Z$ with the modules $\mathcal M_{<p}(S)$ defines an analytic ring structure. Using Proposition~\ref{prop:normalizeanalyticring} above, this then also implies that $\mathbb R$ with the same modules $\mathcal M_{<p}(S)$ defines an analytic ring, and that the forgetful functor $D(\mathrm{Liq}_r(\mathbb R))\to D(\Cond(\Ab))$ is fully faithful. In particular, we see that a liquid $\mathbb R$-vector space admits a unique $\mathbb R$-linear structure; more precisely, for all $V\in \mathrm{Liq}_p(\mathbb R)$, the map
\[
R\intHom_{\mathbb Z}(\mathbb R,V)\to V
\]
is an isomorphism.

Finally, the exact statement of Theorem~\ref{thm:keyreal} follows by observing that, using the last displayed isomorphism.
\[
R\intHom_{\mathbb R}(\mathcal M_{p'}(S),V)\cong R\intHom_{\mathbb R}(\mathcal M_{p'}(S),R\intHom_{\mathbb Z}(\mathbb R,V))\cong R\intHom_{\mathbb Z}(\mathcal M_{p'}(S),V),
\]
and the latter has been computed above.

\newpage

\section{Lecture VII: $\mathbb Z((T))_{>r}$ is a principal ideal domain}

In this lecture, we prove the following theorem.

\begin{theorem}[\cite{HarbaterConvergentArithmetic}]\label{thm:principalideal} For any real number $r$ with $0<r<1$, the ring
\[
\mathbb Z((T))_{>r} = \{\sum_{n\gg -\infty} a_n T^n\mid \exists r'>r, |a_n| (r')^n\to 0\}
\]
is a principal ideal domain. The nonzero prime ideals are the following:
\begin{enumerate}
\item For any nonzero complex number $x\in \mathbb C$ with $|x|\leq r$, the kernel of the map
\[
\mathbb Z((T))_{>r}\to \mathbb C: \sum a_n T^n\mapsto \sum a_n x^n,
\]
where this kernel depends only on (and determines) $x$ up to complex conjugation;
\item For any prime number $p$, the ideal $(p)$;
\item For any prime number $p$ and any topologically nilpotent unit $x\in \overline{\mathbb Q}_p$, the kernel of the map
\[
\mathbb Z((T))_{>r}\to \overline{\mathbb Q}_p: \sum a_n T^n\mapsto \sum a_n x^n,
\]
where this kernel depends only on (and determines) the Galois orbit of $x$.
\end{enumerate}
\end{theorem}

\begin{proof} It is clear that the ideals described are prime ideals. First, we check that all of them are principal, and that the respective quotient rings are already fields (so that these ideals are also maximal).

In case (1), assume first that $x$ is real, so $0<x\leq r$ or $-r\leq x<0$. The second case reduces to the first under the involution of $\mathbb Z((T))_{>r}$ taking $T$ to $-T$, so we assume for concreteness that $0<x\leq r$. We start by verifying the surjectivity of
\[
\mathbb Z((T))_{>r}\to \mathbb R: \sum a_n T^n\mapsto \sum a_n x^n.
\]
Fix some integer $N$ such that $x\geq \tfrac 1N$. Then we claim that any element $y\in \mathbb R_{\geq 0}$ is represented by a power series with coefficients in $[0,N-1]$. Note that when $x=\tfrac 1N$, this is precisely an $N$-adic expansion of $y$. For the proof, take the maximal $n$ such that $x^n\leq y$, and the maximal integer $a_n$ such that $a_n x^n\leq y$. Then necessarily $a_n\leq N-1$ (otherwise $x^{n+1}\leq y$). Now pass to $y-a_nx^n$ and repeat, noting that the sequence of $n$'s is strictly decreasing.

To find a generator, we first find a polynomial $g_n\in 1+T^n\mathbb R[T]$ such that inside the disc $\{0<|y|\leq r\}$, the only zero is $x$, with multiplicity $1$. We argue by induction on $n=1$; for $n=1$, we can take $g_1=1-x^{-1}T$. Now write $g_n=1+a_nT^n+\ldots$. Then, taking any integer $m>|a_n|$, we can take
\[
g_{n+1} = g_n(1-\tfrac{a_n}mT^n)^m.
\]

Now choose $n$ large enough so that $r^n<2(1-r)$ and let $g=g_n$. We can find some $h=1+c_nT^n+c_{n+1}T^{n+1}+\ldots\in 1+T^n\mathbb R[[T]]$ such that all $|c_i|\leq \tfrac 12$ and $f=gh\in 1+T^n\mathbb Z[[T]]$. Note that $h$ is invertible on $\{y\mid 0<|y|\leq r\}$, as
\[
|\sum_{i\geq n} c_iy^i|\leq \sum_{i\geq n} \frac 12 r^i = \frac 12 \frac{r^n}{1-r}<1.
\]

We claim that $f$ generates the kernel of evaluation at $x$. Indeed, if $g\in \mathbb Z((T))_{>r}$ vanishes at $x$, then $\frac gf\in \mathbb Z((T))$ (as $f\in 1+T\mathbb Z[[T]]\subset \mathbb Z[[T]]^\times$) and it still defines a holomorphic function on $\{0<|y|<r'\}$ for some $r'>r$; thus, $\frac gf\in \mathbb Z((T))_{>r}$.

Still in case (1), if now $x$ is complex, we argue similarly: First, surjectivity of
\[
\mathbb Z((T))_{>r}\to \mathbb C: \sum a_n T^n\mapsto \sum a_n x^n
\]
can be proved in a similar way, this time allowing coefficients in $[-N,N]$ for $N$ sufficiently large. To find a generator, argue as above by finding first a polynomial $g_n\in 1+T^n\mathbb R[T]$ whose only zeroes in $\{0<|y|\leq r\}$ are $x$ and its complex conjugate $\overline{x}$ with multiplicity $1$, starting with $g_1=(1-x^{-1}T)(1-\overline{x}^{-1}T)$. Multiplying by a power series $h$ as before then produces the desired generator $f=gh\in \mathbb Z((T))_{>r}$.

In case (2), it is clear that the ideals are principal. We claim that the map
\[
\mathbb Z((T))_{>r}/p\to \mathbb F_p((T))
\]
is an isomorphism. It is clearly injective. For surjectivity, we use that there is a set-theoretic section $\mathbb F_p\to \mathbb Z$ with image $\{0,1,\ldots,p-1\}$, so we can always lift to a power series with bounded coefficients $a_n$. But any power series with bounded coefficients lies in $\mathbb Z((T))_{>r}$.

It remains to handle case (3). Let $K\subset \overline{\mathbb Q}_p$ be the field generated by $x$, which is a finite extension of $\mathbb Q_p$; let $d$ be its degree. First, we check that the map
\[
f_x: \mathbb Z((T))_{>r}\to K: \sum a_n T^n\mapsto \sum a_n x^n
\]
is surjective. More precisely, if we denote by $\Lambda\subset \mathcal O_K$ the $\mathbb Z_p$-lattice generated by $x$, we claim that the map
\[
\mathbb Z[[T]]_{>r}\to \Lambda: \sum a_n T^n\mapsto \sum a_n x^n
\]
is surjective. For this, observe that if $x^d=up^m$ for some unit $u\in \mathcal O_K$ and integer $m>0$, then any element of $\Lambda/x^d$ can be written uniquely as a sum $\sum_{j=0}^{d-1} a_j x^j$ with all $a_j\in \{0,\ldots,p^m-1\}$. It follows that one can reach any element of $\Lambda$ uniquely by a power series $\sum a_n T^n$ where all $a_n\in \{0,\ldots,p^m-1\}$. Indeed, the projection to $\Lambda/x^d$ determines the first $d$ coefficients uniquely, and then we can divide by $T^d$ and induct.

Moreover, representing the element $p^m\in \Lambda$ in this way, we get an equation
\[
p^m = f_x(\sum_{n>0} a_n T^n)
\]
where all $a_n\in \{0,\ldots,p^m-1\}$. Consider the element
\[
g_x = p^m - \sum a_n T^n
\]
in the kernel of $f_x$. Note that $g_x$ may have additional archimedean zeroes. However, as we already have generators for them, we can write $g_x=g_{x,1} g_{x,2}$ where $g_{x,1}$ has no archimedean zeroes, while $g_{x,2}\in 1+T\mathbb Z[[T]]_{>r}$. We claim that $g_{x,1}$ generates the kernel of $f_x$. Indeed, it clearly lies in the kernel of $f_x$; we have to see that if $h\in \mathbb Z((T))_{>r}$ is any element in the kernel of $f_x$, then $\frac{h}{g_{x,1}}\in \mathbb Q((T))$ still lies in $\mathbb Z((T))_{>r}$. As $g_{x,1}\in p^m + T\mathbb Z[[T]]$, the inverse of $g_{x,1}$ has only powers of $p$ in the denominator, so $\frac{h}{g_{x,1}}\in \mathbb Z[\tfrac 1p]((T))$. But it also lies in $\mathbb Z_p[[T]]$ as $g_{x,1}$ is a generator of the kernel of $\mathbb Z_p[[T]]\to \mathcal O_K$. It remains to see that the convergence condition is satisfied, but for this we simply note that $g_{x,1}$ has no archimedean zeroes, so the quotient still defines a holomorphic function on $\{0<|y|<r'\}$ for some $r'>r$.

Finally, we can finish the proof that $\mathbb Z((T))_{>r}$ is a principal ideal domain. Take any nonzero element $f\in \mathbb Z((T))_{>r}$. We want to see that, up to a unit, it is a product of the generators of the principal ideals listed in the statement. Up to scaling by a unit, it is of the form
\[
f= a_0+\sum_{n>0} a_n T^n
\]
where $a_0>0$. We can now see that it lies in only finitely many of the principal ideals listed above: This is a standard result in complex analysis for the first points, and for the second and third type of points, only primes $p$ dividing $a_0$ are relevant. There are only finitely many such, and for any $p$ these primes are all contracted from $\mathbb Z_p((T))$, where the relevant finiteness holds. (We also see that it lies in any of these maximal ideals with finite multiplicity.) Dividing by the corresponding generators, we can assume that $f$ does not lie in any of the principal ideals listed in the statement. This implies that $a_0=1$ as otherwise there will be a zero at a $p$-adic place, by the similar result for $\mathbb Z_p((T))$. Then $f^{-1}\in 1+T\mathbb Z[[T]]$, and it still has convergence radius $>r$ as $f$ has no zeroes on the closed disc of radius $r$. Thus, $f$ is invertible, as desired.
\end{proof}

Now we can prove Theorem~\ref{thm:Ainflike} from the last lecture.

Given any $0<r'<r<1$, we consider the map $\mathbb Z((T))_r\to \mathbb R: T\mapsto r'$. Applying the previous lemma to $x=r'$, we get some $f=f_{r'}\in \mathbb Z((T))_{>r}$ that vanishes only at $r'$. We claim that the sequence
\[
0\to \mathbb Z((T))_r\xrightarrow{f_{r'}} \mathbb Z((T))_r\to \mathbb R\to 0
\]
is exact. Surjectivity was proved above (even $\mathbb Z((T))_{>r}$ surjects). On the other hand, if $g\in \mathbb Z((T))_r$ vanishes at $x$, then $gf_{r'}^{-1}\in \mathbb Z((T))$ and it follows from a standard complex analysis consideration that it still lies in $\mathbb Z((T))_r$.

We want to see that the sequence is still exact as condensed abelian groups, and we want to identify the spaces of measures. For this, we prove the following quantitative version:

\begin{proposition} There are constants $C_1,\ldots,C_4$ with the following properties:
\begin{enumerate}
\item If $g\in \mathbb Z((T))_{r,\leq c}$ then $f_{r'}g\in \mathbb Z((T))_{r,\leq C_1c}$.
\item Conversely, if $g\in \mathbb Z((T))_r$ such that $f_{r'}g\in \mathbb Z((T))_{r,\leq c}$, then $g\in \mathbb Z((T))_{r,\leq C_2c}$.
\item Let $0<p<1$ such that $(r')^p=r$. If $g\in \mathbb Z((T))_{r,\leq c}$, then $g(x)\in \mathbb R_{\ell^p\leq C_3c}$.
\item Conversely, for any $z\in \mathbb R_{\ell^p\leq c}$, there is some $g\in \mathbb Z((T))_{r,\leq C_4c}$ such that $g(x)=z$.
\end{enumerate}
\end{proposition}

The proposition easily implies exactness as condensed abelian groups (think in terms of ind-(compact Hausdorff) sets, and use that for maps of compact Hausdorff spaces, surjectivity can be checked on points).

\begin{proof} Take $C_1$ so that $f_{r'}\in \mathbb Z((T))_{r,\leq C_1}$, then the claim about $C_1$ is clear. For the claim about $C_2$, note that the norm can be expressed in terms of the restriction to a function on $\{y\mid |r''|\leq |y|\leq |r'''|\}$ for $r'<r''<r<r'''$. On that strip, $f_{r'}$ is invertible, so picking a bound of the norm for the inverse gives the claim.

For $C_3$, let $g=\sum a_nT^n$ with $\sum |a_n|r^n\leq c$. Then
\[
|g(x)|^p\leq \sum |a_n|^p (r')^{np} = \sum |a_n|^p r^n\leq \sum |a_n|r^n\leq c,
\]
so we can in fact take $C_3=1$. Here, we used critically that all $a_n$ are integers, which implies that $|a_n|^p\leq |a_n|$: This is true whenever $a_n=0$ or $|a_n|\geq 1$ (when $p<1$).

For $C_4$, take any $z\in \mathbb R$ with $|z|^p\leq c$ and assume $z\neq 0$. Take $n$ minimal so that $(r')^n\leq |z|$, and write $z=a_n(r')^n + z'$ where $a_n\in \mathbb Z$ and $|z'|<(r')^n$. Note that $|a_n|\leq (r')^{-1}$. Continuing, we can write $z=g(x)$ for some $g=a_nT^n+\ldots\in \mathbb Z((T))$ with $|a_i|\leq (r')^{-1}$ for all $i$. Then
\[
\sum |a_i|r^i\leq (r')^{-1} \frac{r^n}{1-r}\leq \frac{1}{r'(1-r)} |z|^p,
\]
so we can take $C_4=\frac{1}{r'(1-r)}$.
\end{proof}

From here, by observing that these bounds immediately generalize to finite free modules, one easily deduces Theorem~\ref{thm:Ainflike}~(2). For part (3), we need a corresponding version. Note first that
\[
\mathbb Z((T))_r/(f_{r'})^m\cong \mathbb R[X]/X^m
\]
where the isomorphism sends $T$ to $[r']=(r')^{1+X}=r'+r'\log r' X + \ldots$. In other words, we have a short exact sequence
\[
0\to \mathbb Z((T))_r\xrightarrow{f_{r'}^m} \mathbb Z((T))_r\xrightarrow{\theta_m} \mathbb R[X]/X^m\to 0.
\]

\begin{proposition} There are constants $C_1,\ldots,C_4$ (depending on $m$) with the following properties:
\begin{enumerate}
\item If $g\in \mathbb Z((T))_{r,\leq c}$ then $f_{r'}^mg\in \mathbb Z((T))_{r,\leq C_1c}$.
\item Conversely, if $g\in \mathbb Z((T))_r$ such that $f_{r'}^mg\in \mathbb Z((T))_{r,\leq c}$, then $g\in \mathbb Z((T))_{r,\leq C_2c}$.
\item If $g\in \mathbb Z((T))_{r,\leq c}$, then $\theta_m(g)\in (\mathbb R[X]/X^m)_{\ell^p\leq C_3c}$.
\item Conversely, for any $z\in (\mathbb R[X]/X^m)_{\ell^p\leq c}$, there is some $g\in \mathbb Z((T))_{r,\leq C_4c}$ such that $\theta_m(g)=z$.
\end{enumerate}
\end{proposition}

\begin{proof} Parts (1) and (2) follow inductively from parts (1) and (2) of the previous proposition. It remains to handle part (3): Part (4) then follows by successive approximation (using parts (3) and (4) of the previous proposition, and the other claims already established).

Thus let $g=\sum a_n T^n\in \mathbb Z((T))_r$ with $\sum |a_n|r^n\leq c$. Then
\[
\theta_m(g)=\sum a_n [r']^n\in \mathbb R[X]/X^m.
\]
We need to bound this. Recall from Proposition~\ref{prop:stableimage} that there is some constant $C$ such that for all maps $f: S\to S'$ of finite sets, one has
\[
f((\mathbb R[X]/X^m)[S]_{\ell^p\leq c})\subset (\mathbb R[X]/X^m)[S']_{\ell^p\leq Cc}.
\]
We claim that we can take $C_3=C$; equivalently,
\[
\sum a_n [r']^n\in (\mathbb R[X]/X^m)_{\ell^p\leq Cc}.
\]
To check this, we can assume that the sum is finite (as the target is closed). Let $S$ be a finite set of cardinality $\sum |a_n|$. Then we can define an element $\tilde{z}\in (\mathbb R[X]/X^m)[S]$ whose coefficients are of the form $\pm [r']^n=[\pm r']^n$, such that the image of $\tilde{z}$ in $\mathbb R[X]/X^m$ (the sum of the coefficients) is $\sum a_n [r']^n$. Thus, it suffices to see that $\tilde{z}\in (\mathbb R[X]/X^m)[S]_{\ell^p\leq c}$. But this is clear: One has to bound the sum of $(r')^{np}=r^n$, where there are $|a_n|$ occurences of $r^n$, giving in total at most $\sum |a_n|r^n\leq c$.
\end{proof}

The lecturer finds the way in which the varying $\ell^p$-norms, as well as the subtle $B_{\mathrm{dR}}^+$-type deformations, arise from the very simple-minded spaces of measures for $\mathbb Z((T))_r$ quite striking. The discreteness of $\mathbb Z$ is critical for this behaviour: In some sense, the real $B_{\mathrm{dR}}^+$ arises via some kind of discretization.

\newpage

\section{Lecture VIII: Reduction to ``Banach spaces''}

Fix some $0<r<1$. Recall that we want to prove Theorem~\ref{thm:key}:

\begin{theorem}\label{thm:keyrestated} Let $K$ be a condensed $\mathbb Z[T^{-1}]$-module that is the kernel of some map
\[
f: \bigoplus_{i\in I} \mathcal M(S_i,\mathbb Z((T))_{>r})\to \bigoplus_{j\in J} \mathcal M(S_j',\mathbb Z((T))_{>r})
\]
where all $S_i$ and $S_j'$ are extremally disconnected. Then for all $1>r'>r$ and all profinite sets $S$, the map
\[
R\intHom_{\mathbb Z[T^{-1}]}(\mathcal M(S,\mathbb Z((T))_{r'}),K)\to R\intHom_{\mathbb Z[T^{-1}]}(\mathbb Z[T^{-1}][S],K)
\]
in $D(\Cond(\Ab))$ is an isomorphism.
\end{theorem}

Today, we make some preliminary reductions, and in particular reduce to the case where $K$ is a suitable analogue of a Banach space. First, we note the following simple proposition.

\begin{proposition}\label{prop:liquidabelianlevel} Let $S$ and $S^\prime$ be two profinite sets and let $f: S\to \mathcal M(S',\mathbb Z((T))_r)$ be a map. Then for any $r'\geq r$ less than $1$, there is a unique $\mathbb Z[T^{-1}]$-linear map
\[
\mathcal M(S,\mathbb Z((T))_{r'})\to \mathcal M(S',\mathbb Z((T))_r)
\]
extending the given map $f$.
\end{proposition}

\begin{proof} First, we prove uniqueness. Thus assume that some map $g: \mathcal M(S,\mathbb Z((T))_{r'})\to \mathcal M(S',\mathbb Z((T))_r)$ has trivial restriction to $S$. By $\mathbb Z[T^{-1}]$-linearity, it vanishes on $\mathbb Z[T^{\pm 1}][S]$, but this is dense in $\mathcal M(S,\mathbb Z((T))_{r'})$ (and everything is quasiseparated). Thus, the map $g$ vanishes.

For existence, we may assume that $r'=r$, and by rescaling that $f$ maps $S$ into $\mathcal M(S',\mathbb Z((T))_r)_{\leq 1}$. We want to show that there is a map of profinite sets
\[
\mathcal M(S,\mathbb Z((T))_r)_{\leq 1}\to \mathcal M(S',\mathbb Z((T))_r)_{\leq 1}
\]
extending the given map on the dense subset $\mathbb Z[T^{\pm 1}][S]$. Writing $S'$ as an inverse limit of finite sets, this reduces to the case that $S'$ is finite. Then
\[
\mathcal M(S',\mathbb Z((T))_r)_{\leq 1}\subset \prod_{n\geq 0,s'\in S'} \mathbb Z\cdot T^n [s'] = \mathbb Z[[T]][S'].
\]
Moreover, it is defined as an inverse limit of finite subsets of $(\mathbb Z[T]/T^m)[S']$, so we can work modulo $T^m$ for some $m$. Then the map $S\to \mathcal M(S',\mathbb Z((T))_r)_{\leq 1}\to (\mathbb Z[T]/T^m)[S']$ factors through some finite quotient $S\to S_m$, and we can also reduce to the case that $S$ is finite. In that case, the desired map
\[
\mathcal M(S,\mathbb Z((T))_r)\to \mathcal M(S',\mathbb Z((T))_r)
\]
clearly exists, and is given by $\sum_{s\in S} a_s [s]\mapsto \sum_{s\in S} a_s f(s)$. One verifies that this indeed maps $\mathcal M(S,\mathbb Z((T))_r)_{\leq 1}$ into $\mathcal M(S',\mathbb Z((T))_r)_{\leq 1}$.
\end{proof}

The following proposition is easy to prove by using that we choose exactly $\mathbb Z[T^{-1}]$ as the base ring. (It is also true that $\mathcal M_p(S)$ is a pseudocoherent $\mathbb R$-module, but this is somewhat more tricky to prove.)

\begin{proposition}\label{prop:liquidpseudocoh} For any profinite set $S$, the condensed $\mathbb Z[T^{-1}]$-module $\mathcal M(S,\mathbb Z((T))_r)$ is pseudocoherent.
\end{proposition}

\begin{proof} Let $M=\mathcal M(S,\mathbb Z((T))_r)$. We have the Breen--Deligne resolution
\[
\ldots\to \bigoplus_{j=1}^{n_i} \mathbb Z[M^{a_{ij}}]\to\ldots\to \mathbb Z[M^2]\to \mathbb Z[M]\to M\to 0.
\]
Here $n_i$ and all $a_{ij}$ are nonnegative integers. Note that $M$ has an endomorphism (multiplication by $T^{-1}$), and by functoriality in $M$, the whole sequence is a sequence of condensed $\mathbb Z[T^{-1}]$-modules. Thus, it is enough to see that each $\mathbb Z[M^a]$ is pseudocoherent as $\mathbb Z[T^{-1}]$-module. Replacing $S$ by $a$ copies of itself, we can assume that $a=1$. Now $M$ has a profinite subset $M_{\leq 1}\subset M$ and $M=\bigcup_{n\geq 0} T^{-n} M_{\leq 1}$. In fact, we claim that the sequence
\[
0\to \mathbb Z[M_{\leq r}][T^{-1}]\xrightarrow{T^{-1}-[T^{-1}]}\mathbb Z[M_{\leq 1}][T^{-1}]\to \mathbb Z[M]\to 0
\]
is an exact sequence of $\mathbb Z[T^{-1}]$-modules. This will give the desired result as the other two terms are pseudocoherent $\mathbb Z[T^{-1}]$-modules. But the sequence is the filtered colimit of the sequences
\[
0\to \bigoplus_{i=0}^{n-1} \mathbb Z[M_{\leq r}]\cdot T^{-i}\xrightarrow{T^{-1}-[T^{-1}]}\bigoplus_{i=0}^{n} \mathbb Z[M_{\leq 1}]\cdot T^{-i}\to \mathbb Z[T^{-n} M_{\leq 1}]\to 0
\]
which are exact by induction on $n$: For $n=0$, the first term is zero and the second map is an equality, and for $n>0$ the quotient by the sequence for $n-1$ (multiplied by $T^{-1}$) is an isomorphism $\mathbb Z[M_{\leq r}]\cong \mathbb Z[M_{\leq 1}]$ between the terms of degree $0$ (given by multiplication by $T^{-1}: M_{\leq r}\cong M_{\leq 1}$).
\end{proof}

It follows that $R\intHom_{\mathbb Z[T^{-1}]}(\mathcal M(S,\mathbb Z((T))_{r'}),K)$ commutes with filtered colimits in $K$. In particular, we can assume that $K$ is the kernel of some map
\[
\mathcal M(S_1,\mathbb Z((T))_{>r})\to \bigoplus_{j\in J} \mathcal M(S_j',\mathbb Z((T))_{>r}).
\]
Note that by Proposition~\ref{prop:liquidabelianlevel} and Proposition~\ref{prop:liquidpseudocoh}, for any $r'>r$ one has
\[
\Hom(\mathcal M(S_1,\mathbb Z((T))_{r'}),\bigoplus_{j\in J} \mathcal M(S_j',\mathbb Z((T))_{>r})) = \bigoplus_{j\in J} \mathcal M(S_j',\mathbb Z((T))_{>r})(S_1).
\]
Passing to the limit $r'\to r$, this shows that also
\[
\Hom(\mathcal M(S_1,\mathbb Z((T))_{>r}),\bigoplus_{j\in J} \mathcal M(S_j',\mathbb Z((T))_{>r})) = \bigoplus_{j\in J} \mathcal M(S_j',\mathbb Z((T))_{>r})(S_1).
\]
In particular, the given map factors over a finite direct sum, and we may assume that $K$ is the kernel of some map
\[
\mathcal M(S_1,\mathbb Z((T))_{>r})\to \mathcal M(S_2,\mathbb Z((T))_{>r}),
\]
and we recall that any such map is uniquely induced by a map $S_1\to \mathcal M(S_2,\mathbb Z((T))_{>r})$. In fact, by a triangle (and noting that $\mathcal M(S_1,\mathbb Z((T))_{>r})$ itself is a special case of such an image, when taking the identity map), it is enough to prove the similar assertion for the image of the map
\[
\mathcal M(S_1,\mathbb Z((T))_{>r})\to \mathcal M(S_2,\mathbb Z((T))_{>r}).
\]
In other words, we are reduced to the following assertion. (We could assume that $S_1$ and $S_2$ are extremally disconnected instead of merely profinite, but this will be of no use.)

\begin{theorem} Let $S_1$ and $S_2$ be profinite sets and let $f: \mathcal M(S_1,\mathbb Z((T))_{>r})\to \mathcal M(S_2,\mathbb Z((T))_{>r})$ be a map of condensed $\mathbb Z[T^{-1}]$-modules; these are in bijection with maps $f|_{S_1}: S_1\to \mathbb \mathcal M(S_2,\mathbb Z((T))_{>r})$. Let $M$ be the image of $f$. Then for any $r'>r$, the map
\[
R\intHom_{\mathbb Z[T^{-1}]}(\mathcal M(S,\mathbb Z((T))_{r'}),M)\to R\intHom_{\mathbb Z[T^{-1}]}(\mathbb Z[T^{-1}][S],M)
\]
in $D(\Cond(\Ab))$ is an isomorphism.
\end{theorem}

We wish to reduce further. The datum is the map $S_1\to \mathcal M(S_2,\mathbb Z((T))_{>r})$. This factors over $\mathcal M(S_2,\mathbb Z((T))_{>r_0})$ for some $r_0>r$, and fixing such a factorization, we get modules $M_{r''}$ for all $r_0\geq r''>r$ as the image of $\mathcal M(S_1,\mathbb Z((T))_{r''})\to \mathcal M(S_2,\mathbb Z((T))_{r''})$, and $M$ is the filtered colimit of the $M_{r''}$. Note that $M_{r''}$ can be endowed with subspaces $M_{r'',\leq c}$ by taking the image of $\mathcal M(S_1,\mathbb Z((T))_{r''})_{\leq c}$. These are the inverse limits of the images $M_{r'',i,\leq c}$ in $\mathcal M(S_{2,i},\mathbb Z((T))_{r''})$, where we write $S_2$ as an inverse limit of finite sets $S_{2,i}$. We need the following proposition.

One can actually identify the module $M_{r'',i}$ that appears, by applying the following proposition with $r''$ in place of $r$.

\begin{proposition}\label{prop:imagefree} Let $S_1$ be a profinite set and let $S_1\to \mathbb Z((T))_{>r}^n$ be a map for some integer $n$. Then the image $M$ of
\[
\mathcal M(S_1,\mathbb Z((T))_r)\to \mathbb Z((T))_r^n
\]
is isomorphic (as condensed $\mathbb Z[T^{-1}]$-module) to $\mathbb Z((T))_r^{n'}$ for some $n'\leq n$.
\end{proposition}

\begin{proof} By Theorem~\ref{thm:principalideal}, the image of
\[
\mathcal M(S_1,\mathbb Z((T))_{>r})\to \mathbb Z((T))_{>r}^n
\]
on underlying modules is isomorphic to $\mathbb Z((T))_{>r}^{n'}$ for some $n'\leq n$. Let $M'=\mathbb Z((T))_{>r}^{n'}$ as a condensed $\mathbb Z[T^{-1}]$-module, which comes equipped with a map to $\mathbb Z((T))_{>r}^n$. This map is injective as both are quasiseparated as condensed sets and the underlying map is injective. Moreover, the quotient of condensed $\mathbb Z[T^{-1}]$-modules
\[
\mathbb Z((T))_{>r}^n/M'
\]
is still quasiseparated, by the structure of finitely generated $\mathbb Z((T))_{>r}$-modules (using Theorem~\ref{thm:principalideal}). It follows that the map $S_1\to \mathbb Z((T))_{>r}^n$ factors over $M'$, and replacing $\mathbb Z((T))_{>r}^n$ by $M'$, we may assume that $\mathcal M(S_1,\mathbb Z((T))_{>r})\to \mathbb Z((T))_{>r}^n$ is surjective (on underlying modules). In particular, the image contains a basis, which implies that also
\[
\mathcal M(S,\mathbb Z((T))_r)\to \mathbb Z((T))_r^n
\]
is surjective as condensed modules, giving the desired result.
\end{proof}

The next proposition expresses $M_{r''}$ in terms of the $M_{r'',i}$.

\begin{proposition} For any $r''$, the complex
\[
0\to M_{r''}\to \bigcup_{c>0}\prod_i M_{r'',i,\leq c}\to \bigcup_{c>0} \prod_{i_0\to i_1} M_{r'',i_1,\leq c}\to \ldots\to \bigcup_{c>0} \prod_{i_0\to \ldots\to i_m} M_{r'',i_m,\leq c}\to \ldots
\]
of condensed $\mathbb Z[T^{-1}]$-modules is exact. More precisely, taking $E$-valued points for some extremally disconnected set $E$, it is exact, and if
\[
g\in \ker(\prod_{i_0\to\ldots\to i_m} M_{r'',i_m,\leq c}(E)\to \prod_{i_0\to\ldots\to i_{m+1}} M_{r'',i_{m+1}}(E))
\]
then there exists some $h\in \prod_{i_0\to\ldots\to i_{m-1}} M_{r'',i_{m-1},\leq c}(E)$ with $d(h)=g$.
\end{proposition}

We refer to the appendix for the construction of the complex.

\begin{proof} The claim is equivalent to the surjectivity of some map of profinite sets, more precisely that of
\[\begin{aligned}
\prod_{i_0\to\ldots\to i_{m-1}} M_{r'',i_{m-1},\leq c}&\times_{\prod_{i_0\to\ldots\to i_m} M_{r'',i_m}} \prod_{i_0\to\ldots\to i_m} M_{r'',i_m,\leq c}\\
&\to \ker(\prod_{i_0\to\ldots\to i_m} M_{r'',i_m,\leq c}\to \prod_{i_0\to\ldots\to i_{m+1}} M_{r'',i_{m+1}}).
\end{aligned}\]
We can write the index category $\mathcal I$ of $i$'s as a filtered union of finite categories $\mathcal I_j$ that admit an initial object. The corresponding map is a cofiltered limit of similar maps for $\mathcal I_j$ in place of $\mathcal I$. Therefore it suffices to prove the similar assertion for each $\mathcal I_j$. In that case, the presence of an initial object gives a contracting homotopy, whose explicit description as given in the appendix to this lecture shows that it preserves the $\leq c$-subspace.
\end{proof}

In particular, it is enough to prove the following statement.

\begin{theorem} Fix radii $1>r'>r_0>r>0$. Let $S_1$ be a profinite set and let $f_i: S_1\to \mathbb Z((T))_{>r_0}^{n_i}$ be maps for some index set $i\in I$. For $r_0\geq r''>r$, let $M_{r'',i}\subset \mathbb Z((T))_{r''}^{n_i}$ be the image of the induced maps $\mathcal M(S_1,\mathbb Z((T))_{r''})\to \mathbb Z((T))_{r''}^{n_i}$, with the subspace $M_{r'',i,\leq c}\subset M_{r'',i}$ defined as the image of $\mathcal M(S_1,\mathbb Z((T))_{r''})_{\leq c}$.

Then, letting
\[
M=\varinjlim_{r''>r} \bigcup_{c>0} \prod_i M_{r'',i,\leq c},
\]
the map
\[
R\intHom_{\mathbb Z[T^{-1}]}(\mathcal M(S,\mathbb Z((T))_{r'}),M)\to R\intHom_{\mathbb Z[T^{-1}]}(\mathbb Z[T^{-1}][S],M)
\]
is an isomorphism in $D(\Cond(\Ab))$ for all profinite sets $S$.
\end{theorem}

At this point, we will make a change from Smith to Banach spaces: The $M_{r'',i}$ are some kind of Smith spaces, but their norm allows one to define canonical Banach spaces inside them.

\begin{definition} Let $N$ be a condensed abelian group written as an increasing union of compact Hausdorff subsets $N_{\leq c}$ for $c\geq 0$, satisfying the following conditions:
\begin{enumerate}
\item One has $N_{\leq 0} = 0$;
\item Each $N_{\leq c}$ is symmetric, i.e.~$-N_{\leq c}=N_{\leq c}$;
\item It is exhaustive: $N=\bigcup_{c>0} N_{\leq c}$;
\item It satisfies the triangle inequality: $N_{\leq c}+N_{\leq c'}\subset N_{\leq c+c'}$;
\item It is continuous: $N_{\leq c}=\bigcap_{c'>c} N_{\leq c'}$.
\end{enumerate}
In that situation, we let $N^B$ be the condensed abelian group which takes any profinite set $S$ to the completion of the normed abelian group of locally constant maps from $S$ to $N$.
\end{definition}

In other words, as discussed in the appendix, $N^B$ is the condensed abelian group associated to the normed abstract abelian group $N(\ast)$.

\begin{proposition}\label{prop:passingtobanach} There is a unique map $N^B\to N$ of condensed sets extending the identity on underlying abelian groups. It is an injective map of condensed abelian groups.
\end{proposition}

\begin{proof} Given a map $f: S\to N^B$, choose a sequence of locally constant maps $f_0,f_1,\ldots: S\to N$ with limit $f$ (by definition $N^B(S)$ is the space of Cauchy sequences modulo the space of null sequences). This defines a map $S\times \mathbb N\to N$, and it is easy to see that it takes image in $N_{\leq c}$ for some $c$. Let $\Gamma\subset (S\times \mathbb N)\times N_{\leq c}$ be its graph, and let $\overline{\Gamma}\subset (S\times (\mathbb N\cup \{\infty\}))\times N_{\leq c}$ be the closure. We claim that $\overline{\Gamma}\to S\times (\mathbb N\cup\{\infty\})$ is an isomorphism. As it is a map of compact Hausdorff spaces, it is enough to check bijectivity, so pick some $s\in S$. We need to see that there is a unique preimage of $(s,\infty)$. Note that existence is clear, as $\overline{\Gamma}$ has dense image, and thus full image. On the other hand, uniqueness is a simple consequence of the Cauchy property and the separatedness $0=\bigcap_{c>0} N_{\leq c}$ (which follows from condition (1) and (5)).

In particular, restricting to $\infty$, we get a map $S\to N_{\leq c}$, thus defining a map $S\to N$. It is clear that this is functorial, additive, and independent of the choice of the $f_i$ (as nullsequences will converge to zero), and defines a map of condensed sets $N^B\to N$ (and is the only possible choice). The rest of the assertions are also easy to verify.
\end{proof}

In our situation, the following proposition ensures that replacing $M_{r'',i}$ by $M_{r'',i}^B$ is harmless.

\begin{proposition} For any $r_0\geq r''>r'''>r$, the map $M_{r'',i}\to M_{r''',i}$ factors over $M_{r''',i}^B$.
\end{proposition}

\begin{proof} By Proposition~\ref{prop:imagefree}, one has $M_{r'',i}\cong \mathbb Z((T))_{r''}^{n_i'}$ for some $n_i'$, and then
\[
M_{r''',i}=M_{r'',i}\otimes_{\mathbb Z((T))_{r''}} \mathbb Z((T))_{r'''}\cong \mathbb Z((T))_{r'''}^{n_i'}.
\]
Thus, it suffices to see that $\mathbb Z((T))_{r''}\to \mathbb Z((T))_{r'''}$ factors over $\mathbb Z((T))_{r'''}^B$. It suffices to check this on $\mathbb Z((T))_{r'',\leq 1}$, and then it follows from the observation that $\mathbb Z((T))_{r'',\leq 1}\cap T^m\mathbb Z[[T]]$ maps into $\mathbb Z((T))_{r''',\leq (\frac{r'''}{r''})^m}$, so writing any $S$-valued section of $\mathbb Z((T))_{r'',\leq 1}$ as $\sum_{n\geq 0} a_n T^n$ with $a_n\in C(S,\mathbb Z)$, this sum is actually convergent in $\mathbb Z((T))_{r'''}^B$.
\end{proof}

This implies that
\[
\varinjlim_{r''>r} \bigcup_{c>0} \prod_i M_{r'',i,\leq c} = \varinjlim_{r''>r} \bigcup_{c>0} \prod_i (M_{r'',i,\leq c}^B).
\]
After this change, we can actually prove a statement for an individual $r''$ (which gives the previous statement by passage to the filtered colimit over all $r''>r$, using Proposition~\ref{prop:liquidpseudocoh} again). We are reduced to the following statement (where the $r$ in the present statement is any of the $r''$ in the previous statement).

\begin{theorem} Fix radii $1>r'>r>0$. Let $S_1$ be a profinite set and let $f_i: S_1\to \mathbb Z((T))_{>r}^{n_i}$ be maps for some index set $i\in I$. Let $M_i\subset \mathbb Z((T))_r^{n_i}$ be the image of the induced maps $\mathcal M(S_1,\mathbb Z((T))_r)\to \mathbb Z((T))_r^{n_i}$, with the subspace $M_{i,\leq c}\subset M_i$ defined as the image of $\mathcal M(S_1,\mathbb Z((T))_r)_{\leq c}$.

Then, letting
\[
M=\bigcup_{c>0} \prod_i M_{i,\leq c}^B,
\]
the map
\[
R\intHom_{\mathbb Z[T^{-1}]}(\mathcal M(S,\mathbb Z((T))_{r'}),M)\to R\intHom_{\mathbb Z[T^{-1}]}(\mathbb Z[T^{-1}][S],M)
\]
is an isomorphism in $D(\Cond(\Ab))$ for all profinite sets $S$.
\end{theorem}

Note that we are doing something slightly funny here as we are using the procedure $N\mapsto N^B$ inside the product over $i$: While $\bigcup_{c>0} \prod_i M_{i,\leq c}$ is of Smith-type (countable union of compact Hausdorff), the present construction is a strange mixture.

We will now get rid of the product over $i$ by formulating a quantitative version of the statement for a single $i$. To show that the internal Hom agrees, we have to see that for any profinite set $S'$, one has
\[
R\Hom_{\mathbb Z[T^{-1}]}(\mathcal M(S,\mathbb Z((T))_{r'})[S']/\mathbb Z[T^{-1}][S\times S'],M) = 0.
\]
The source is a pseudocoherent $\mathbb Z[T^{-1}]$-module, so we may pick a projective resolution $P_\bullet = P_\bullet(r',S,S')$ of $\mathcal M(S,\mathbb Z((T))_{r'})[S']/\mathbb Z[T^{-1}][S\times S']$, where each $P_i = \mathbb Z[T^{-1}][E_i]$ is the free condensed $\mathbb Z[T^{-1}]$-module on some extremally disconnected set $E_i$. Then we have to see that the corresponding complex
\[
0\to M(E_0)\to M(E_1)\to \ldots
\]
is acyclic.

Now we observe that $M$ is naturally the union of the subsets $M_{\leq c}=\prod_i M_{i,\leq c}^B$ over all $c>0$. Using this observation, it is enough to prove the following quantitative result.

\begin{theorem}\label{thm:constants} Fix profinite sets $S$ and $S'$ and radii $1>r'>r>0$, as well as a projective resolution $P_\bullet$ of $\mathcal M(S,\mathbb Z((T))_{r'})[S']/\mathbb Z[T^{-1}][S\times S']$, where each $P_i = \mathbb Z[T^{-1}][E_i]$ is the free condensed $\mathbb Z[T^{-1}]$-module on some extremally disconnected set $E_i$.

Then for any $m$ there is some $C_m>0$ with the following property. For any profinite set $S_1$ with a map $f: S_1\to \mathbb Z((T))_{>r}^n$, letting $M$ be the image of the induced map $\mathcal M(S_1,\mathbb Z((T))_r)\to \mathbb Z((T))_r^n$ with its subspaces $M_{\leq c}$, the complex
\[
0\to M^B(E_0)\to M^B(E_1)\to \ldots
\]
computing $R\Hom(\mathcal M(S,\mathbb Z((T))_{r'})[S']/\mathbb Z[T^{-1}][S\times S'],M^B)$ is exact, and if $g\in \ker(M^B(E_m)\to M^B(E_{m+1}))$ with $||g||\leq c$, then there is some $h\in M^B(E_{m-1})$ with $||h||\leq C_mc$ such that $d(h)=g$.
\end{theorem}

Now $M^B$ with its norm is an example of the following structure:

\begin{definition} An $r$-normed $\mathbb Z[T^{\pm 1}]$-module is a normed $\mathbb Z[T^{\pm 1}]$-module $V$ satisfying $||T v||=r||v||$ for all $v\in V$.
\end{definition}

To an $r$-normed $\mathbb Z[T^{\pm 1}]$-module $V$, one can associate a condensed $\mathbb Z[T^{\pm 1}]$-module $\widehat{V}$, as in the appendix. One main reason that we reduced to this ``Banach space'' setting is that one can use resolutions by profinite sets, instead of extremally disconnected sets, by Proposition~\ref{prop:normedcompletion}. We end this lecture by reducing to the following statement.

\begin{theorem}\label{thm:banachcomplete} Fix radii $1>r'>r>0$. Then for all $r$-normed $\mathbb Z[T^{\pm 1}]$-modules $V$ and all profinite sets $S$, the map
\[
R\Hom_{\mathbb Z[T^{-1}]}(\mathcal M(S,\mathbb Z((T))_{r'}),\widehat{V})\to \widehat{V}(S)
\]
is a quasi-isomorphism.
\end{theorem}

\begin{proof}[Theorem~\ref{thm:banachcomplete} implies Theorem~\ref{thm:constants}] Fix the profinite set $S$ and take a projective resolution $P_\bullet$ of $\mathcal M(S,\mathbb Z((T))_{r'})/\mathbb Z[T^{-1}][S]$, so that each $P_i=\mathbb Z[T^{-1}][E_i]$ is the free condensed $\mathbb Z[T^{-1}]$-module on an extremally disconnected set $E_i$.

As $R\Gamma(S,\widehat{V})=\widehat{V}(S)$ is concentrated in degree $0$, Theorem~\ref{thm:banachcomplete} implies that the complex
\[
0\to \widehat{V}(E_0)\to \widehat{V}(E_1)\to \ldots
\]
computing $R\Hom_{\mathbb Z[T^{-1}]}(\mathcal M(S,\mathbb Z((T))_{r'})/\mathbb Z[T^{-1}][S],V)$, is acyclic. We claim that this implies that for all $m\geq 0$, there is some constant $C_m$ (independent of $V$) such that for all $g\in \ker(\widehat{V}(E_m)\to \widehat{V}(E_{m+1}))$ with $||g||\leq c$, there is some $h\in \widehat{V}(E_{m-1})$ with $||h||\leq C_mc$ such that $d(h)=g$.

Indeed, assume the contrary. Then we can find a sequence $V_1,V_2,\ldots$ of $r$-normed $\mathbb Z[T^{\pm 1}]$-modules, equipped with norms $||\cdot||_i$ on $V_i$, together with elements $g_i\in \ker(\widehat{V}_i(E_m)\to \widehat{V}_i(E_{m+1}))$ with $||g_i||\leq c$ for some fixed $c$, such that for each $i=1,2,\ldots$ there is no $h\in V_i(E_{m-1})$ with $||h_i||\leq r^{-2i}c$ such that $d(h_i)=g_i$.

Let $V$ be the direct sum $V_1\oplus V_2\oplus \ldots$ equipped with the supremum norm; this is again an $r$-normed $\mathbb Z[T^{\pm 1}]$-module. Moreover, $g=(Tg_1,T^2g_2,\ldots)$ defines an element of $\widehat{V}(E_m)$. Using the exactness of the complex for $V$, we get some $h\in V$ with $d(h)=g$. Then there is some $i$ such that $||h||\leq r^{-i}c$, and then the image $T^{-i}h_i\in V_i$ of $h$ satisfies $||h_i||\leq r^{-2i}c$ and $d(h_i)=g_i$, which contradicts our assumption.

Now, we also see that for all profinite sets $S'$, the complex
\[
0\to \widehat{V}(E_0\times S')\to \widehat{V}(E_1\times S')\to \ldots
\]
is exact, and for all $m\geq 0$, there is some constant $C_m$ (independent of $V$) such that for all $g\in \ker(\widehat{V}(E_m\times S')\to \widehat{V}(E_{m+1}\times S'))$ with $||g||\leq c$, there is some $h\in \widehat{V}(E_{m-1}\times S')$ with $||h||\leq C_mc$ such that $d(h)=g$. Indeed, this complex is obtained by completing the complex of locally constant maps from $S'$ into the previous one, which gives the desired bounds by Proposition~\ref{prop:completeexact}.

But this complex computes $R\Hom_{\mathbb Z[T^{-1}]}(\mathcal M(S,\mathbb Z((T))_{r'})[S']/\mathbb Z[T^{-1}][S\times S'],V)$, via the resolution $P_\bullet[S']$. This is not quite a resolution by extremally disconnected sets, but one can refine it by one (resolving each term), and use Proposition~\ref{prop:normedcompletion}. Applying this to $V=M^B$ with its norm, we get the result.
\end{proof}

\begin{remark} Reductions very similar to the ones in this lecture could be done directly in the setting of liquid $\mathbb R$-vector spaces, reducing Theorem~\ref{thm:keyreal} to the following statement: For any $0<p'<p\leq 1$ and any $p$-Banach space $V$ and profinite set $S$, the map
\[
R\Hom_{\mathbb R}(\mathcal M_{p'}(S),V)\to V(S)
\]
is a quasi-isomorphism. Equivalently, for all $i>0$,
\[
\mathrm{Ext}^i_{\mathbb R}(\mathcal M_{p'}(S),V)=0.
\]
It is in the proof of this statement that one runs into the issue pointed out after Theorem~\ref{thm:pbanachcohom}.
\end{remark}

\newpage

\section*{Appendix to Lecture VIII: Completions of normed abelian groups}

\begin{definition}\label{def:normedabelian}\leavevmode
\begin{enumerate}
\item A normed abelian group is an abelian group $M$ equipped with a map
\[
||\cdot||: M\to \mathbb R_{\geq 0}
\]
satisfying $||0||=0$, $||-m||=||m||$ and $||m+n||\leq ||m||+||n||$ for all $m,n\in M$.
\item A normed abelian group is separated if $||m||=0$ implies $m=0$.
\item A normed abelian group is complete if it is separated and for every Cauchy sequence $(m_0,m_1,\ldots)$ (i.e., $||m_i-m_j||\to 0$ as $i,j\to \infty$) there is some $m\in M$ with $||m-m_i||\to 0$ as $i\to \infty$.
\end{enumerate}
\end{definition}

As is well-known, the inclusions of complete into separated into all normed abelian groups admit left adjoints: Separation takes $M$ to the quotient of $M$ by $\{m\in M\mid ||m||=0\}$; and completion takes $M$ to the quotient of the space of Cauchy sequences $(m_0,m_1,\ldots)$ ($||m_i-m_j||\to 0$ as $i,j\to \infty$) by the space of nullsequences $(m_0,m_1,\ldots)$ ($||m_i||\to 0$ as $i\to \infty$), with the norm defined by declaring $||(m_0,m_1,\ldots)||$ to be the limit of $||m_i||$ as $i\to \infty$.

We often use the following exactness property:

\begin{proposition}\label{prop:completeexact} Let $M_0\xrightarrow{d_0} M_1\xrightarrow{d_1} M_2\xrightarrow{d_2} M_3$ be a four-term complex of bounded maps of normed abelian groups. Assume that, for some positive constants $C$ and $D$, for all $y\in \ker(d_1: M_1\to M_2)$ there is some $x\in M_0$ with $d_0(x)=y$ and $||x||\leq C||y||$, and similarly for all $z\in \ker(d_2: M_2\to M_3)$, there is some $y\in M_1$ with $d_1(y)=z$ and $||y||\leq D||z||$.

Then $\widehat{M}_0\xrightarrow{\widehat{d}_0} \widehat{M}_1\xrightarrow{\widehat{d}_1} \widehat{M}_2\xrightarrow{\widehat{d}_2} \widehat{M}_3$ is a complex, and for all $\widehat{y}\in \widehat{M}_1$ and all $\epsilon>0$ there is some $\widehat{x}\in \widehat{M}_0$ with $\widehat{d}_1(\widehat{x})=\widehat{y}$ and $||\widehat{x}||\leq (C+\epsilon)||\widehat{y}||$.
\end{proposition}

\begin{proof} First, we claim that $\ker(d_1: M_1\to M_2)$ is dense in $\ker(\widehat{d}_1: \widehat{M}_1\to \widehat{M}_2)$. Pick any $\widehat{y}\in \ker(\widehat{d}_1)$ and $\delta>0$ and take $y\in M_1$ such that $||\widehat{y}-y||\leq \delta$. Let $z=d_1(y)\in M_2$, which has norm $||z||=||d_1(y)||=||d_1(y-\widehat{y})||$ bounded by $C_{d_1}\delta$, where $C_{d_1}$ is the norm of $d_1$. We can thus find some $y'\in M_1$ with $||y'||\leq DC_{d_1}\delta$ and $d_1(y')=z$. Replacing $y$ by $y-y'$, we can thus find $y\in \ker(d_1: M_1\to M_2)$ such that still $||\widehat{y}-y||\leq (1+DC_{d_1})\delta$; as $\delta$ was arbitrary, this gives the desired density.

This implies that one can write $\widehat{y}$ as a sum $y_0+y_1+\ldots$ with $y_i\in \ker(d_1)$ and $||y_i||\leq \epsilon_i$ for $i>0$ for any given sequence of positive numbers $\epsilon_1\geq \epsilon_2\geq \ldots$. Indeed, we can inductively choose the $y_i$ so that $||\widehat{y}-y_0-\ldots-y_i||\leq \tfrac 12 \epsilon_{i+1}$, in which case $||y_i||\leq \tfrac 12(\epsilon_i+\epsilon_{i+1})\leq \epsilon_i$. Taking the sequence of $\epsilon_i$'s sufficiently small so that $\sum_{i>0} \epsilon_i\leq \tfrac {||\widehat{y}||}{2C} \epsilon$, we can lift all $y_i$ to $x_i$ with $||x_i||\leq C||y_i||$, and then $\widehat{x}=x_0+x_1+\ldots$ maps to $\widehat{y}$ and satisfies
\[
||\widehat{x}||\leq ||x_0||+C\sum_{i>0} \epsilon_i\leq C||y_0||+C\sum_{i>0} \epsilon_i\leq C||\widehat{y}||+2C\sum_{i>0} \epsilon_i\leq (C+\epsilon)||\widehat{y}||.\qedhere
\]
\end{proof}

\begin{definition} Let $M$ be a normed abelian group. Let $\widehat{M}$ be the condensed abelian group taking any profinite $S$ to the completion of the normed abelian group of locally constant maps from $S$ to $M$ (equipped with the supremum norm).
\end{definition}

We note that implicit here is that this actually is a condensed abelian group.

\begin{proposition}\label{prop:normedcompletion} The condensed abelian group $\widehat{M}$ is canonically identified with the condensed abelian group associated to the topological abelian group $\widehat{M}_{\mathrm{top}}$ given by the completion of $M$ equipped with the topology induced by the norm. The norm defines a natural map of condensed sets
\[
||\cdot||: \widehat{M}\to \mathbb R_{\geq 0}.
\]

Moreover, for any hypercover $S_\bullet\to S$ of a profinite set $S$ by profinite sets $S_i$, the complex
\[
0\to \widehat{M}(S)\to \widehat{M}(S_0)\to \widehat{M}(S_1)\to \ldots
\]
is exact, and whenever $f\in \ker(\widehat{M}(S_m)\to \widehat{M}(S_{m+1}))$ with $||f||\leq c$, then for any $\epsilon>0$ there is some $g\in \widehat{M}(S_{m-1})$ with $||g||\leq (1+\epsilon)c$ such that $d(g)=f$.
\end{proposition}

\begin{proof} For the final assertion, follow the proof of \cite[Theorem 3.3]{Condensed}: When $S$ and all $S_i$ are finite, the hypercover splits, so a contracting homotopy gives the result with constant $1$. In general, write the hypercover as a cofiltered limit of hypercovers of finite sets by finite sets, pass to the filtered colimit, and complete, using Proposition~\ref{prop:completeexact}.

For the identification with the condensed abelian group associated to the topological abelian group $\widehat{M}_{\mathrm{top}}$, note that in the supremum norm any continuous function from $S$ to $\widehat{M}_{\mathrm{top}}$ can be approximated by locally constant functions arbitrarily well, and that the space of continuous functions from $S$ to $\widehat{M}_{\mathrm{top}}$ is complete with respect to the supremum norm. That $||\cdot||$ defines a map of condensed sets $\widehat{M}\to \mathbb R_{\geq 0}$ follows for example from this identification with $\underline{\widehat{M}_{\mathrm{top}}}$, as the norm is by definition a continuous map $\widehat{M}_{\mathrm{top}}\to \mathbb R_{\geq 0}$.
\end{proof}

\newpage

\section*{Appendix to Lecture VIII: Derived inverse limits}

We recall the following construction of derived inverse limits. Let $\mathcal I$ be some index category, and let $F: \mathcal I\to \mathcal A$ be a functor to some abelian category $\mathcal A$ that admits all colimits and compact projective generators (in particular, infinite products are exact in $\mathcal A$).

We build a cosimplicial object $L_F^\bullet$ of $\mathcal A$ whose $n$-th term is
\[
L_F^n = \prod_{i_0\to\ldots\to i_n} F(i_n)
\]
where the product is over all chains of $n$ composable morphisms $i_0\to i_1\to\ldots \to i_n$. We note that the index set of the product is equivalently the set of maps $\Delta^n_{\mathrm{cat}}\to \mathcal I$, where $\Delta^n_{\mathrm{cat}}$ is the category associated to the ordered set $\{0,\ldots,n\}$. Moreover, $F(i_n)$ can be understood as the colimit of $F$ restricted to $\Delta^n_{\mathrm{cat}}$. With this description, it is clear that a map of simplices $\Delta^n\to \Delta^m$ induces a map $L_F^n\to L_F^m$.

Before going on, let us analyze the case that $\mathcal I$ has an initial object $i\in \mathcal I$. In that case, $L_F^\bullet$ first of all extends to an augmented cosimplicial object $F(i)\to L_F^\bullet$, and this augmented cosimplicial object has terms
\[
\prod_{i=i_{-1}\to i_0\to \ldots\to i_n} F(i_n)
\]
for $n=-1,0,\ldots$, which are even functorial in maps $\Delta^{n+1}\to \Delta^{m+1}$ preserving the initial vertex. Formally (cf.~e.g.~(the dual of) \cite[Lemma 6.1.3.16]{LurieHTT}), this implies that the complex
\[
0\to F(i)\to L_F^0\to L_F^1\to \ldots
\]
associated to the augmented cosimplicial object is acyclic. More concretely, this is a direct verification: the contracting homotopy is given by the map $L_F^0\to F(i)$ taking the component at $i$, and the maps $L_F^n\to L_F^{n-1}$ given by restricting to components with $i=i_0$.

The following proposition is known as the ``Bousfield--Kan formula'', cf.~\cite[Chapter XI]{BousfieldKan}.

\begin{proposition} The complex associated to $L_F^\bullet$ by Dold-Kan computes $R\lim_{\mathcal I} F$.
\end{proposition}

We will prove this only when $\mathcal I$ is cofiltered, which is the argument that is used in the main text.

\begin{proof} It is clear that $\lim_{\mathcal I} F$ is the kernel of $L_F^0\to L_F^1$. Picking an injective resolution of $F$ (in the category of functors $\mathcal I\to \mathcal A$) and using exactness of infinite products, we get a natural map from $R\lim_{\mathcal I} F$ to $L_F^\bullet$. We claim that this is an equivalence.

If $\mathcal I$ is cofiltered, we can write $\mathcal I$ as a filtered union of finite subcategories $\mathcal I_j$ that contain an initial object. Then both sides are equal to the derived cofiltered limit of their variants for the $\mathcal I_j$. Thus, we can assume that $\mathcal I$ is finite and has an initial object $i$. But in this case, we have proved the exactness above.
\end{proof}

\newpage

\section{Lecture IX: End of proof}

Recall that in the last lecture, we reduced Theorem~\ref{thm:key} to the following result.

\begin{theorem}\label{thm:explicit1} Fix radii $1>r'>r>0$. Then for all $r$-normed $\mathbb Z[T^{\pm 1}]$-modules $V$ and all profinite sets $S$, the map
\[
R\Hom_{\mathbb Z[T^{-1}]}(\mathcal M(S,\mathbb Z((T))_{r'}),\widehat{V})\to \widehat{V}(S)
\]
is a quasi-isomorphism.
\end{theorem}

We note that actually the case of $\widehat{V}=\mathbb Z((T))_r^B$ itself seems to carry all essential difficulty: We believe that any argument one can give in that case will work in general.

In that sense, we have simplified the target of the $R\Hom$ as far as possible. It is time to understand the source. In the last lecture, we observed that the theorem automatically implies a more precise version bounding the norms of preimages of differentials. Our proof will actually go back to such explicit bounds, but for explicit resolutions of the source, which we will now construct. Write
\[
\overline{\mathcal M}_{r'}(S) := \mathcal M(S,\mathbb Z((T))_{r'})/\mathbb Z[T^{-1}][S].
\]
This is, as a condensed abelian group (but not as $\mathbb Z[T^{-1}]$-module), a direct summand $\mathcal M(S,T\mathbb Z[[T]]_{r'})$ of $\mathcal M(S,\mathbb Z((T))_{r'})$, allowing only positive powers of $T$.\footnote{At this point, we critically use that we chose $\mathbb Z[T^{-1}]$ as our base ring; already taking $\mathbb Z[T^{\pm 1}]$ would destroy this argument.} To see this, use Proposition~\ref{prop:freecondensedgroup} to see that the contribution to $\mathcal M(S,\mathbb Z((T))_{r'})$ from nonpositive powers of $T$ is exactly $\mathbb Z[T^{-1}][S]$. In particular, we can write
\[
\overline{\mathcal M}_{r'}(S) = \bigcup_{c>0} \overline{\mathcal M}_{r'}(S)_{\leq c},
\]
where
\[
\overline{\mathcal M}_{r'}(S)_{\leq c} = \varprojlim_i \overline{\mathcal M}_{r'}(S_i)_{\leq c}
\]
when writing $S$ as an inverse limit of finite sets $S_i$, and for finite $S$
\[
\overline{\mathcal M}_{r'}(S)_{\leq c} = \{(\sum_{n\geq 1} a_{n,s} T^n)_s\mid \sum_{n\geq 1,s\in S} |a_{n,s}|(r')^n\leq c\}.
\]

We need to resolve this explicitly as a condensed $\mathbb Z[T^{-1}]$-module. The Breen--Deligne resolution gives us the resolution
\[
\ldots\to \mathbb Z[\overline{\mathcal M}_{r'}(S)^2]\to \mathbb Z[\overline{\mathcal M}_{r'}(S)]\to \overline{\mathcal M}_{r'}(S)\to 0
\]
as condensed $\mathbb Z[T^{-1}]$-modules. Following a suggestion of Commelin found during the ``Liquid Tensor Experiment'' formalization of the results of this lecture in Lean, we will actually use a slightly different explicit complex that is not quite a resolution, but just as good for our purposes, cf.~\cite{CommelinLTE}. More precisely, we use MacLane's $Q$-complex \cite{MacLaneQ},\footnote{Sometimes also denoted $Q'$ while $Q$ is reserved for a closely related but slightly different complex. To keep notation light, we simply write $Q$ here.} which is an explicit complex
\[
Q(A): \ldots \to \mathbb Z[A^{2^n}]\to \ldots\to \mathbb Z[A^2]\to \mathbb Z[A]
\]
such that for torsion-free $A$, one has a natural isomorphism
\[
H_i(Q(A))\cong H_i(Q(\mathbb Z))\otimes A
\]
and moreover $H_0(Q(A))=A$; we refer to the appendix for its construction. It follows formally that
\[
R\Hom_{\mathbb Z[T^{-1}]}(\overline{\mathcal M}_{r'}(S),\hat{V})=0
\]
if and only if
\[
R\Hom_{\mathbb Z[T^{-1}]}(Q(\overline{\mathcal M}_{r'}(S)),\hat{V})=0.
\]
Indeed, the lowest nonzero cohomology group of these complexes agrees, by looking at the spectral sequence
\[
\mathrm{Ext}^i_{\mathbb Z[T^{-1}]}(H_j(Q(\mathbb Z))\otimes \overline{\mathcal M}_{r'}(S),\hat{V})\Rightarrow \mathrm{Ext}^{i+j}_{\mathbb Z[T^{-1}]}(Q(\overline{\mathcal M}_{r'}(S)),\hat{V}).
\]

Each term $\mathbb Z[\overline{\mathcal M}_{r'}(S)^a]$ (where $a=2^n$) admits a two-term resolution
\[
0\to \mathbb Z[T^{-1}][\overline{\mathcal M}_{r'}(S)^a]\xrightarrow{T^{-1}-[T^{-1}]} \mathbb Z[T^{-1}][\overline{\mathcal M}_{r'}(S)^a]\to \mathbb Z[\overline{\mathcal M}_{r'}(S)^a]\to 0
\]
that is functorial in the $\mathbb Z[T^{-1}]$-module $\mathbb Z[\overline{\mathcal M}_{r'}(S)^a]$. This gives a representation of $Q(\overline{\mathcal M}_{r'}(S))$ by the double complex
\[\xymatrix{
\ldots\ar[r] &\mathbb Z[T^{-1}][\overline{\mathcal M}_{r'}(S)^2]\ar[d]^{T^{-1}-[T^{-1}]}\ar[r] & \mathbb Z[T^{-1}][\overline{\mathcal M}_{r'}(S)]\ar[d]^{T^{-1}-[T^{-1}]}\\
\ldots\ar[r] &\mathbb Z[T^{-1}][\overline{\mathcal M}_{r'}(S)^2]\ar[r] & \mathbb Z[T^{-1}][\overline{\mathcal M}_{r'}(S)].
}\]
Unfortunately, the terms are not of the form $\mathbb Z[T^{-1}][E_i]$ with profinite $E_i$. To remedy this, we write the double complex as the filtered colimit of the double complexes
\begin{equation}\label{eq:doublecomplex}
\xymatrix{
\ldots\ar[r] &\mathbb Z[T^{-1}][\overline{\mathcal M}_{r'}(S)^2_{\leq 2^{-1} r'c}]\ar[d]^{T^{-1}-[T^{-1}]}\ar[r] & \mathbb Z[T^{-1}][\overline{\mathcal M}_{r'}(S)_{\leq r'c}]\ar[d]^{T^{-1}-[T^{-1}]}\\
\ldots\ar[r] &\mathbb Z[T^{-1}][\overline{\mathcal M}_{r'}(S)^2_{\leq 2^{-1} c}]\ar[r] & \mathbb Z[T^{-1}][\overline{\mathcal M}_{r'}(S)_{\leq c}]
}\end{equation}
for varying $c>0$, cf.~Lemma~\ref{lem:basehomotopy} in the appendix to this lecture.

It is easy to understand what happens to the vertical part under mapping to $V$:

\begin{lemma}\label{lem:Tinv} For any $r$-normed $\mathbb Z[T^{\pm 1}]$-module $V$, any $c>0$ and any $a$, the map
\[
\widehat{V}(\overline{\mathcal M}_{r'}(S)_{\leq c}^a)\xrightarrow{T^{-1}-[T^{-1}]^\ast} \widehat{V}(\overline{\mathcal M}_{r'}(S)_{\leq r'c}^a)
\]
is surjective, has norm bounded by $r^{-1}+1$, and for any $f\in \widehat{V}(\overline{\mathcal M}_{r'}(S)_{\leq r'c}^a)$ and $\epsilon>0$ there is some $g\in \widehat{V}(\overline{\mathcal M}_{r'}(S)_{\leq c}^a)$ with $f(x)=T^{-1}g(x)-g(T^{-1}x)$ and $||g||\leq \frac{r}{1-r}(1+\epsilon) ||f||$.
\end{lemma}

\begin{proof} Given $f: \overline{\mathcal M}_{r'}(S)_{\leq r'c}^a\to \widehat{V}$, choose an extension to a map $\tilde{f}: \overline{\mathcal M}_{r'}(S)^a\to \widehat{V}$ with $||\tilde{f}||\leq (1+\epsilon)||f||$. Such an extension exists: By induction (and using a sequence of $\epsilon_n$'s with $\prod_n (1+\epsilon_n)\leq 1+\epsilon$), it suffices to see that for any closed immersion $A\subset B$ of profinite sets and a map $f_A: A\to \widehat{V}$, there is an extension $f_B: B\to \widehat{V}$ of $f_A$ with $||f_B||\leq (1+\epsilon)||f_A||$. To see this, write $f_A$ as a (fast) convergent sum of maps that factor over a finite quotient of $A$; for maps factoring over a finite quotient of $A$, the extension is clear (and can be done in a norm-preserving way), as any map from $A$ to a finite set can be extended to a map from $B$ to the same finite set.

Given $\tilde{f}$, we can now define $g: \overline{\mathcal M}_{r'}(S)_{\leq c}^a$ by
\[
g(x) = T\tilde{f}(x)+T^2\tilde{f}(T^{-1}x)+\ldots+T^{n+1}\tilde{f}(T^{-n}x)+\ldots\in \widehat{V};
\]
then $||g||\leq \frac r{1-r}||\tilde{f}||\leq \frac r{1-r}(1+\epsilon)||f||$.
\end{proof}

Let $\widehat{V}(\overline{\mathcal M}(S)_{\leq c})^{T^{-1}}\subset \widehat{V}(\overline{\mathcal M}(S)_{\leq c})$ be the kernel of this map, with the subspace norm. For varying $c$, we now get varying normed complexes
\[
C_c^\bullet: \widehat{V}(\overline{\mathcal M}_{r'}(S)_{\leq c})^{T^{-1}}\to \widehat{V}(\overline{\mathcal M}_{r'}(S)_{\leq 2^{-1} c}^2)^{T^{-1}}\to \ldots .
\]
We will need to use the following qualitative notion of exactness.

\begin{definition} For each sufficiently large $c$ (i.e.~all $c\geq c_0$ for some $c_0>0$), let
\[
C_c^\bullet: C_c^0\to C_c^1\to\ldots
\]
be a complex of complete normed abelian groups, and for $c'>c$, let $\mathrm{res}_{c',c}^\bullet: C_{c'}^\bullet\to C_c^\bullet$ be a map of complexes, satisfying the obvious associativity condition. This datum is admissible if all differentials and maps $\mathrm{res}_{c',c}^i$ are norm-nonincreasing.

For integers $m\geq 0$ and constants $k \geq 1$, $c_0'>0$, the datum $(C_c^\bullet)_c$ is $\leq k$-exact in degrees $\leq m$ and for $c\geq c_0'$ if the following condition is satisfied. For all $c\geq c_0'$ and all $x\in C_{kc}^i$ with $i\leq m$ there is some $y\in C_c^{i-1}$ (which is defined to be $0$ when $i=0$) such that
\[
||\mathrm{res}_{kc,c}^i(x)-d_c^{i-1}(y)||_{C_c^i}\leq k||d_{kc}^i(x)||_{C_{kc}^{i+1}}.
\]
\end{definition}

We apply this to
\[
C_c^\bullet: \widehat{V}(\overline{\mathcal M}_{r'}(S)_{\leq c})^{T^{-1}}\to \widehat{V}(\overline{\mathcal M}_{r'}(S)_{\leq 2^{-1} c}^2)^{T^{-1}}\to \ldots
\]
given by mapping \eqref{eq:doublecomplex} into $\hat{V}$ and using Lemma~\ref{lem:Tinv}. Now we state the following result.

\begin{theorem}\label{thm:explicit2} Fix radii $1>r'>r>0$. For any $m$ there is some $k$ and $c_0$ such that for all profinite sets $S$ and all $r$-normed $\mathbb Z[T^{\pm 1}]$-modules $V$, the system of complexes
\[
C_c^\bullet: \widehat{V}(\overline{\mathcal M}_{r'}(S)_{\leq c})^{T^{-1}}\to \widehat{V}(\overline{\mathcal M}_{r'}(S)_{\leq 2^{-1} c}^2)^{T^{-1}}\to \ldots
\]
is $\leq k$-exact in degrees $\leq m$ for $c\geq c_0$.
\end{theorem}

Let us first check that this implies Theorem~\ref{thm:explicit1}.

\begin{proof}[Theorem~\ref{thm:explicit2} implies Theorem~\ref{thm:explicit1}] By the preceding discussion, one can compute
\[
R\Hom_{\mathbb Z[T^{-1}]}(Q(\overline{\mathcal M}_{r'}(S)),\widehat{V})
\]
as the derived inverse limit of $C_c^\bullet$ over all $c>0$; equivalently, all $c\geq c_0$. Theorem~\ref{thm:explicit2} implies that for any $m\geq 0$ the pro-system of cohomology groups $H^m(C_c^\bullet)$ is pro-zero (as $H^m(C_{kc}^\bullet)\to H^m(C_c^\bullet)$ is zero). Thus, the derived inverse limit vanishes, as desired.
\end{proof}

We will prove Theorem~\ref{thm:explicit2} by induction on $m$. The basic idea of the proof is to use the homotopy between adding internally and adding externally in $Q(-)$, codified in Lemma~\ref{lem:homotopyNelements}, and the different scaling behaviour of the norm on source and target. To make use of the different scaling behaviour, we need to write an element of size $C$ as a sum of $N$ terms each of size roughly $C/N$. As we work with $\mathbb Z$-modules, we cannot simply divide by $N$ (and even if we would work over $\mathbb R$, this would not have the desired effect on norms), but Lemma~\ref{lem:key} shows that it can be done. We note that Lemma~\ref{lem:key} is essentially an application of convex geometry. In the end, we are solving a problem about non-convex real vector spaces by using convex geometry, but in a discretized setting, for $\mathbb Z$-lattices!

Unfortunately, the induction requires us to prove a stronger statement. Namely, the homotopy will naturally introduce versions of the complex in Theorem~\ref{thm:explicit2} that involve not $\overline{\mathcal M}_{r'}(S)_{\leq c}$ but $\overline{\mathcal M}_{r'}(S)_{\leq c/N}^N$. In fact, further applications of the homotopy will lead to further modifications of this space. The following generalization turns out to be stable under the required modifications done in the induction step.

Consider any polyhedral lattice $\Lambda$, by which we mean a finite free abelian group equipped with a norm $||\cdot||_\Lambda: \Lambda\otimes \mathbb R\to \mathbb R$ (so $\Lambda\otimes \mathbb R$ is a Banach space) that is given by the supremum of finitely many linear functions on $\Lambda$ with rational coefficients; equivalently, the ``unit ball'' $\{\lambda\in \Lambda\otimes \mathbb R\mid ||\lambda||_\Lambda\leq 1\}$ is a rational polyhedron.

Endow $\Hom(\Lambda,\overline{\mathcal M}_{r'}(S))$ with the subspaces
\[
\Hom(\Lambda,\overline{\mathcal M}_{r'}(S))_{\leq c} = \{f: \Lambda\to \overline{\mathcal M}_{r'}(S)\mid \forall x\in \Lambda, f(x)\in \overline{\mathcal M}_{r'}(S)_{\leq c||x||}\}.
\]
As $\Lambda$ is polyhedral, it is enough to check the given condition for finitely many $x$.

For example, if $\Lambda=\mathbb Z^N$ with norm
\[
||(x_1,\ldots,x_n)||_\Lambda=\tfrac 1N(|x_1|+\ldots+|x_N|),
\]
then
\[
\Hom(\Lambda,\overline{\mathcal M}_{r'}(S))_{\leq c} = \overline{\mathcal M}_{r'}(S)_{\leq c/N}^N,
\]
as in the motivation above.

We can then define double complexes like \eqref{eq:doublecomplex} in this $\Lambda$-variant. Lemma~\ref{lem:Tinv} stays true with the same constants. Now we claim the following generalization of Theorem~\ref{thm:explicit2}. Critically, $k$ depends only on $m$, while $c_0$ is allowed to depend on $\Lambda$; this is required to make the induction work.

\begin{theorem}\label{thm:explicit3} Fix radii $1>r'>r>0$. For any $m$ there is some $k$ such that for all polyhedral lattices $\Lambda$ there is a constant $c_0(\Lambda)>0$ such that for all profinite sets $S$ and all $r$-normed $\mathbb Z[T^{\pm 1}]$-modules $V$, the system of complexes
\[
C_{\Lambda,c}^\bullet: \widehat{V}(\Hom(\Lambda,\overline{\mathcal M}_{r'}(S))_{\leq c})^{T^{-1}}\to \widehat{V}(\Hom(\Lambda,\overline{\mathcal M}_{r'}(S))_{\leq 2^{-1} c}^2)^{T^{-1}}\to \ldots
\]
is $\leq k$-exact in degrees $\leq m$ for $c\geq c_0(\Lambda)$.
\end{theorem}

We note that it can be ensured that the transition maps in the complex are norm-nonincreasing. Indeed, the maps
\[
\widehat{V}(\Hom(\Lambda,\overline{\mathcal M}_{r'}(S))_{\leq 2^{-i} c}^{2^i})\to \widehat{V}(\Hom(\Lambda,\overline{\mathcal M}_{r'}(S))_{\leq 2^{-i-1} c}^{2^{i+1}})
\]
have bounded norm, depending only on $i$ and $r$ (as they are a certain universal finite sum of maps induced by maps between the profinite sets in paranthesis, each of which induces a map of norm bounded by $1$), so up to rescaling the norm on the $i$-th term by some universal constant (depending only on $i$ and $r$) this can be arranged. We make and fix this choice of norms in the statement of Theorem~\ref{thm:explicit3}, and the rest of the proof.

The homotopy constructed in this way has to be used to leverage exactness at the next step. This is a spectral sequence argument, abstracted in the following proposition.

\begin{proposition}\label{prop:key} Fix an integer $m\geq 0$ and a constant $k$. Then there exists an $\epsilon>0$ and a constant $k_0\geq k$, depending (only) on $k$ and $m$, with the following property.

Consider an admissible system of double complexes $M^{p,q}_c$ (where $p,q\geq 0$, $c\geq c_0$) of complete normed abelian groups as well as some $k'\geq k_0$ and some $H>0$, such that
\[\xymatrix{
M^{0,0}_c\ar[r]^{d'^{0,0}_c}\ar[d]^{d^{0,0}_c} & M^{0,1}_c\ar[r]^{d'^{0,1}_c}\ar[d]^{d^{0,1}_c} & M^{0,2}_c\ar[r]^{d'^{0,2}_c}\ar[d]^{d^{0,2}_c} & \ldots\\
M^{1,0}_c\ar[r]^{d'^{1,0}_c}\ar[d]^{d^{1,0}_c} & M^{1,1}_c\ar[r]^{d'^{1,1}_c}\ar[d]^{d^{1,1}_c} & M^{1,2}_c\ar[r]^{d'^{1,2}_c}\ar[d]^{d^{1,2}_c} & \ldots\\
M^{2,0}_c\ar[r]^{d'^{2,0}_c}\ar[d]^{d^{2,0}_c} & M^{2,1}_c\ar[r]^{d'^{2,1}_c}\ar[d]^{d^{2,1}_c} & \ddots\\
\vdots & \vdots
}\]
\begin{enumerate}
\item for $i=0,\ldots,m+1$, the rows $M^{i,\bullet}_c$ are $\leq k$-exact in degrees $\leq m-1$ for $c\geq c_0$;
\item for $j=0,\ldots,m$, the columns $M^{\bullet,j}_c$ are $\leq k$-exact in degrees $\leq m$ for $c\geq c_0$;
\item for $q=0,\ldots,m$ and $c\geq c_0$, there is a map $h^q_{k'c}: M^{0,q+1}_{k'c}\to M^{1,q}_c$ with
\[
||h^q_{k'c}(x)||_{M^{1,q}_c}\leq H||x||_{M^{0,q+1}_{k'c}}
\]
for all $x\in M^{0,q+1}_{k'c}$, and such that for all $c\geq c_0$ and $q=0,\ldots,m$ the ``homotopic'' map
\[
\mathrm{res}_{k'^2c,k'c}^{1,q}\circ d^{0,q} + h^q_{k'^2c}\circ d'^{0,q}_{k'^2c} + d'^{1,q-1}_{k'c}\circ h^{q-1}_{k'^2c}: M^{0,q}_{k'^2c}\to M^{1,q}_{k'c}
\]
factors as a composite of the restriction $\mathrm{res}_{k'^2c,c}^{0,q}$ and a map
\[
\delta^{0,q}_c: M^{0,q}_c\to M^{1,q}_{k'c}
\]
that is a map of complexes (in degrees $\leq m$), and satisfies the estimate
\begin{equation}\label{eq:homotopicmapsmall}
||\delta^{0,q}_c(x)||_{M^{1,q}_{k'c}}\leq \epsilon ||x||_{M^{0,q}_c}
\end{equation}
for all $x\in M^{0,q}_c$.
\end{enumerate}
Then the first row is $\leq \max(k'^2,2k_0H)$-exact in degrees $\leq m$ for $c\geq c_0$.
\end{proposition}

We note that the homotopy in (3) is bounded in two ways: On the one hand, the homotopic map factors over a deep restriction, and on the other hand its norm is bounded by $\epsilon$. If we imagine $c$ as a time parameter and $M^{p,q}_c$ as knowing about ``history up to time $c$'' so that the restriction maps $M^{p,q}_{c'}\to M^{p,q}_c$ correspond to forgetting the recent history, then the homotopic map $\delta^{0,q}_c$ has the effect of ``travelling forward in time''. It is also important that while the required bound on the homotopic map is very strong, one does not need strong control on the homotopy $h$: There must be some bound $H$ on its norm, but $H$ can be arbitrary.

\begin{proof} First, we treat the case $m=0$. If $m=0$, we claim that one can take $\epsilon=\tfrac 1{2k}$ and $k_0=k$. We have to prove exactness at the first step. Let $x_{k'^2c}\in M^{0,0}_{k'^2c}$ and denote $x_{k'c}=\mathrm{res}_{k'^2c,k'c}^{0,0}(x)$ and $x_c=\mathrm{res}_{k'^2c,c}^{0,0}(x)$. Then by assumption (2) (and $k'\geq k$), we have
\[
||x_c||_{M^{0,0}_c}\leq k||d^{0,0}_{k'c}(x_{k'c})||_{M^{1,0}_{k'c}}.
\]
On the other hand, by (3),
\[
||\mathrm{res}_{k'^2c,k'c}^{1,0}(d^{0,0}_{k'^2c}(x))+ h^0_{k'^2c}(d'^{0,0}_{k'^2c}(x))||_{M^{1,0}_{k'c}}\leq \epsilon ||x_c||_{M^{0,0}_c},
\]
noting that the left-hand side agrees with $\delta^{0,0}_c(x_c)$ by assumption. In particular, noting that $\mathrm{res}_{k'^2c,k'c}^{1,0}(d^{0,0}_{k'^2c}(x)) = d^{0,0}_{k'c}(x_{k'c})$, we get
\[
||x_c||_{M^{0,0}_c}\leq k||d^{0,0}_{k'c}(x_{k'c})||_{M^{1,0}_{k'c}}\leq k\epsilon ||x_c||_{M^{0,0}_c} + kH ||d'^{0,0}_{k'^2c}(x)||_{M^{0,1}_{k'^2c}}.
\]
Thus, taking $\epsilon=\tfrac 1{2k}$ as promised, and bringing $2^{-1}||x_c||_{M^{0,0}_c}$ to the left-hand side, this implies
\[
||x_c||_{M^{0,0}_c}\leq 2kH ||d'^{0,0}_{k'^2c}(x)||_{M^{0,1}_{k'^2c}}.
\]
This gives the desired $\leq \max(k'^2,2k_0H)$-exactness in degrees $\leq m$ for $c\geq c_0$.

Now we argue by induction on $m$. Consider the complex $N^{p,q}$ given by $M^{p,q+1}$ for $q\geq 1$ and $N^{p,0} = M^{p,1}/\overline{M^{p,0}}$ (the quotient by the closure of the image, which is also the completion of $M^{p,1}/M^{p,0}$), equipped with the quotient norm. Using the normed version of the snake lemma, Proposition~\ref{prop:snakelemma} in the appendix to this lecture, one checks that this satisfies the assumptions for $m-1$, with $k$ replaced by $\max(k^4,k^3+k+1)$. To verify condition (3), note that the maps $\delta^{0,q}_c$ induce similar maps after passing to this quotient complex. To verify the estimate \eqref{eq:homotopicmapsmall}, note that it is nontrivial only for $N^{0,0} = M^{0,1}/\overline{M^{0,0}}$. In that case, for any given $a>0$ one can lift $x\in N^{0,0}_c$ to $\tilde{x}\in M^{0,1}_c$ with $||\tilde{x}||_{N^{0,0}_c}\leq ||x||_{M^{0,1}_c}+a$. This implies
\[
||\delta^{0,q}_c(x)||_{N^{1,0}_{k'c}}\leq ||\delta^{0,q}_c(\tilde{x})||_{M^{1,1}_{k'c}}\leq \epsilon ||\tilde{x}||_{M^{0,1}_c}\leq \epsilon ||x||_{M^{0,1}_c} + \epsilon a
\]
for all $a>0$, and hence the desired inequality by taking the infimum over all $a$.
\end{proof}

Finally, we can prove the key combinatorial lemma, ensuring that any element of $\Hom(\Lambda,\overline{\mathcal M}_{r'}(S))$ can be decomposed into $N$ elements whose norm is roughly $\tfrac 1N$ of the original element. As preparation, we have the following simple result.

\begin{lemma} Let $\Lambda$ be a finite free abelian group, let $N$ be a positive integer, and let $\lambda_1,\ldots,\lambda_m\in \Lambda$ be elements. Then there is a finite subset $A\subset \Lambda^\vee$ such that for all $x\in \Lambda^\vee=\Hom(\Lambda,\mathbb Z)$ there is some $x'\in A$ such that $x-x'\in N\Lambda^\vee$ and for all $i=1,\ldots,m$, the numbers $x'(\lambda_i)$ and $(x-x')(\lambda_i)$ have the same sign, i.e.~are both nonnegative or both nonpositive.
\end{lemma}

\begin{proof} It suffices to prove the statement for all $x$ such that $\lambda_i(x)\geq 0$ for all $i$; indeed, applying this variant to all $\pm \lambda_i$, one gets the full statement.

Thus, consider the submonoid $\Lambda^\vee_+\subset \Lambda^\vee$ of all $x$ that pair nonnegatively with all $\lambda_i$. This is a finitely generated monoid by Gordan's lemma; let $y_1,\ldots,y_M$ be a set of generators. Then we can take for $A$ all sums $n_1y_1+\ldots+n_My_M$ where all $n_j\in \{0,\ldots,N-1\}$.
\end{proof}

Now we have the key lemma:

\begin{lemma}\label{lem:key} Let $\Lambda$ be a polyhedral lattice. Then for all positive integers $N$ there is a constant $d$ such that for all $c>0$ and all profinite sets $S$ one can write any $x\in \Hom(\Lambda,\overline{\mathcal M}_{r'}(S))_{\leq c}$ as
\[
x=x_1+\ldots+x_N
\]
where all $x_i\in \Hom(\Lambda,\overline{\mathcal M}_{r'}(S))_{\leq c/N+d}$.
\end{lemma}

\begin{proof} The desired statement is equivalent to the surjectivity of the map of profinite sets
\[
\Hom(\Lambda,\overline{\mathcal M}_{r'}(S))_{\leq c/N+d}^N\times_{\Hom(\Lambda,\overline{\mathcal M}_{r'}(S))_{\leq c+Nd}} \Hom(\Lambda,\overline{\mathcal M}_{r'}(S))_{\leq c}\to \Hom(\Lambda,\overline{\mathcal M}_{r'}(S))_{\leq c}.
\]
Note that, as a functor of $S$, both sides commute with cofiltered limits, so it is enough to handle finite $S$, by Tychonoff.

Pick $\lambda_1,\ldots,\lambda_m\in \Lambda$ generating the norm. We fix a finite subset $A\subset \Lambda^\vee$ satisfying the conclusion of the previous lemma. Write, for finite $S$,
\[
x=\sum_{n\geq 1, s\in S} x_{n,s} T^n [s]
\]
with $x_{n,s}\in \Lambda^\vee$. Then we can decompose
\[
x_{n,s} = N x_{n,s}^0 + x_{n,s}^1
\]
where $x_{n,s}^1\in A$ and we have the same-sign property of the last lemma. Letting $x^0 = \sum_{n\geq 1, s\in S} x_{n,s}^0 T^n [s]$, we get a decomposition
\[
x = Nx^0 + \sum_{a\in A} a x_a
\]
with $x_a\in \overline{\mathcal M}_{r'}(S)$ (with the property that in the
basis given by the $T^n [s]$, all coefficients are $0$ or $1$). Crucially,
we know that for all $i=1,\ldots,m$, we have
\[
||x(\lambda_i)|| = N ||x^0(\lambda_i)|| + \sum_{a\in A} |a(\lambda_i)| ||x_a||
\]
by using the same sign property of the decomposition.

Using this decomposition of $x$, we decompose each term into $N$ summands. This is trivial for the first term $Nx^0$, and each summand of the second term reduces to the similar problem for $\Lambda=\mathbb Z$. In that case, one can take $d=1$, as follows by decomposing any sum with terms of size at most $1$ into $N$ such partial sums whose sums differ by at most $1$. (It follows that in general one can take for $d$ the supremum over all $i$ of $\sum_{a\in A} |a(\lambda_i)|$.)
\end{proof}

Finally, we can give the main argument.

\begin{proof}[Proof of Theorem~\ref{thm:explicit3}] We argue by induction on $m$, so assume the result for $m-1$ (this is no assumption for $m=0$, so we do not need an induction start). This gives us some $k\geq 2$ for which the statement of Theorem~\ref{thm:explicit3} holds true for $m-1$; if $m=0$, simply take any $k\geq 2$. Moreover, we assume $k$ is at least some constant that depends only on $m$ and $r$. (More precisely, we will need to invoke Lemma~\ref{lem:Tinv} below --- its constants depend only on $m$ and $r$.) After this choice of $k$, we fix $\epsilon$ and $k_0$ as provided by Proposition~\ref{prop:key}; take any $k'\geq k_0$, for example $k'=k_0$. Finally, choose a positive integer $b$ so that $2k'(\tfrac r{r'})^b\leq \epsilon$, and let $N$ be the minimal power of $2$ that satisfies
\[
k'/N\leq (r')^b.
\]
Then in particular $r^bN\leq 2k'(\tfrac{r}{r'})^b\leq \epsilon$.

We consider the diagonal embedding
\[
\Lambda\hookrightarrow \Lambda' = \Lambda^N,
\]
where we endow $\Lambda'$ with the norm
\[
||(\lambda_1,\ldots,\lambda_N)||_{\Lambda'} = \tfrac 1N(||\lambda_1||_\Lambda+\ldots+||\lambda_N||_\Lambda).
\]
For any $m\geq 1$, let $\Lambda'^{(m)}$ be given by $\Lambda'^m / \Lambda\otimes (\mathbb Z^m)_{\sum=0}$; then $\Lambda'^{(\bullet)}$ is a cosimplicial polyhedral lattice, the \v{C}ech conerve of $\Lambda\to \Lambda'$. For $m=0$, we set $\Lambda'^{(0)} = \Lambda$. It is clear that all of these are polyhedral lattices.

In particular, for any $c>0$, we have
\[
\Hom(\Lambda'^{(m)},\overline{\mathcal M}_{r'}(S))_{\leq c} = \Hom(\Lambda',\overline{\mathcal M}_{r'}(S))_{\leq c}^{m/\Hom(\Lambda,\overline{\mathcal M}_{r'}(S))_{\leq c}},
\]
the $m$-fold fibre product of $\Hom(\Lambda',\overline{\mathcal M}_{r'}(S))_{\leq c}$ over $\Hom(\Lambda,\overline{\mathcal M}_{r'}(S))_{\leq c}$; and
\[
\Hom(\Lambda',\overline{\mathcal M}_{r'}(S))_{\leq c} = \Hom(\Lambda,\overline{\mathcal M}_{r'}(S))_{\leq c/N}^N,
\]
with the map to $\Hom(\Lambda,\overline{\mathcal M}_{r'}(S))_{\leq c}$ given by the sum map.

Consider the collection of double complexes $C_{\Lambda'^{(\bullet)},c}^\bullet$ associated to this cosimplicial polyhedral lattice by Dold--Kan. Up to rescaling the norms in the complex for $\Lambda'^{(m)}$ by a universal constant (something like $(m+2)!$), the differentials are norm-nonincreasing (as they are an alternating sum of $m+1$ face maps, all of which are of norm $\leq 1$ for the standard norm), so this collection of normed double complexes is admissible. By induction, the first condition of Proposition~\ref{prop:key} is satisfied for all $c\geq c_0$ with $c_0$ large enough (depending on $\Lambda$ but not $V$ or $S$). By Lemma~\ref{lem:key}, and noting that $\Hom(\Lambda'^{(\bullet)},\overline{\mathcal M}_{r'}(S))_{\leq c}$ is the \v{C}ech nerve of
\[
\Hom(\Lambda,\overline{\mathcal M}_{r'}(S))_{\leq c/N}^N\xrightarrow{\sum} \Hom(\Lambda,\overline{\mathcal M}_{r'}(S))_{\leq c},
\]
also the second condition is satisfied, provided we take $c_0$ large enough so that $2^{-m}(k-1)r'c_0/N$ is at least the $d$ of Lemma~\ref{lem:key} (so this choice of $c_0$ again depends on $\Lambda$). Indeed, then one can, for all $i=0,\ldots,m$, splice a surjection of profinite sets between the maps
\[
\Hom(\Lambda,\overline{\mathcal M}_{r'}(S))_{\leq 2^{-i}c/N}^{2^iN}\to\Hom(\Lambda,\overline{\mathcal M}_{r'}(S))_{\leq 2^{-i}c}^{2^i}
\]
and
\[
\Hom(\Lambda,\overline{\mathcal M}_{r'}(S))_{\leq k2^{-i}c/N}^{2^iN}\to \Hom(\Lambda,\overline{\mathcal M}_{r'}(S))_{\leq k2^{-i}c}^{2^i},
\]
and so the transition map between the columns of that double complex factors over a similar complex arising from a simplicial hypercover of profinite sets, so the constants are bounded by Proposition~\ref{prop:normedcompletion}, Lemma~\ref{lem:Tinv}, and (the version for kernels instead of cokernels of) Proposition~\ref{prop:snakelemma}. This is the argument invoking Lemma~\ref{lem:Tinv} alluded to in the first paragraph of the proof, regarding the choice of $k$.

Finally, to check the third condition, we use Lemma~\ref{lem:homotopyNelements} to find, in degrees $\leq m$, a homotopy between the two maps from the first row
\[
\widehat{V}(\Hom(\Lambda,\overline{\mathcal M}_{r'}(S))_{\leq c})^{T^{-1}}\to \widehat{V}(\Hom(\Lambda,\overline{\mathcal M}_{r'}(S))_{\leq 2^{-1} c}^2)^{T^{-1}}\to \ldots
\]
to the second row
\[
\widehat{V}(\Hom(\Lambda,\overline{\mathcal M}_{r'}(S))_{\leq c/N}^N)^{T^{-1}}\to \widehat{V}(\Hom(\Lambda,\overline{\mathcal M}_{r'}(S))_{\leq 2^{-1} c/N}^{2N})^{T^{-1}}\to \ldots
\]
respectively induced by the addition $\Hom(\Lambda,\overline{\mathcal M}_{r'}(S))_{\leq c/N}^N\to \Hom(\Lambda,\overline{\mathcal M}_{r'}(S))_{\leq c}$ (which is the map that forms part of the double complex), and the map that is the sum of the $N$ maps induced by the $N$ projection maps
\[
\Hom(\Lambda,\overline{\mathcal M}_{r'}(S))_{\leq c/N}^N\to \Hom(\Lambda,\overline{\mathcal M}_{r'}(S))_{\leq c/N}\subset \Hom(\Lambda,\overline{\mathcal M}_{r'}(S))_{\leq c}.
\]
This homotopy is induced from the universal homotopy from Lemma~\ref{lem:homotopyNelements}, and in particular its norm is bounded by some constant $H$ (depending only on $N$ and $r$).

Finally, it remains to establish the estimate \eqref{eq:homotopicmapsmall} on the homotopic map. We note that this takes $x\in \widehat{V}(\Hom(\Lambda,\overline{\mathcal M}_{r'}(S))_{\leq k'^2 2^{-i} c}^{2^i})^{T^{-1}}$ (with $i=q$ in the notation of \eqref{eq:homotopicmapsmall}) to the element
\[
y\in \widehat{V}(\Hom(\Lambda,\overline{\mathcal M}_{r'}(S))_{\leq k'2^{-i}c/N}^{2^i N})^{T^{-1}}
\]
that is the sum of the $N$ pullbacks along the $N$ projection maps $\Hom(\Lambda,\overline{\mathcal M}_{r'}(S))_{\leq k'2^{-i}c/N}^{2^i N}\to \Hom(\Lambda,\overline{\mathcal M}_{r'}(S))_{\leq k'^2 2^{-i}c}^{2^i}$. We note that these actually take image in $\Hom(\Lambda,\overline{\mathcal M}_{r'}(S))_{\leq 2^{-i}c}^{2^i}$ as $N\geq k'$, so this actually gives a well-defined map
\[
\widehat{V}(\Hom(\Lambda,\overline{\mathcal M}_{r'}(S))_{\leq 2^{-i}c}^{2^i})^{T^{-1}}\to \widehat{V}(\Hom(\Lambda,\overline{\mathcal M}_{r'}(S))_{\leq k'2^{-i}c/N}^{2^i N})^{T^{-1}}.
\]
We need to see that this map is of norm $\leq \epsilon$. Now note that by our choice of $N$, we actually have $k'2^{-i}c/N\leq (r')^b 2^{-i}c$, so this can be written as the composite of the restriction map
\[
\widehat{V}(\Hom(\Lambda,\overline{\mathcal M}_{r'}(S))_{\leq 2^{-i}c}^{2^i})^{T^{-1}}\to \widehat{V}(\Hom(\Lambda,\overline{\mathcal M}_{r'}(S))_{\leq (r')^b 2^{-i}c}^{2^i})^{T^{-1}}
\]
and
\[
\widehat{V}(\Hom(\Lambda,\overline{\mathcal M}_{r'}(S))_{\leq (r')^b 2^{-i}c}^{2^i})^{T^{-1}}\to \widehat{V}(\Hom(\Lambda,\overline{\mathcal M}_{r'}(S))_{\leq k'2^{-i}c/N}^{2^i N})^{T^{-1}}.
\]
The first map has norm exactly $r^b$, by $T^{-1}$-invariance, and as multiplication by $T$ scales the norm with a factor of $r$ on $\widehat{V}$.\footnote{Here is where we use $r'>r$, ensuring different scaling behaviour of the norm on source and target.} The second map has norm at most $N$ (as it is a sum of $N$ maps of norm $\leq 1$). Thus, the total map has norm $\leq r^bN$. But by our choice of $N$, we have $r^bN\leq \epsilon$, giving the result.

Thus, we can apply Proposition~\ref{prop:key}, and get the desired $\leq \max(k'^2,2k_0H)$-exactness in degrees $\leq m$ for $c\geq c_0$, where $k'$, $k_0$ and $H$ were defined only in terms of $k$, $m$, $r'$ and $r$, while $c_0$ depends on $\Lambda$ (but not on $V$ or $S$). This proves the inductive step.
\end{proof}

\begin{remark} It has been formally verified that the proof yields constants $k=k(m)$ that grow doubly-exponentially in $m$.
\end{remark}

This completes the proof of all results announced so far.

\newpage

\section*{Appendix to Lecture IX: Some normed homological algebra}

In this appendix, we gather a few results about homological algebra with normed abelian groups, the proofs of which are just obtained by keeping track of constants in the standard proofs.

\begin{proposition}\label{prop:snakelemma} Let $M^\bullet_c$ and $M'^\bullet_c$ be two admissible collections of complexes of complete normed abelian groups, where $c\geq c_0$. Let $f^\bullet_c: M^\bullet_c\to M'^\bullet_c$ be a collection of maps between these collections of complexes that is norm-nonincreasing and commutes with restriction maps, and assume that it satisfies
\[
||\mathrm{res}^i_{kc,c}(x)||_{M^i_c}\leq k||f^i_{kc}(x)||_{M'^i_{kc}}
\]
for all $i=0,\ldots,m+1$ and all $x\in M^i_{kc}$. Let $N^\bullet_c=M'^\bullet_c/\overline{M^\bullet_c}$ (which equals the completion of $M'^{\bullet}_c/M^\bullet_c$) be the collection of quotient complexes, with the quotient norm; this is again an admissible collection of complexes.

Assume that $M^\bullet_c$ and $M'^\bullet_c$ are $\leq k$-exact in degrees $\leq m$ for $c\geq c_0$. Then $N^\bullet_c$ is $\leq \max(k^4,k^3+k+1)$-exact in degrees $\leq m-1$ for $c\geq c_0$.
\end{proposition}

\begin{proof} We make the following preliminary observation. Take any $i=0,\ldots,m+1$ and $m'_{kc}\in M'^i_{kc}$ with image $n_{kc}\in N^i_{kc}$. By the definition of the quotient norm, for any $\epsilon>0$ we can find some $m_{kc}\in M^i_{kc}$ such that $||m'_{kc}-f^i_{kc}(m_{kc})||\leq ||n_{kc}||+\epsilon$. We would like to replace this by the stronger assertion that we can find $m_{kc}\in M^i_{kc}$ such that
\[
||m'_{kc}-f^i_{kc}(m_{kc})||\leq (1+\epsilon)||n_{kc}||.
\]
This is obviously possible as long as $||n_{kc}||>0$, but in case $||n_{kc}||=0$, it may not be possible, because $M^\bullet_c\to M'^\bullet_c$ may not have closed image.

However, we claim that, letting $m'_c\in M'^i_c$ be the restriction of $m'_{kc}\in M'^i_{kc}$, with image $n_c\in N^i_c$, we can always find some $m_c\in M^i_c$ such that
\[
||m'_c-f^i_c(m_c)||\leq (1+\epsilon)||n_{kc}||.
\]
By the above, we only need to prove this when $||n_{kc}||=0$. Choose a sequence $m_{kc,0},m_{kc,1},\ldots$ in $M^i_{kc}$ such that $||m'_{kc}-f^i_{kc}(m_{kc,j})||\to 0$ for $j\to \infty$. In particular, $||f^i_{kc}(m_{kc,j}-m_{kc,j'})||\to 0$ for $j,j'\to \infty$. By the displayed bound in the statement of the proposition, this ensures that $||m_{c,j}-m_{c,j'}||\to 0$ where $m_{c,j}\in M^i_c$ is the image of $m_{kc,j}$. Thus, we get a Cauchy sequence in $M^i_c$ whose limit $m_c\in M^i_c$ will satisfy $||m'_c-f^i_c(m_c)||=0$ (i.e.~$m'_c=f^i_c(m_c)$).

Now we start the proof of the proposition. Let $n^i_{k^4c}\in N^i_{k^4c}$ for $i\leq m-1$, with image $n^{i+1}_{k^4c}\in N^{i+1}_{k^4c}$, and let $C:=||n^{i+1}_{k^4c}||_{N^{i+1}_{k^4c}}$. We need to find an element $n^i_c\in N^{i-1}_c$ such that
\[
||n^i_c - d^{i-1}_{N,c}(n^{i-1}_c)||_{N^i_c}\leq (k^3+k+1)C,
\]
where we change the subscript when applying restriction maps.

Pick any preimage $m'^i_{k^4c}\in M'^i_{k^4c}$ of $n^i_{k^4c}$, and let $m'^{i+1}_{k^4c}\in M'^{i+1}_{k^4c}$ be its image. By the preliminary observation, we can find $m^{i+1}_{k^3c}\in M^{i+1}_{k^3c}$ such that
\[
m'^{i+1}_{k^3c} = f^{i+1}_{k^3c}(m^{i+1}_{k^3c}) + m''^{i+1}_{k^3c}
\]
with $||m''^{i+1}_{k^3c}||_{M'^{i+1}_{k^3c}}\leq (1+\epsilon)C$, where we choose $\epsilon$ so that $(k^3+k)(1+\epsilon)\leq k^3+k+1$.

Let $m^{i+2}_{k^3c}\in M^{i+2}_{k^3c}$ be the image of $m^{i+1}_{k^3c}$. Applying the differential to the last displayed equation, and using that this kills $m'^{i+1}_{k^3c}$, and that $f^\bullet_{k^3c}$ is a map of complexes, we see that
\[
f^{i+2}_{k^3c}(m^{i+2}_{k^3c}) = -m''^{i+2}_{k^3c},
\]
where similarly $m''^{i+2}_{k^3c}$ is the differential of $m''^{i+1}_{k^3c}$. We get
\[\begin{aligned}
||m^{i+2}_{k^2c}||_{M^{i+2}_{k^2c}}&\leq k||f^{i+2}_{k^3c}(m^{i+2}_{k^3c})||_{M'^{i+2}_{k^3c}} = k||m''^{i+2}_{k^3c}||_{M'^{i+2}_{k^3c}}\\
&\leq k||m''^{i+1}_{k^3c}||_{M'^{i+1}_{k^3c}}\leq k(1+\epsilon)C.
\end{aligned}\]
On the other hand, we can find some $m^i_{kc}\in M^i_{kc}$ such that
\[
||m^{i+1}_{kc}-d^i_{kc}(m^i_{kc})||\leq k||m^{i+2}_{k^2c}||_{M^{i+2}_{k^2c}}\leq k^2(1+\epsilon)C.
\]
Now let $m'^i_{kc,\mathrm{new}} = m'^i_{kc}-f^i_{kc}(m^i_{kc})\in M'^i_{kc}$; this is a lift of $n^i_{kc}$. Then the image $m'^{i+1}_{kc,\mathrm{new}}$ in $M'^{i+1}_{kc}$ satisfies
\[
m'^{i+1}_{kc,\mathrm{new}} = m'^{i+1}_{kc}-f^{i+1}_{kc}(m^{i+1}_{kc}) + f^{i+1}_{kc}(m^{i+1}_{kc}-d^i_{kc}(m^i_{kc})) = m''^{i+1}_{kc} + f^{i+1}_{kc}(m^{i+1}_{kc}-d^i_{kc}(m^i_{kc})).
\]
In particular,
\[
||m'^{i+1}_{kc,\mathrm{new}}||_{M'^{i+1}_{kc}}\leq (1+\epsilon)C+ k^2(1+\epsilon)C.
\]
Now we can find $m'^{i-1}_c\in M'^{i-1}_c$ such that
\[
||m'^i_{c,\mathrm{new}} - d'^{i-1}_c(m'^{i-1}_c)||_{M'^i_c}\leq k||m'^{i+1}_{kc,\mathrm{new}}||_{M'^{i+1}_{kc}}\leq (k^3+k)(1+\epsilon)C.
\]
In particular, letting $n^{i-1}_c\in N^{i-1}_c$ be the image of $m'^{i-1}_c$, we get
\[
||n^i_c - d^{i-1}_{N,c}(n^{i-1}_c)||_{N^i_c}\leq (k^3+k)(1+\epsilon)C,
\]
so by our choice of $\epsilon$ this gives the desired result.
\end{proof}

We need the some results about the Breen--Deligne resolution for normed abelian groups. Let us consider here abelian groups $M$ (in any topos) equipped with an increasing filtration $M_{\leq c}\subset M$ by subobjects indexed by the positive real numbers, such that $0\in M_{\leq c}$, $-M_{\leq c} = M_{\leq c}$ and $M_{\leq c}+M_{\leq c'}\subset M_{\leq c+c'}$; we need no further conditions. Let us call these pseudo-normed abelian groups.

Actually, following a suggestion of Commelin, we will work with MacLane's ``cubical'' $Q$-complex
\[
Q(M):\ldots\to \mathbb Z[M^{2^n}]\ldots \to \mathbb Z[M^2]\to \mathbb Z[M].
\]
We will define the differentials momentarily, but we first we stress that this is not an actually a resolution of $M$. It is however close to a resolution, as we will discuss below, and this suffices for our applications.

The differentials are given as follows:
\[\begin{aligned}
\mathbb Z[M^2]&\to \mathbb Z[M]\\
[(m_1,m_2)]&\mapsto [m_1+m_2]-[m_1]-[m_2]
\end{aligned}\]
and
\[\begin{aligned}
\mathbb Z[M^4]\to &\ \mathbb Z[M^2]\\
[(m_{11},m_{12},m_{21},m_{22})]\mapsto &\ [(m_{11}+m_{21},m_{12}+m_{22})]-[(m_{11},m_{12})]-[(m_{21},m_{22})]\\
-&\ [(m_{11}+m_{12},m_{21}+m_{22})]+[(m_{11},m_{21})]+[(m_{12},m_{22})]
\end{aligned}\]
and in general by the obvious generalization involving $3n$ summands (in each of the $n$ directions, take the difference between summing opposite faces internally, and summing opposite faces externally). Rather than writing out the differentials, we note that they are uniquely determined by the following condition:

Consider the following two maps of complexes $\sigma_1,\sigma_2: Q(M^2)\to Q(M)$. The map $\sigma_1$ is $Q$ applied to the sum map $M^2\to M$, and $\sigma_2$ is the sum of the two maps induced by applying $Q$ to the two projections $M^2\to M$. Then these $\sigma_1$ and $\sigma_2$ are homotopic via the chain homotopy that is the identity map $\mathbb Z[M^{2^n}]\to \mathbb Z[M^{2^n}]$ in all degrees. This is the content of the following lemma.

\begin{lemma}\label{lem:basehomotopy} There is a unique definition of a functorial complex
\[
Q(M): \ldots\to  \mathbb Z[M^{2^n}]\to\ldots\to\mathbb Z[M^2]\to\mathbb Z[M]
\]
with the following property. For an abelian group $M$, the maps $\sigma_1,\sigma_2$ from
\[
Q(M^2): \ldots \to \mathbb Z[M^{2^{n+1}}]\to\ldots\to\mathbb Z[M^4]\to\mathbb Z[M^2]
\]
to
\[
Q(M): \ldots \to \mathbb Z[M^{2^n}]\to\ldots\to\mathbb Z[M^2]\to\mathbb Z[M],
\]
induced by addition $M^2\to M$, respectively the sum of the two maps induced by two projections $M^2\to M$, are homotopic, via the homotopy
\[
h_n: \mathbb Z[M^{2^{n+1}}]\to \mathbb Z[M^{2^{n+1}}]
\]
that is $h_n=\mathrm{id}$ for all $n\geq 0$.

If $M$ is a pseudo-normed abelian group object in any topos, then
\[
Q(M)_{\leq c}: \ldots\to \mathbb Z[M^{2^n}_{\leq 2^{-n} c}]\to \ldots\to \mathbb Z[M^2_{\leq 2^{-1} c}]\to \mathbb Z[M_{\leq c}]
\]
is a well-defined complex, and $\sigma_1$ and $\sigma_2$ are well-defined as maps of complexes from
\[
Q(M^2)_{\leq c/2}: \ldots \to \mathbb Z[M^{2^{n+1}}_{\leq 2^{-n-1}c}]\to\ldots\to\mathbb Z[M^4_{\leq 2^{-2} c}]\to\mathbb Z[M^2_{\leq 2^{-1} c}]
\]
to
\[
Q(M)_{\leq c}: \ldots \to \mathbb Z[M^{2^n}_{\leq 2^{-n}c}]\to\ldots\to\mathbb Z[M^2_{\leq 2^{-1} c}]\to\mathbb Z[M_{\leq c}]
\]
for all $c>0$. Moreover, $h_\bullet$ is also defined as a homotopy from $Q(M^2)_{\leq c/2}$ to $Q(M)_{\leq c}$.
\end{lemma}

\begin{proof} It is clear that the defining property gives an inductive definition of the differentials, and their inductive definition makes it clear that they map $Q(M^2)_{\leq c/2}$ to $Q(M)_{\leq c}$.
\end{proof}

\begin{lemma} The functor $M\mapsto H_i(Q(M))$ preserves direct sums and filtered colimits. In particular, for any torsion-free abelian group, there is a natural isomorphism
\[
H_i(Q(M))\cong H_i(Q(\mathbb Z))\otimes M.
\]
\end{lemma}

\begin{proof} Commutation with filtered colimits is clear. For additivity, note that the functor $M\mapsto H_i(Q(M))$ is enriched in abelian groups, by using the homotopy defining $Q$. Thus, it preserves direct sums.
\end{proof}

Now we need the following generalization of Lemma~\ref{lem:basehomotopy} to adding $N$ elements.

\begin{lemma}\label{lem:homotopyNelements} Let $N$ be a power of $2$. The maps of complexes $\sigma_{1,N},\sigma_{2,N}$ from
\[
Q(M^N): \ldots \to \mathbb Z[M^{2^n N}]\to\ldots\to\mathbb Z[M^{2N}]\to\mathbb Z[M^{N}]
\]
to
\[
Q(M): \ldots \to \mathbb Z[M^{2^n}]\to\ldots\to\mathbb Z[M^2]\to\mathbb Z[M],
\]
induced by addition $M^N\to M$, respectively the sum of the $N$ maps induced by the $N$ projections $M^N\to M$, are homotopic, via some functorial homotopy
\[
h_n^N: \mathbb Z[M^{2^n N}]\to \mathbb Z[M^{2^{n+1}}]
\]
which moreover satisfies the following bound:

If $M$ is a pseudo-normed abelian group object in any topos, then $\sigma_{1,N}$ and $\sigma_{2,N}$ are well-defined as maps of complexes from
\[
Q(M^N)_{\leq c/N}: \ldots \to \mathbb Z[M^{2^n N}_{\leq 2^{-n}c/N}]\to\ldots\to\mathbb Z[M^{2N}_{\leq 2^{-1} c/N}]\to\mathbb Z[M^N_{\leq c/N}]
\]
to
\[
Q(M)_{\leq c}: \ldots \to \mathbb Z[M^{2^n}_{\leq 2^{-n}c}]\to\ldots\to\mathbb Z[M^2_{\leq 2^{-1} c}]\to\mathbb Z[M_{\leq c}]
\]
for all $c>0$. In that case, $h_n^N$ defines well-defined maps
\[
\mathbb Z[M^{2^n N}_{\leq 2^{-n} c/N}]\to \mathbb Z[M^{2^{n+1}}_{\leq 2^{-n-1} c}]
\]
for all $c>0$.
\end{lemma}

\begin{proof} Let $N=2^m$. For each $j=0,\ldots,m-1$, the two maps from $C(M^{2^{j+1}})$ to $C(M^{2^j})$ from the previous lemma are homotopic, and we use the homotopy from that lemma. Composing homotopies (which amounts concretely to a certain sum) we get the desired homotopy from $C(M^{2^m})$ to $C(M)$. The bound on the homotopy follows from this formula.
\end{proof}

\newpage

\section{Lecture X: Some computations with liquid modules}

Now we have done the hard work to prove that there is a well-behaved category of liquid real vector spaces. In the rest of the course, we want to define analytic spaces, and show that various classical objects, like complex-analytic spaces, or rigid-analytic varieties, give examples.

In this lecture, we some basic computations in the category of liquid real vector spaces, or liquid $\mathbb Z((T))_{>r}$-modules; some of these computations will also be necessary for the discussion of complex-analytic spaces.

\begin{proposition}\label{prop:computeliquidification} Fix any radius $0<r<1$.
\begin{enumerate}
\item For any compact Hausdorff space $S$, let
\[
\mathcal M(S,\mathbb Z((T))_r)\subset \prod_{n\in \mathbb Z}\mathbb Z[S]\cdot T^n
\]
be the union over all $c>0$ of the (compact Hausdorff) subspace $\mathcal M(S,\mathbb Z((T))_r)_{\leq c}$ whose $S'$-valued points are all maps $S'\to \prod_{n\in \mathbb Z}\mathbb Z[S]\cdot T^n$ such that for all $s'\in S'$, the induced element $(\mu_n\cdot T^n)_n$, $\mu_n\in \mathbb Z[S]$ satisfies $\mu_n\in \mathbb Z[S]_{\leq a_n}$ for some nonnegative integers $a_n$ such that $\sum a_n r^n\leq c$.

This definition agrees with the previous one for profinite sets $S$, and in general if $S_\bullet\to S$ is a hypercover of a compact Hausdorff $S$ by profinite sets $S_i$, then
\[
\mathcal M(S_\bullet,\mathbb Z((T))_r)\to \mathcal M(S,\mathbb Z((T))_r)
\]
is a resolution. The derived $r$-liquidification of $\mathbb Z[T^{-1}][S]$ is given by
\[
\mathcal M(S,\mathbb Z((T))_{>r}) = \varinjlim_{r'>r} \mathcal M(S,\mathbb Z((T))_{r'}).
\]
\item For any compact abelian group $A$, the derived $r$-liquidification of $A[T^{-1}]$ is given by $A((T))$.
\end{enumerate}
\end{proposition}

By the first part, we can also define $\mathcal M_p(S)$ via base change to $\mathbb R$ for general compact Hausdorff $S$; equivalently, via resolving $S$ by profinite sets $S_i$ and forming the complex $\mathcal M_p(S_\bullet)$, which will be quasi-isomorphic to $\mathcal M_p(S)$ in degree $0$. Then the derived $p$-liquidification of $\mathbb R[S]$ is equal to $\mathcal M_{<p}(S)$, and in particular sits in degree $0$. We note that it seems slightly tricky to do these things directly over $\mathbb R$.

\begin{proof} For the first part, note that for profinite $S$, this is essentially the definition, with Proposition~\ref{prop:freecondensedgroup}. For general compact Hausdorff $S$, consider a simplicial resolution of $S$ by profinite sets $S_i$, and the corresponding resolution
\[
\ldots\to \mathbb Z[S_1]\to \mathbb Z[S_0]\to \mathbb Z[S]\to 0.
\]
We consider this as a complex of pseudonormed abelian groups, with the subgroups $\mathbb Z[S]_{\leq n}$. We claim that whenever $x\in \mathbb Z[S_i]_{\leq n}$ with $d(x)=0$ then there is some $y\in \mathbb Z[S_{i+1}]_{\leq n}$ with $x=d(y)$. Note that this statement depends only on the underlying sets of $S$ and the $S_i$, and the hypercover $S_\bullet\to S$ of sets splits; this produces a contracting homotopy which easily gives the statement (the contracting homotopy is like in the second appendix to Lecture VIII). From here it is easy to see that also
\[
\mathcal M(S_\bullet,\mathbb Z((T))_r)\to \mathcal M(S,\mathbb Z((T))_r)
\]
is a resolution. Passing to the colimit over $r'>r$ gives the result (as $\mathcal M(S_\bullet,\mathbb Z((T))_{>r})$ is the derived $r$-liquidification of $\mathbb Z[T^{-1}][S_i]$, and thus the complex $\mathcal M(S_\bullet,\mathbb Z((T))_{>r})$ computes the derived $r$-liquidification of $\mathbb Z[T^{-1}][S]$).

For the second part, we need to see that $A((T))$ is $r$-liquid and that derived $r$-liquidification of $B=A((T))/A[T^{-1}]$ vanishes. It is easy to see that any map $S\to A((T))$ extends to a map $\tilde{f}: \mathcal M(S,\mathbb Z((T))_r)\to A((T))$; as $A((T))$ is quasiseparated, uniqueness is automatic, so $A((T))$ is $r$-liquid.

Now note that $B=A((T))/A[T^{-1}]$ is a compact abelian group with $\mathbb Z[T^{-1}]$-module structure. It can be resolved via the Breen--Deligne resolution, so it suffices to see that for any compact Hausdorff space $S$ with an endomorphism $T^{-1}: S\to S$, the derived $r$-liquidification of $\mathbb Z[S]$ (considered as $\mathbb Z[T^{-1}]$-module) vanishes. This can be described as the cone of $T^{-1}-[T^{-1}]$ on the derived $r$-liquidification of $\mathbb Z[T^{-1}][S]$, the latter of which is described by the first part. Now $T^{-1}-[T^{-1}]$ has the inverse $T+[T^{-1}] T^2 + [T^{-1}]^2 T^3 + \ldots$, giving the result.
\end{proof}

The following corollary generalizes \cite[Theorem 4.3 (ii)]{Condensed}.

\begin{corollary} Let $V$ be a $p$-liquid $\mathbb R$-vector space and let $A$ be a compact abelian group. Then $R\intHom(A,V)=0$. Equivalently, the derived $p$-liquidification of $A$ vanishes.
\end{corollary}

\begin{proof} Write $\mathbb R=\mathbb Z((T))_{>r}/(2-T^{-1})$ where $r=2^{-p}$. Then the derived $p$-liquidification of $A$ is the cone of $2-T^{-1}$ on the derived $r$-liquidification of $A[T^{-1}]$. The latter is $A((T))$ by part (2) of the proposition, and $2-T^{-1}$ is invertible on this (with inverse $-T-2T^2-4T^3-\ldots$).
\end{proof}

\begin{remark} In the proof of \cite[Theorem 4.3 (ii)]{Condensed}, we also used the action of $2-[2]$ on a Breen--Deligne resolution; this seems closely related to writing $\mathbb R=\mathbb Z((T))_{>r}/(2-T^{-1})$.
\end{remark}

The corollary implies formally that if $V$ is a $p$-liquid $\mathbb R$-vector space considered as a trivial representation of a compact abelian group $A$, then the group cohomology $H^i(A,V)$ vanishes for $i>0$. (Indeed, by general nonsense there is a spectral sequence starting with $\Ext^i(\bigwedge^j A,V)$ converging to $H^{i+j}(A,V)$.) The proof uses commutativity critically.

\begin{question} Let $G$ be a compact nonabelian group, say $G=\mathrm{SU}(2)$, and let $V$ be a $p$-liquid $\mathbb R$-vector space, e.g.~a $p$-Banach. Does $H^i(G,V)$ vanish for $i>0$?
\end{question}

Here is a related question.

\begin{question} Consider $\mathbb C$ as a condensed ring, and correspondingly the algebraic $K$-theory $K(\mathbb C)$ as a condensed spectrum (sending an extremally disconnected $S$ to $K(C(S,\mathbb C))$). What is the $p$-liquidification of $K(\mathbb C)$? What is the $r$-liquidification of $K(\mathbb C)[T^{-1}]$?\footnote{For the latter, we remark that one can define liquidification also over $\mathbb S((T))_{>r}$ where $\mathbb S$ is the sphere spectrum.}
\end{question}

We note that the solidification can be computed, cf.~\cite{AokiSolidK}:\footnote{Again, solidification can be defined over $\mathbb S$.}

\begin{proposition} The solidification of $K(\mathbb C)$ is $\mathrm{ku}$, connective topological $K$-theory.
\end{proposition}

\begin{proof} The essential point is that $\mathbb S[\mathrm{GL}_n(\mathbb C)]^\solid\cong \mathbb S[|\mathrm{GL}_n(\mathbb C)|]$ where $|\mathrm{GL}_n(\mathbb C)|$ is the corresponding anima (cf.~\cite[Example 6.5]{Condensed}).
\end{proof}

Of course, solidification forgets all information about real vector spaces, and we want to recover those via $p$-liquidification. This has the slightly annoying feature that, at least a priori, it depends on $p$. The $r$-liquidification of $K(\mathbb C)[T^{-1}]$ has the virtue that it specializes to all possible $p$-liquidifications via specializing $T$ to different values; and it remembers the integral information. We note that for $K_1(\mathbb C)=\mathbb C^\times$, the $r$-liquidification of $\mathbb C^\times[T^{-1}]$ can be computed: Using the decomposition $\mathbb C^\times\cong \mathbb R\times \mathbb R/\mathbb Z$, it is
\[
\mathbb C^\times((T))_{>r} = \mathbb R((T))_{>r}\times \mathbb R/\mathbb Z((T)).
\]

On the other hand, let us try to understand the structure of $\mathcal M(S,\mathbb Z((T))_r)$ and $\mathcal M_p(S)$ better.

\begin{proposition} Let $S$ be a compact Hausdorff space and $0<p<1$. Then for any $\mu\in \mathcal M_{p}(S)$ (element of the underlying set) there is a sequence $s_0,s_1,\ldots$ in $S$ and real numbers $x_0,x_1,\ldots$ with $\sum |x_n|^{p}<\infty$, such that
\[
\mu=\sum_{n=0}^\infty x_n [s_n].
\]

Similarly, for $0<r<1$ and any $\mu\in \mathcal M(S,\mathbb Z((T))_r)$, there is a sequence $s_0,s_1,\ldots\in S$ and elements $x_0,x_1,\ldots\in \mathbb Z((T))_r$ such that $x_n\in \mathbb Z((T))_{r,\leq c_n}$ with $\sum c_n<\infty$ and
\[
\mu=\sum_{n=0}^\infty x_n [s_n].
\]
\end{proposition}

Thus, all $p$-measures are just countable sums of Dirac measures. In particular, Haar measures which are ``equidistributed'' over the whole space, will never lie in $\mathcal M_{p}(S)$ for $p<1$ (except for finite groups). This phenomenon, of measures being countable sums of Dirac measures, actually happens more generally for measures ``of bounded entropy'', and this is the minimal condition that the Ribe extension forces on us.

We note that this is only a description of the underlying set: As a condensed set, $\mathcal M_{p}(S)$ is not a filtered colimit over its subspaces for (closures of) countable subspaces of $S$. Indeed, the map $S\to \mathcal M_{p}(S)$ does not factor over such a subspace in general.

\begin{proof} It suffices to handle the assertion over $\mathbb Z((T))_r$; the case of $\mathcal M_p(S)$ follows via base change. But then there is some $c<\infty$ such that
\[
\mu\in \mathcal M(S,\mathbb Z((T))_r)_{\leq c}\subset \prod_n \mathbb Z[S]\cdot T^n,
\]
and each coefficient of $T^n$ involves only finitely many elements of $S$. It follows that one can write $\mu = \sum_{n=0}^\infty x_n [s_n]$ for some countable sequence of elements $s_n\in S$, which we assume has no repetitions. The condition $\mu\in \mathcal M(S,\mathbb Z((T))_r)$ then amounts to the condition that if $x_n\in \mathbb Z((T))_{r,\leq c_n}$ with $c_n$ chosen minimal, then $\sum c_n\leq c$. This gives the result.
\end{proof}

Finally, we use the formalism in the simplest situation of relevance to complex-analytic geometry.

\begin{proposition} Let $V$ be the condensed $\mathbb C$-vector space of overconvergent holomorphic functions on the closed unit disc $\{|z|\leq 1\}$. Equivalently,
\[
V=\varinjlim_{r>1} \{\sum a_n z^n\in \mathbb C[[z]]\mid \sum |a_n|r^n<\infty\}
\]
where each term in the filtered colimit is regarded as a Banach space. Then for any $0<p<1$, the tensor product
\[
V\otimes^{L\mathrm{liq}}_{\mathbb C,p} V
\]
is given by
\[
\varinjlim_{r>1} \{\sum a_n b_m z^n w^m\in \mathbb C[[z,w]]\mid \sum |a_nb_m|r^{n+m}<\infty\}
\]
of overconvergent holomorphic functions on the $2$-dimensional disc $\{|z|,|w|\leq 1\}$.
\end{proposition}

\begin{proof} Decompose $\mathcal M_{<p}(\mathbb N\cup\{\infty\}) = W\oplus \mathbb R\cdot [\infty]$. Then $W$ is the union of the spaces $W_{p'}$ for $p'<p$ whose compact subspaces are given by the space of sequences $(x_0,x_1,\ldots)\in \mathbb R$ with $\sum |x_i|^{p'}\leq c$. The key observation is now that
\[
V = \varinjlim_{r>1} W_{\mathbb C}
\]
along the maps $f_r: W_{\mathbb C}\to V$ sending $(x_0,x_1,\ldots)$ to $\sum x_n r^{-n} z^n$; this is a simple verification, based on the observation that multiplication by the sequence $(1,t,t^2,\ldots)$ with $t<1$ takes bounded sequences to $\ell^p$-summable sequences for all $0<p<1$.

But now the liquid tensor product of $\mathcal M_{<p}(\mathbb N\cup \{\infty\})$ with itself is
\[
\mathcal M_{<p}((\mathbb N\cup \{\infty\})\times (\mathbb N\cup \{\infty\}))
\]
of which $W\otimes^{L\mathrm{liq}}_p W$ is the direct factor whose compact subspaces are the space of sequences $(x_{nm})_{n,m\geq 0}\in \mathbb R$ with $\sum |x_{n,m}|^{p'}\leq c$ for some $p'<p$ and $c>0$. This gives the result by passing to the colimit over $r$.
\end{proof}

\begin{corollary} Let $D_1,D_2\subset \mathbb C$ be two closed discs with empty intersection. Let $\mathcal O(D_1)$ and $\mathcal O(D_2)$ be the spaces of overconvergent holomorphic functions, considered as condensed $\mathbb C[z]$-modules. Then
\[
\mathcal O(D_1)\otimes^{L\mathrm{liq}}_{\mathbb C[z],p} \mathcal O(D_2) = 0.
\]
\end{corollary}

\begin{proof} Forming the liquid tensor product over $\mathbb C$, one gets $\mathcal O(D_1\times D_2)$ by the previous proposition. On this space, $z_1-z_2$ is invertible, so the liquid tensor product over $\mathbb C[z]$ vanishes.
\end{proof}

We see that for these questions, the choice of $p$ is largely irrelevant, and that one gets the expected tensor products. In \cite{Complex} we develop much of the basic theory of complex-analytic geometry from scratch, using the foundation of liquid vector spaces.

\newpage

\section{Lecture XI: Animation}

In the remaining lectures, we will define a category of analytic spaces, and give various examples.\footnote{A historical comment: With Clausen, we first attempted to find an analytic version of the category of schemes before finding an analytic version of the category of stacks. In the end, we were never happy with any of our attempts to define such a category of analytic spaces, but are happy with the definition of analytic stacks in \cite{AnalyticStacks}. Also, many of the applications of the formalism really make use of analytic stacks. Thus, the theory developed in the remaining lectures should be seen as a historical artifact.} Let us at this point remark that the resulting theory of analytis spaces will, notably, be an absolute theory: Usually, one defines complex-analytic spaces or rigid-analytic varieties over some nonarchimedean field $K$, so works over a fixed base field (and works with $K$-algebras satisfying some topological finiteness condition). Here, just like with schemes, it will not be necessary to specify any base in advance.\footnote{Let us mention that Ben-Bassat, Kremnitzer, and coauthors, have also recently undertaken a general study of analytic spaces, including some notion of quasicoherent sheaves, cf.~e.g.~\cite{BenBassatKremnizer}. Their theory is based on working in the symmetric monoidal category of ind-Banach spaces, roughly in the manner paraphrased in the first lecture (that is the ``relative algebraic geometry'' of To\"en--Vaqui\'e--Vezzosi).}

There are two ways to define schemes: As locally ringed topological spaces, or as certain functors on the category of rings. In the latter approach, the essential non-formal input is just the notion of a localization $A\to A[f^{-1}]$ of rings; from the category of rings with this class of morphisms, one can recover the category of schemes. Indeed, schemes will be a full subcategory of the category of functors from rings to sets; in particular, we ask that they are Zariski sheaves (where a finite family of localizations $A\to A[f_i^{-1}]$ is a cover if the base change functor from $A$-modules to $\prod_i A[f_i^{-1}]$-modules is faithful). An affine scheme is a functor of the form $F_A: B\mapsto \Hom(A,B)$ for some ring $A$. A map $F\to G$ of Zariski sheaves is an open immersion if it is injective and for every ring $B$ with a section $s\in G(B)$, corresponding to a map $F_B\to G$, the fibre product $F\times_G F_B$ is the colimit of $F_A$ over all localizations $B\to A$ for which $F_A\to F_B$ factors over $F\times_G F_B$. Finally, a scheme is a functor $F$ that admits an open cover by $F_A$'s.

We will follow this route for the definition of analytic spaces. Thus, the next goal is to define a notion of localization $(A,\mathcal M)\to (B,\mathcal N)$ of analytic rings. The notion is supposed to have the following properties:
\begin{enumerate}
\item It is stable under base change.
\item It is stable under composition.
\item It is stable under filtered colimits.
\end{enumerate}
The first condition actually requires some explanation; we will come to this later. Regarding the third condition, note that this means that in the usual algebraic context we allow general localizations $A\to A[S^{-1}]$ at multiplicative subsets. It turns out that in the analytic category, there is no reasonable ``finiteness'' one can a priori impose on the localizations, so we will allow general ones.

\begin{exercise} Verify that the above formal definition of schemes also gives a well-defined notion when one allows general localizations $A\to A[S^{-1}]$ of rings.
\end{exercise}

Let us first give several examples.

\begin{example}\leavevmode
\begin{enumerate}
\item If $A\to B=A[S^{-1}]$ is a usual localization of (discrete) rings, with trivial analytic ring structure, then $A\to B$ will be allowed.
\item The map $(\mathbb Z[T],\mathbb Z)_\solid\to (\mathbb Z[T],\mathbb Z[T])_\solid$, corresponding to localization to the subset $\{|T|\leq 1\}$ of the adic space $\Spa(\mathbb Z[T],\mathbb Z)$.
\item More generally, if $X=\Spa(A,A^+)$ is a (reasonable) adic space and $U\subset X$ is a rational subset, the map
\[
(A,A^+)_\solid\to (\mathcal O_X(U),\mathcal O_X^+(U))_\solid .
\]
\item In the context of the last lecture, if $D_1\subset D_2$ is an inclusion of closed discs in $\mathbb C$, then $\mathcal O(D_2)\to \mathcal O(D_1)$ (endowed with the analytic ring structure induced from the $p$-liquid structure on $\mathbb C$ for some chosen $0<p\leq 1$) is a localization.
\end{enumerate}
\end{example}

The parenthetical word ``reasonable'' occurs here as general adic spaces are not well-behaved: For example, the structure presheaf may fail to be a sheaf. This failure will actually be corrected by the formalism of analytic spaces, and the correction will be by slightly changing the localizations in general.

For example, consider the case $U=\{|f|\leq 1\}$ for some $f\in A$. Then $\mathcal O_X(U)$ is by definition the quotient of the convergent power series algebra $A\langle T\rangle$ by the closure of the ideal generated by $T-f$. On the other hand, in practice this ideal is already closed and moreover $T-f$ is a non-zerodivisor (in fact, by \cite[Lemma 2.4.10]{KedlayaLiu2}, this is automatic if $(A,A^+)$ is sheafy and $A$ has a topologically nilpotent unit, for example is an algebra over a nonarchimedean field). Thus, in this case
\[
\mathcal O_X(U) = A\langle T\rangle/(T-f) = A\langle T\rangle/^{\mathbb L} (T-f) = [A\langle T\rangle\xrightarrow{T-f} A\langle T\rangle].
\]

On the other hand, in general both of these properties can fail. Note that algebraically, the most natural object is the complex on the right, and indeed this will be the localization in the world of analytic rings. Working with topological rings, it is impossible to work with the quotient $A\langle T\rangle / (T-f)$ if the ideal $(T-f)$ is not closed. However, passing to the condensed world, the quotient $\underline{A\langle T\rangle}/(T-f)$ is perfectly well-behaved. If $T-f$ is a zerodivisor, it however becomes necessary to work even with the derived reduction $\underline{A\langle T\rangle}/^{\mathbb L} (T-f)$: A condensed \emph{animated} ring.

We had seen last semester that localizations will in general not be flat, and turn modules concentrated in degree $0$ into complexes (in nonnegative homological degrees). The same will happen in general for localizations of analytic rings, so we need to also derive our rings. This is the subject of derived algebraic geometry, which replaces the usual category of rings by its animation, the $\infty$-category of animated rings.\footnote{We refer to the work of To\"en--Vezzosi, \cite{ToenVezzosi}, and Lurie, \cite{LurieSAG}, for developments of derived algebraic geometry. Note that these sources say ``simplicial ring'' when they mean the concept of what we call an ``animated ring''.} In the rest of this lecture, we will recall the basic theory of ``animation'' and apply it in condensed mathematics.

\subsection{Animation} Let us briefly recall the notion of animation.\footnote{In classical language, this is a non-abelian derived category in the sense of Quillen. It has also been studied by Rosicky, \cite{Rosicky}, and we follow Lurie, \cite[Section 5.5.8]{LurieHTT}, with language that has been coined by Clausen inspired by Beilinson's \cite{BeilinsonTopEpsilon}, and used first in \cite{CesnaviciusScholze}.} Let $\mathcal C$ be a category that admits all small colimits. Recall that an object $X\in \mathcal C$ is compact (also called finitely presented) if $\Hom(X,-)$ commutes with filtered colimits. An object $X\in\mathcal C$ is projective if $\Hom(X,-)$ commutes with reflexive coequalizers, i.e.~colimits along $\Delta_{\leq 1}^{\mathrm{op}}$ (coequalizers of parallel arrows $Y\rightrightarrows Z$ with a simultaneous section $Z\to Y$ of both maps). Taken together, an object $X\in\mathcal C$ is compact projective if $\Hom(X,-)$ commutes with all filtered colimits and reflexive coequalizers; equivalently, it commutes with all (so-called) $1$-sifted colimits.

Let $\mathcal C^{\mathrm{cp}}\subset \mathcal C$ be the full subcategory of compact projective objects. There is a fully faithful embedding $\mathrm{sInd}(\mathcal C^{\mathrm{cp}})\to \mathcal C$ from the $1$-sifted Ind-category of $\mathcal C^{\mathrm{cp}}$ (the full subcategory of $\Fun((\mathcal C^{\mathrm{cp}})^{\mathrm{op}},\mathrm{Set})$ generated under small $1$-sifted colimits by the Yoneda image; equivalently, the category freely generated under small $1$-sifted colimits by $\mathcal C^{\mathrm{cp}}$). If $\mathcal C$ is generated under small colimits by $\mathcal C^{\mathrm{cp}}$, then this functor is an equivalence
\[
\mathrm{sInd}(\mathcal C^{\mathrm{cp}})\cong \mathcal C.
\]
If $\mathcal C^{\mathrm{cp}}$ is small, then
\[
\mathrm{sInd}(\mathcal C^{\mathrm{cp}})\subset \Fun((\mathcal C^{\mathrm{cp}})^{\mathrm{op}},\mathrm{Set})
\]
is exactly the full subcategory of functors that take finite coproducts in $\mathcal C^{\mathrm{cp}}$ to products in $\mathrm{Set}$.

\begin{example}\leavevmode
\begin{enumerate}
\item If $\mathcal C=\mathrm{Set}$ is the category of sets, then $\mathcal C^{\mathrm{cp}}$ is the category of finite sets, which generates $\mathcal C$ under small colimits.
\item If $\mathcal C=\mathrm{Ab}$ is the category of abelian groups, then $\mathcal C^{\mathrm{cp}}$ is the category of finite free abelian groups, which generates $\mathcal C$ under small colimits.
\item If $\mathcal C=\mathrm{Ring}$ is the category of commutative rings, then $\mathcal C^{\mathrm{cp}}$ is the category of retracts of polynomial rings $\mathbb Z[X_1,\ldots,X_n]$, which generates $\mathcal C$ under small colimits.
\item If $\mathcal C=\mathrm{Cond}(\mathrm{Set})$ is the category of condensed sets, then $\mathcal C^{\mathrm{cp}}$ is the category of extremally disconnected profinite sets, which generates $\mathcal C$ under small colimits.\footnote{In this example (and the ones to follow), $\mathcal C^{\mathrm{cp}}$ is itself a large category.}
\item If $\mathcal C=\mathrm{Cond}(\mathrm{Ab})$ is the category of condensed abelian groups, then $\mathcal C^{\mathrm{cp}}$ is the category of direct summands of $\mathbb Z[S]$ for extremally disconnected $S$, which generates $\mathcal C$ under small colimits.
\item If $\mathcal C=\mathrm{Cond}(\mathrm{Ring})$ is the category of condensed rings, then $\mathcal C^{\mathrm{cp}}$ is the category of retracts of $\mathbb Z[\mathbb N[S]]$ for extremally disconnected $S$, where $\mathbb N[S]$ is free condensed abelian monoid on $S$ and thus $\mathbb Z[\mathbb N[S]]$ is the free condensed ring on $S$. Again, $\mathcal C^{\mathrm{cp}}$ generates $\mathcal C$ under small colimits.
\end{enumerate}
\end{example}

\begin{definition} Let $\mathcal C$ be a category that admits all small colimits and is generated under small colimits by $\mathcal C^{\mathrm{cp}}$. The animation of $\mathcal C$ is the $\infty$-category $\mathrm{Ani}(\mathcal C)$ freely generated under sifted colimits by $\mathcal C^{\mathrm{cp}}$.
\end{definition}

\begin{example} If $\mathcal C=\mathrm{Set}$, then $\mathrm{Ani}(\mathcal C)=\mathrm{Ani}(\mathrm{Set})=:\mathrm{Ani}$ is the $\infty$-category of animated sets, or anima for brevity. In standard language, this is the $\infty$-category of ``spaces''.

Let us describe the nature of the $\infty$-category of anima. Any anima has a set of connected components, giving a functor $\pi_0:\mathrm{Ani}\to \mathrm{Set}$ (in fact, $\pi_0$ is simply given by the universal property of $\mathrm{Ani}$, as $\mathrm{Set}$ is a category with all sifted colimits with a functor from finite sets), which has a fully faithful right adjoint $\mathrm{Set}\hookrightarrow \mathrm{Ani}$. Given an anima $A$ with a point $a\in A$ (meaning a map $a: \ast\to A$), one can define groups $\pi_i(A,a)$ for $i\geq 1$, which are abelian for $i\geq 2$. The map $a: \ast\to A$ is an equivalence if and only if $\pi_0 A$ is a point and $\pi_i(A,a)=0$ for all $i\geq 1$. An anima is defined to be $i$-truncated if $\pi_j(A,a)=0$ for all $a\in A$ and $j>i$. Then $A$ is $0$-truncated if and only if it is in the essential image of $\mathrm{Set}\hookrightarrow \mathrm{Ani}$. The inclusion of $i$-truncated anima into all anima has a left adjoint $\tau_{\leq i}$. For all anima $A$, the natural map
\[
A\to \lim_i \tau_{\leq i} A
\]
is an equivalence; this is the ``convergence of the Postnikov tower''. Picking any $a\in A$ and $i\geq 1$, the fibre of $\tau_{\leq i} A\to \tau_{\leq i-1} A$ over the image of $a$ is an Eilenberg-MacLane anima $K(\pi_i(A,a),i)$. Here, an Eilenberg-MacLane anima $K(\pi,i)$, with $i\geq 1$ an integer and $\pi$ a group that is abelian if $i>1$, is a pointed connected anima with $\pi_j=0$ for $j\neq i$ and $\pi_i = \pi$. It is unique up to unique isomorphism. In fact, the $\infty$-category of pointed connected anima $(A,a)$ with $\pi_j(A,a)=0$ for $j\neq i$ is equivalent to the category of groups when $i=1$, and to the category of abelian groups when $i\geq 2$.
\end{example}

There are several ways to construct $\mathrm{Ani}(\mathcal C)$. It can be defined as the full sub-$\infty$-category of
\[
\mathrm{Fun}((\mathcal C^{\mathrm{cp}})^{\mathrm{op}},\mathrm{Ani})
\]
generated under sifted colimits by the Yoneda image. If $\mathcal C^{\mathrm{cp}}$ is small, this agrees with the full sub-$\infty$-category of all contravariant functors $F: \mathcal C^{\mathrm{cp}}\to \mathrm{Ani}$ that take finite coproducts to products. On the other hand, recall that all sifted colimits are generated by filtered colimits and geometric realizations of simplicial objects. This makes simplicial objects central. In fact, $\mathrm{Ani}$ can be defined as the $\infty$-category obtained from the category of simplicial sets $\mathrm{sSet}$ by inverting weak equivalences. Similarly, for any $\infty$-category $\mathcal C$ generated under small colimits by $\mathcal C^{\mathrm{cp}}$, one can describe $\mathrm{Ani}(\mathcal C)$ as the $\infty$-category obtained from simplicial objects in $\mathcal C$ by inverting weak equivalences.\footnote{We regard it thus as a misnomer to call objects of $\mathrm{Ani}(\mathrm{Ring})$ simplicial rings: One has changed the morphisms drastically.}

\begin{example}\leavevmode
\begin{enumerate}
\item As indicated above, for $\mathcal C=\mathrm{Set}$, one gets the $\infty$-category of anima (a.k.a.~spaces).
\item For $\mathcal C=\mathrm{Ab}$, the $\infty$-category $\mathrm{Ani}(\mathrm{Ab})$ of animated abelian groups is, by the Dold-Kan equivalence, equivalent to the $\infty$-derived category of abelian groups $\mathcal D_{\geq 0}(\mathrm{Ab})$ in nonnegative homological degrees. This motivates the term ``nonabelian derived category'' for the general construction of animation.
\item For $\mathcal C=\mathrm{Ring}$, one gets the $\infty$-category of animated rings.
\item For $\mathcal C=\mathrm{Cond}(\mathrm{Set})$, one gets the $\infty$-category $\mathrm{Ani}(\mathrm{Cond}(\mathrm{Set}))$ of animated condensed sets. The two operations $\mathrm{Ani}(-)$ and $\mathrm{Cond}(-)$ actually commute by Lemma~\ref{lem:condensedanimation} below, so this is equivalently the $\infty$-category $\mathrm{Cond}(\mathrm{Ani}(\mathrm{Set}))=\mathrm{Cond}(\mathrm{Ani})$ of condensed anima.
\item For $\mathcal C=\mathrm{Cond}(\mathrm{Ab})$, one gets, by Dold-Kan again, the $\infty$-derived category $\mathcal D_{\geq 0}(\mathrm{Cond}(\mathrm{Ab}))$ of condensed abelian groups in nonnegative homological degrees.
\item For $\mathcal C=\mathrm{Cond}(\mathrm{Ring})$, one gets the $\infty$-category of animated condensed rings, or equivalently of condensed animated rings.
\end{enumerate}
\end{example}

The operations of animation and passage to condensed objects commute:

\begin{definition} Let $\mathcal C$ be an $\infty$-category that admits all small colimits. For any uncountable strong limit cardinal $\kappa$, the $\infty$-category $\mathrm{Cond}_\kappa(\mathcal C)$ of $\kappa$-condensed objects of $\mathcal C$ is the category of contravariant functors from $\kappa$-small extremally disconnected profinite sets $S$ to $\mathcal C$ that take finite coproducts to products.

Moreover, using the fully faithful left adjoints to the forgetful functors,
\[
\Cond(\mathcal C):=\varinjlim_\kappa \mathrm{Cond}_\kappa(\mathcal C).
\]
\end{definition}

\begin{lemma}\label{lem:condensedanimation} Let $\mathcal C$ be a category that is generated under small colimits by $\mathcal C^{\mathrm{cp}}$. Then $\mathrm{Cond}(\mathcal C)$ is still generated under small colimits by its compact projective objects, and there is a natural equivalence of $\infty$-categories
\[
\mathrm{Cond}(\mathrm{Ani}(\mathcal C))\cong \mathrm{Ani}(\mathrm{Cond}(\mathcal C)).
\]
\end{lemma}

\begin{proof} A family of compact projective generators of $\mathrm{Cond}(\mathcal C)$ is parametrized by pairs $(S,X)$ of an extremally disconnected $S$ and a compact projective object $X\in \mathcal C^{\mathrm{cp}}$, given by the sheafification of sending $T$ to the colimit over $\Hom(T,S)$ many copies of $X$. These objects still define compact projective objects of the $\infty$-category $\mathrm{Cond}(\mathrm{Ani}(\mathcal C))$, and generate the latter under small colimits. This implies the desired result.
\end{proof}

We will be interested in condensed animated rings. However, let us make some remarks about the nature of the $\infty$-category of condensed anima, as this combines two different flavours of topology.

The traditional approach to homotopy theory is to define the $\infty$-category of anima, or ``spaces'', by taking the category of CW complexes, and inverting weak equivalences. This becomes potentially very confusing when one is also considering condensed sets. Namely, CW complexes also embed fully faithfully into condensed sets.

Thus, there are two natural functors from CW complexes to condensed anima: A fully faithful functor, via the full subcategory of condensed sets; and another non-faithful functor, factoring over the full-$\infty$-subcategory of anima. Let us (abusively) denote the first functor by $X\mapsto X$, and the second by $X\mapsto |X|$. Then $S^1$ is a physical circle, while $|S^1|$ is some ghostly appearance of a point with an internal automorphism (the anima $\ast/\mathbb Z$).

\begin{lemma} Let $X$ be a CW complex (more generally, any condensed set that is a filtered colimit of condensed sets that are built out of pushouts of $S^{n-1}=\partial D^n\hookrightarrow D^n$). There is a universal anima $Y$ with a map $X\to Y$ of condensed anima. There is a functorial identification $Y\cong|X|$.
\end{lemma}

In particular, there is a natural map $X\to|X|$ of condensed anima. For $S^1$, we get a cartesian square (of condensed anima)
\[\xymatrix{
\mathbb R\ar[r]\ar[d] & \ast\ar[d]\\
S^1\ar[r] & |S^1|
}\]
involving the universal cover $\mathbb R\to S^1$. If $X$ is a connected CW complex, then $X\times_{|X|} \tau_{\geq 2} |X|\to X$ is the universal cover of $X$. When $X$ has higher homotopy groups, then $X\times_{|X|} \ast$ is an actual condensed anima in that it is nontrivial both in the condensed and in the animated direction. We invite the reader to get a mental image of $S^2\times_{|S^2|} \ast$.

\begin{proof} As $X$ is an iterated pushout of $S^{n-1}=\partial D^n\hookrightarrow D^n$ (and these pushouts are compatible with the fully faithful functors from compactly generated topological spaces to condensed sets to condensed anima), it suffices to prove the existence of $Y$ in those cases; and this reduces, moreover, inductively to the case $X=D^n$. In that case, the natural map $Y\to \Hom(X,Y)$ is an equivalence for any anima $Y$. To check this, it suffices to see that it induces a bijection on $\pi_0$: Indeed, applying the statement on the level of $\pi_0$ to $\Hom(Z,Y)$ for any other anima $Z$ will then give $\pi_0\Hom(Z,Y)\cong \pi_0 \Hom(Z,\Hom(X,Y))$ for all $Z$ (as $\Hom(Z,\Hom(X,Y))\cong \Hom(X,\Hom(Z,Y))$), which means that $Y\to \Hom(X,Y)$ is an equivalence.

To check that $\pi_0 Y\to \pi_0 \Hom(D^n,Y)$ is an isomorphism, we pass up the Postnikov tower. If $Y$ is discrete, this follows from connectedness of $D^n$. If $Y=BG$ for some discrete group $G$, then $\Hom(X,Y)$ classifies $G$-torsors over $X$ (as a condensed set). It is not hard to see that these are locally on $X$ trivial, as any trivialization at a point spreads uniquely to small closed neighborhoods (after a resolution by extremally disconnected sets, and thus by descent on $X$). Thus, they are classified by usual $G$-torsors on the topological space $X=D^n$, and these are all split. If $Y=K(\pi,n)$ for $n>1$, then $\Hom(X,Y)$ is given by $H^n_{\mathrm{cond}}(X,\pi)$. By last semester, this agrees with singular cohomology, and so it vanishes for $X=D^n$.

We see that for $X=D^n$, the map $X\to |X|=\ast$ is the universal map to an anima. As the functor sending $X$ to this initial anima commutes with colimits, as does $X\to |X|$, this gives the identification in general.
\end{proof}

\newpage

\section{Lecture XII: Localizations}

We need to generalize the theory of analytic rings from condensed rings to condensed animated rings. As this works in the context of (still unital) associative rings, let us temporarily work in that setting. It has the advantage that the $\infty$-category of animated condensed associative rings can also be described as the $\infty$-category of $\mathbb E_1$-algebras in the symmetric monoidal $\infty$-category $\mathcal D_{\geq 0}(\mathrm{Cond}(\mathrm{Ab}))=\mathrm{Ani}(\mathrm{Cond}(\mathrm{Ab}))$ of animated condensed abelian groups. Any animated condensed ring $A$ thus has an $\infty$-category $\mathcal D_{\geq 0}(A)$ of modules in animated condensed abelian groups. This is a prestable $\infty$-category in the sense of Lurie (cf.~\cite[Appendix C]{LurieSAG}); in particular, it embeds fully faithfully into a stable $\infty$-category $\mathcal D(A)$, as the positive part of a $t$-structure. The heart is equivalent to the category of condensed modules over the condensed ring $\pi_0 A$. (The preceding discussion applies to sheaves on any site.)

The following is a direct generalization of \cite[Definition 7.1, 7.4]{Condensed}.

\begin{definition}\label{def:analyticanimated} An analytic animated associative ring is a pair $(A,\mathcal M)$ consisting of a condensed animated associative ring $A$ together with a covariant functor
\[
\mathcal M: S\mapsto \mathcal M[S]
\]
from extremally disconnected profinite sets $S$ to condensed animated $A$-modules, taking finite coproducts to finite direct sums, together with a natural transformation $S\to \mathcal M[S]$ of condensed anima, with the following property. For any object $C\in \mathcal D_{\geq 0}(A)$ that is a sifted colimit of objects of the form $\mathcal M[S_i]$ for varying $S_i$, the natural map
\[
\intHom_{\mathcal D_{\geq 0}(A)}(\mathcal M[S],C)\to \intHom_{\mathcal D_{\geq 0}(A)}(A[S],C)
\]
of condensed anima is an equivalence for all extremally disconnected profinite sets $S$.
\end{definition}

\begin{remark} The letter $\mathcal M$ is in reference to spaces of {\it m}easures, but it can also be thought of as denoting extra data on {\it m}odules, singling out the correct category of modules inside all of $\mathcal D_{\geq 0}(A)$.
\end{remark}

We note that in the condition, the map is actually a map of condensed animated abelian groups, and by embedding into $\mathcal D(A)$, we could also use the $R\intHom$ and ask for an isomorphism in $\mathcal D(\mathrm{Cond}(\mathrm{Ab}))$, as in \cite[Definition 7.4]{Condensed}. (A priori, the condition on $R\intHom$ is stronger as it involves negative homotopy groups, but it follows from the given condition applied to all shifts $C[i]$, $i\geq 0$.)

The only difference with the situation of \cite[Definition 7.1, 7.4]{Condensed} is that we allow both $A$ and $\mathcal M[S]$ to be complexes, living in nonnegative homological degrees.

\begin{definition} Let $(A,\mathcal M)$ be an analytic animated associative ring. The $\infty$-category $\mathcal D_{\geq 0}(A,\mathcal M)$ is defined to be the full $\infty$-subcategory
\[
\mathcal D_{\geq 0}(A,\mathcal M)\subset \mathcal D_{\geq 0}(A)
\]
spanned by all $C\in \mathcal D_{\geq 0}(A)$ such that for all extremally disconnected profinite sets $S$, the map
\[
\intHom_{\mathcal D_{\geq 0}(A)}(\mathcal M[S],C)\to \intHom_{\mathcal D_{\geq 0}(A)}(A[S],C)
\]
of condensed anima is an equivalence.
\end{definition}

\begin{proposition} Let $(A,\mathcal M)$ be an analytic animated associative ring. The $\infty$-category $\mathcal D_{\geq 0}(\A,\mathcal M)$ is generated under sifted colimits by the objects $\mathcal M[S]$ for varying extremally disconnected profinite sets $S$, which are compact projective objects of $\mathcal D_{\geq 0}(A,\mathcal M)$. The full $\infty$-subcategory
\[
\mathcal D_{\geq 0}(A,\mathcal M)\subset \mathcal D_{\geq 0}(A)
\]
is stable under all limits and colimits and admits a left adjoint
\[
-\otimes_A (A,\mathcal M): \mathcal D_{\geq 0}(A)\to \mathcal D_{\geq 0}(A,\mathcal M)
\]
sending $A[S]$ to $\mathcal M[S]$.

The $\infty$-category $\mathcal D_{\geq 0}(A,\mathcal M)$ is prestable. Its heart is the abelian category $\mathcal D^\heart(A,\mathcal M)$ that is the full subcategory of condensed $\pi_0 A$-modules generated under colimits by $\pi_0 \mathcal M[S]$ for varying $S$. An object $C\in \mathcal D_{\geq 0}(A)$ lies in $\mathcal D_{\geq 0}(A,\mathcal M)$ if and only if all $H_i(C)$ lie in $\mathcal D^\heart(A,\mathcal M)$.

If $A$ has the structure of a condensed animated commutative ring so that $\mathcal D_{\geq 0}(A)$ is naturally a symmetric monoidal $\infty$-category, there is a unique symmetric monoidal structure on $\mathcal D_{\geq 0}(A,\mathcal M)$ making $-\otimes_A (A,\mathcal M)$ symmetric monoidal.
\end{proposition}

\begin{remark} Passing to ``spectrum objects'' (a formal procedure that in the current situation recovers $\mathcal D(-)$ from $\mathcal D_{\geq 0}(-)$), similar statements hold true on the level of $\mathcal D(A,\mathcal M)\subset \mathcal D(A)$.
\end{remark}

A consequence of the proposition is that $\mathcal M[S]$ is determined by the full $\infty$-subcategory
\[
\mathcal D_{\geq 0}(A,\mathcal M)\subset \mathcal D_{\geq 0}(A)
\]
as the image of $A[S]$ under the left adjoint to the full inclusion. Moreover, this $\infty$-subcategory is determined already by the collection $\{\pi_0 \mathcal M[S]\}_S$ of condensed $\pi_0 A$-modules. In other words, analytic ring structures on $A$ are completely determined by data on the level of usual abelian categories. In the first appendix to this lecture, we will characterize analytic ring structures from this perspective.

\begin{proof} By the definition of analytic animated rings, all sifted colimits of objects $\mathcal M[S]$ are in $\mathcal D_{\geq 0}(A,\mathcal M)$. On the other hand, by the definition of $\mathcal D_{\geq 0}(A,\mathcal M)$, the objects $\mathcal M[S]$ are compact projective generators.

By definition of $\mathcal D_{\geq 0}(A,\mathcal M)$, it is stable under all limits. For the left adjoint, it is clear that it exists on $A[S]$ with value $\mathcal M[S]$ by definition of $\mathcal D_{\geq 0}(A,\mathcal M)$, and then it also exists for all of their sifted colimits by the assumption that $(A,\mathcal M)$ is an analytic ring; but these sifted colimits exhaust $\mathcal D_{\geq 0}(A)$. Now the existence of general colimits follows, by forming them first in $\mathcal D_{\geq 0}(A)$, and then applying the left adjoint.

It is now formal that $\mathcal D_{\geq 0}(A,\mathcal M)$ is prestable, cf.~\cite[Corollary C.1.2.3]{LurieSAG}. If $C\in \mathcal D_{\geq 0}(A,\mathcal M)$, then also $\tau_{\geq 1} C\in \mathcal D_{\geq 0}(A,\mathcal M)$ as this is the suspension of the loops of $C$ and $\mathcal D_{\geq 0}(A,\mathcal M)$ is stable under all limits and colimits. Thus, $H_0(C)[0]\in \mathcal D_{\geq 0}(A,\mathcal M)$. The statement about the heart is then formal.

Regarding the symmetric monoidal structure, one possible reference is \cite[Theorem I.3.6]{NikolausScholze}, noting that the condition of being an analytic ring ensures that the kernel of $-\otimes_A (A,\mathcal M)$ is a $\otimes$-ideal (as the kernel is generated under colimits by the cones of $A[S]\to \mathcal M[S]$, and tensoring them over $\mathbb Z$ with $\mathbb Z[T]$ for varying $T$ still lies in the kernel, by evaluating the internal Hom at $T$).
\end{proof}

The definition of analytic animated associative rings suggests to define a map $(A,\mathcal M)\to (B,\mathcal N)$ of such to be a map of condensed animated associative rings $A\to B$ together with a natural transformation $\mathcal M[S]\to \mathcal N[S]$ linear over this map and commuting with the maps from $S$.

Following the discussion in \cite[Lecture 7]{Condensed}, we instead define a map $(A,\mathcal M)\to (B,\mathcal N)$ to be a map $A\to B$ of condensed animated associative rings such that the forgetful functor $\mathcal D_{\geq 0}(B,\mathcal N)\to \mathcal D_{\geq 0}(A)$ takes image in $\mathcal D_{\geq 0}(A,\mathcal M)$. It is enough to check this condition for $\mathcal N[S]$, and in fact for $\pi_0 \mathcal N[S]$, as these generate the abelian heart, and the categories are determined by the heart.

Given a map $(A,\mathcal M)\to (B,\mathcal N)$ of analytic animated associative rings, the maps $S\to \mathcal N[S]$ extend uniquely to maps $\mathcal M[S]\to \mathcal N[S]$ linear over $A\to B$, so one gets a necessarily unique map in the naive sense. The converse is also close to being true: If one has functorial maps $\mathcal M[S]\to \mathcal N[S]$ linear over $A\to B$ and commuting with the map from $S$, then this defines a map $(A,\mathcal M)\to (B,\mathcal N)$ under a mild condition. In fact, it is enough to give functorial maps $\pi_0 \mathcal M[S]\to \pi_0 \mathcal N[S]$ linear over $\pi_0 A\to \pi_0 B$ such that for all maps $S\to \pi_0 A$ from some extremally disconnected profinite set $S$, an induced diagram
\[\xymatrix{
\pi_0 \mathcal M[S]\ar[r]\ar[d] & \pi_0 \mathcal M[S]\ar[d]\\
\pi_0 \mathcal N[S]\ar[r] & \pi_0 \mathcal N[S]
}\]
commutes, cf.~\cite[Proposition 7.14]{Condensed} (whose proof immediately passes to the present situation).

\begin{proposition} Let $f: (A,\mathcal M)\to (B,\mathcal N)$ be a map of analytic animated associative rings. The forgetful functor $\mathcal D_{\geq 0}(B,\mathcal N)\to \mathcal D_{\geq 0}(A,\mathcal M)$ admits a left adjoint
\[
-\otimes_{(A,\mathcal M)} (B,\mathcal N): \mathcal D_{\geq 0}(A,\mathcal M)\to \mathcal D_{\geq 0}(B,\mathcal N)
\]
sending $\mathcal M[S]$ to $\mathcal N[S]$. This functor is naturally symmetric monoidal when $A$ and $B$ and $f: A\to B$ have the structure of condensed animated commutative rings (resp.~of a map between such).
\end{proposition}

\begin{remark} We should really encode more functoriality, especially regarding composition. This is best done by defining the relevant (co)Cartesian fibrations and straightforward in the present situation.
\end{remark}

\begin{proof} The existence and description of the left adjoint is clear. The symmetric monoidal structure follows from \cite[Theorem I.3.6]{NikolausScholze} again.
\end{proof}

One can base change analytic ring structures along maps of condensed rings, as follows. We emphasize that this proposition works only in the animated setting -- the relevant base changes will in general be derived, and we have to keep track of the derived structure. We will from now on consider only the commutative setting.\footnote{In the original notes, several claims were also made in the associative setting, but the story is more complicated there.}

\begin{proposition}\label{prop:inducedanalytic} Let $(A,\mathcal M)$ be an analytic animated commutative ring and let $g: A\to B$ be a map of condensed animated commutative rings. Then the functor
\[
S\mapsto \mathcal N[S]:=B[S]\otimes_A (A,\mathcal M)
\]
defines an analytic animated commutative ring $(B,\mathcal N)$.
\end{proposition}

\begin{proof} Note that all $\mathcal N[S]$ are, as $A$-modules, by definition in $\mathcal D_{\geq 0}(A,\mathcal M)$; thus, for any sifted colimit $N\in \mathcal D_{\geq 0}(B)$ of these, forgetting to $\mathcal D_{\geq 0}(A)$ defines an object of $\mathcal D_{\geq 0}(A,\mathcal M)$.

We want to prove that for all such $N$, the map
\[
\intHom_{\mathcal D_{\geq 0}(B)}(\mathcal N[S],N)\to \intHom_{\mathcal D_{\geq 0}(B)}(B[S],N)
\]
is an isomorphism. From $N\in \mathcal D_{\geq 0}(A,\mathcal M)$ and the definition of $\mathcal N[S]$, we know that
\[
\intHom_{\mathcal D_{\geq 0}(A)}(\mathcal N[S],N) = \intHom_{\mathcal D_{\geq 0}(A)}(B[S],N).
\]
This implies the desired result formally. Indeed, for any $N_1,N_2\in \mathcal D_{\geq 0}(B)$, one has
\[
\intHom_{\mathcal D_{\geq 0}(B)}(N_1,N_2) = \intHom_{\mathcal D_{\geq 0}(B\otimes_A B)}(B,\intHom_{\mathcal D_{\geq 0}(A)}(N_1,N_2)),
\]
so the agreement of $\intHom_{\mathcal D_{\geq 0}(A)}$'s implies the agreement of $\intHom{\mathcal D_{\geq 0}(B)}$'s.
\end{proof}

We note that in our definition of analytic rings, we never enforced that $\mathcal M[\ast]=A$, and in particular $A$ need not be the endomorphism ring of the unit of $\mathcal D_{\geq 0}(A,\mathcal M)$ (the latter is $\mathcal M[\ast]$).

\begin{definition} An analytic animated associative ring $(A,\mathcal M)$ is complete if the map $A\to \mathcal M[\ast]$ is an isomorphism.
\end{definition}

In the associative case, we do not whether we can enforce this. In the commutative case, this can be enforced, except for a small issue: It is not in general clear that $\mathcal M[\ast]$ acquires the structure of an animated condensed commutative ring. As analyzed in the third appendix, this requires that the $\mathrm{Sym}^n$-operations on $\mathcal D_{\geq 0}(A)$ descend to $\mathcal D_{\geq 0}(A,\mathcal M)$. As proved there, this is implied by insisting on the following property:

\begin{definition} An analytic animated commutative ring is a condensed animated commutative ring $A$ together with a functor $\mathcal M: S\mapsto \mathcal M[S]$ as above so that $(A,\mathcal M)$ is an analytic animated associative ring, and such that for all primes $p$, the Frobenius map $\phi_p: A\to A/^{\mathbb L} p$ induces a map of analytic animated associative rings $\phi_p: (A,\mathcal M)\to (A/^{\mathbb L} p,\mathcal M/^{\mathbb L} p)$. A map of analytic animated commutative rings $(A,\mathcal M)\to (B,\mathcal N)$ is a map of condensed animated commutative rings $A\to B$ that induces a map of analytic animated associative rings.
\end{definition}

This condition on Frobenius is analyzed in the second appendix, and shown there to be automatic in a number of situations; it may well be the case that it is automatic in general.\footnote{In \cite{AnalyticStacks}, we work in the setting of light condensed rings, and the argument is strong enough to show that it is automatic in that setting.} By the results of the third appendix, one can complete analytic animated commutative rings.\footnote{If one works in the context of connective $\mathbb E_\infty$-rings, this subtlety about Frobenius does not come up, as then it is clear that completions exist (as the tensor unit in a symmetric monoidal $\infty$-category is an $\mathbb E_\infty$-ring); also all the rest of the discussion of these lectures immediately adapts to that context.} In fact, $\mathcal M[\ast]$ is initial in the $\infty$-category of condensed animated commutative rings $B$ under $A$ with the property that $B\in \mathcal D_{\geq 0}(A,\mathcal M)\subset \mathcal D_{\geq 0}(A)$. Moreover, all $\mathcal M[S]$ are automatically $\mathcal M[\ast]$-modules, making $(\mathcal M[\ast],\mathcal M)$ into an analytic animated commutative ring, such that the forgetful functor
\[
\mathcal D_{\geq 0}(\mathcal M[\ast],\mathcal M)\to \mathcal D_{\geq 0}(A,\mathcal M)
\]
is an equivalence, cf.~Proposition~\ref{prop:normalizeanalyticring} (and its proof, which also applies in the present setting).

In the following, we will only work with complete analytic rings.

\begin{definition} Let $\AnRing$ be the $\infty$-category of complete analytic animated commutative rings.
\end{definition}

Let us analyze this $\infty$-category. For the following proposition, it is again critical to work in the animated context.

\begin{proposition} The $\infty$-category $\AnRing$ admits all small colimits. The initial object is $\mathbb Z$ with $S\mapsto \mathbb Z[S]$. Sifted colimits commute with the functor $(A,\mathcal M)\mapsto \mathcal M[S]$ to $\mathcal D_{\geq 0}(\Cond(\Ab))$. Pushouts are computed as follows:

Let $(B,\mathcal M_B)\leftarrow (A,\mathcal M_A)\to (C,\mathcal M_C)$ be a diagram in $\AnRing$. The pushout $(B,\mathcal M_B)\otimes_{(A,\mathcal M_A)}(C,\mathcal M_C)=(E,\mathcal M_E)$ of this diagram in $\AnRing$ exists. It can be defined as the completion of an analytic ring structure on 
\[
B\otimes_A C.
\]
The corresponding forgetful functor
\[
\mathcal D_{\geq 0}(E,\mathcal M_E)\to \mathcal D_{\geq 0}(B\otimes_A C)
\]
is fully faithful, with essential image given by all $C\in \mathcal D_{\geq 0}(B\otimes_A C)$ whose images in $\mathcal D_{\geq 0}(B)$ and $\mathcal D_{\geq 0}(C)$ lie in $\mathcal D_{\geq 0}(B,\mathcal M_B)$ and $\mathcal D_{\geq 0}(C,\mathcal M_C)$.

The left adjoint
\[
-\otimes_{B\otimes_A C} (E,\mathcal M_E): \mathcal D_{\geq 0}(B\otimes_A C)\to \mathcal D_{\geq 0}(E,\mathcal M_E)
\]
is given by the sequential colimit
\[
-\to -\otimes_B (B,\mathcal M_B)\to (-\otimes_B (B,\mathcal M_B))\otimes_C (C,\mathcal M_C)\to ((-\otimes_B (B,\mathcal M_B))\otimes_C (C,\mathcal M_C))\otimes_B (B,\mathcal M_B) \to \ldots
\]
(where each functor is considered as an endofunctor of $\mathcal D_{\geq 0}(B\otimes_A C)$).
\end{proposition}

In particular, pushouts in $\AnRing$ are in general subtle, as the operations of ``$\mathcal M_B$-completion'' $-\otimes_B (B,\mathcal M_B)$ and ``$\mathcal M_C$-completion'' $-\otimes_C (C,\mathcal M_C)$ do not in general commute.

\begin{proof} The statement about the initial object is clear. For the statement about sifted colimits, let $(A_\bullet,\mathcal M_\bullet)$ be any sifted diagram in $\AnRing$. Let $(B,\mathcal M_B)$ denote their naively computed sifted colimit; then $B$ is an animated condensed commutative ring and $\mathcal M_B$ is a theory of measures. Moreover, the spaces of measures $\mathcal M_B[S]=\colim \mathcal M_\bullet[S]$ are objects of $\mathcal D_{\geq 0}(B)$ whose restriction to $\mathcal D_{\geq 0}(A_i)$ lies in $\mathcal D_{\geq 0}(A_i,\mathcal M_i)$, for all $i$.\footnote{This is clear for filtered colimits, and thus to see it for general sifted colimits, it suffices to check it for geometric realizations. Then any map $A_i\to \colim A_\bullet$ factors over $A_0$, so it suffices to see it for $i=0$. But degeneracies make the $A_i$ compatibly into $A_0$-algebras and hence all $\mathcal M_i[S]$ are in $\mathcal D_{\geq 0}(A_0,\mathcal M_0)$, and the same is true for their geometric realization.} Thus for any sifted colimit $C$ of geometric realizations of $\mathcal M_B[S]$'s,
\[
\intHom_{\mathcal D_{\geq 0}(A_i)}(\mathcal M_i[S],C) = \intHom_{\mathcal D_{\geq 0}(A_i)}(A_i[S],C)
\]
or equivalently
\[
\intHom_{\mathcal D_{\geq 0}(B)}(\mathcal M_i[S]\otimes_{A_i} B,C) = \intHom_{\mathcal D_{\geq 0}(B)}(B[S],C).
\]
Taking the limit over $i$ of this statement and using that the colimit of $\mathcal M_i[S]\otimes_{A_i} B$ is just $\mathcal M_B[S]$ as the colimit is sifted, we get the desired result.

For pushouts, we note that any object in $(E',\mathcal M_{E'})\in \AnRing$ with compatible maps from $(B,\mathcal M_B)\leftarrow (A,\mathcal M_A)\to (C,\mathcal M_C)$ will admit a natural map $\mathcal D_{\geq 0}(E',\mathcal M_{E'})\to \mathcal D_{\geq 0}(B\otimes_A C)$ whose essential image lies in the full $\infty$-subcategory of all objects whose restrictions to the two factors lie in $\mathcal D_{\geq 0}(B,\mathcal M_B)$ and $\mathcal D_{\geq 0}(C,\mathcal M_C)$. Thus, we see that it is enough to prove that the recipee in the statement defines an analytic ring structure on $B\otimes_A C$; but this is easy to see (using, for example, Proposition~\ref{prop:characterizeanalyticstructure}). Passing to the completion gives the desired pushout.
\end{proof}

One instance where the base change in $\AnRing$ is simple is for induced analytic ring structures as in Proposition~\ref{prop:inducedanalytic}.

Another such instance is the following class of maps.

\begin{definition} A map $f: (A,\mathcal M_A)\to (B,\mathcal M_B)$ in $\AnRing$ is steady if for all maps $g: (A,\mathcal M_A)\to (C,\mathcal M_C)$ in $\AnRing$, the functor $M\mapsto M\otimes_B (B,\mathcal M_B)$ preserves the full $\infty$-subcategory of $\mathcal D_{\geq 0}(C\otimes_A B)$ of all objects whose restriction to $C$ lies in $\mathcal D_{\geq 0}(C,\mathcal M_C)$.
\end{definition}

Equivalently, for all $M\in \mathcal D_{\geq 0}(C,\mathcal M_C)$, the object
\[
M\otimes_A (B,\mathcal M_B) = M\otimes_{(A,\mathcal M_A)} (B,\mathcal M_B),
\]
which is a priori an object of $\mathcal D_{\geq 0}(C\otimes_A B)$, lies in $\mathcal D_{\geq 0}(C,\mathcal M_C)$ when restricted to $C$. As its restriction to $B$ lies in $\mathcal D_{\geq 0}(B,\mathcal M_B)$ by construction, this will then actually define an object of
\[
\mathcal D_{\geq 0}((B,\mathcal M_B)\otimes_{(A,\mathcal M_A)}(C,\mathcal M_C)).
\]
In fact, the condition that $f$ is steady means precisely that for all $g$ the colimit
\[
-\to -\otimes_C (C,\mathcal M_C)\to (-\otimes_C (C,\mathcal M_C))\otimes_B (B,\mathcal M_B)\to ((-\otimes_C (C,\mathcal M_C))\otimes_B (B,\mathcal M_B))\otimes_C (C,\mathcal M_C) \to \ldots
\]
stabilizes at $(-\otimes_C (C,\mathcal M_C))\otimes_B (B,\mathcal M_B)$. The pushout $(E,\mathcal M_E)$ of $(B,\mathcal M_B)\leftarrow (A,\mathcal M_A)\to (C,\mathcal M_C)$ is then given by the functor
\[
\mathcal M_E: S\mapsto \mathcal M_C[S]\otimes_{(A,\mathcal M_A)}(B,\mathcal M_B).
\]

Yet another equivalent characterization is the following; it is really just a tautological reformulation of the definition.

\begin{proposition} A map $f: (A,\mathcal M_A)\to (B,\mathcal M_B)$ in $\AnRing$ is steady if and only if for all pushout diagrams
\[\xymatrix{
(E,\mathcal M_E) & (C,\mathcal M_C)\ar[l]_{\tilde{f}}\\
(B,\mathcal M_B)\ar[u]_{\tilde{g}} & (A,\mathcal M_A)\ar[l]_f\ar[u]_g
}\]
and all $M\in \mathcal D(C,\mathcal M_C)$, the base change map
\[
M|_A\otimes_{(A,\mathcal M_A)} (B,\mathcal M_B)\to (M\otimes_{(C,\mathcal M_C)} (E,\mathcal M_E))|_B
\]
is an isomorphism.
\end{proposition}

In the geometric language of the next lecture, this means that in the cartesian diagram
\[\xymatrix{
\AnSpec(E,\mathcal M_E)\ar[r]^{\tilde{f}}\ar[d]^{\tilde{g}} & \AnSpec(C,\mathcal M_C)\ar[d]^g\\
\AnSpec(B,\mathcal M_B)\ar[r]^f & \AnSpec(A,\mathcal M_A),
}\]
the map
\[
f^\ast g_\ast C\to \tilde{g}_\ast \tilde{f}^\ast C
\]
is an isomorphism (using the usual notation for pullbacks and pushforwards of quasicoherent sheaves). Thus, pushforward commutes with base change along steady maps, and this property characterizes steady maps.

The class of steady maps has good properties.

\begin{proposition} The class of steady maps is stable under base change and composition. Moreover, the class of steady maps is closed under all colimits (in $\mathrm{Fun}(\Delta^1,\AnRing)$).
\end{proposition}

\begin{proof} Exercise.
\end{proof}

\begin{definition} A map $f: (A,\mathcal M)\to (B,\mathcal N)$ in $\AnRing$ is a localization if the forgetful functor $\mathcal D_{\geq 0}(B,\mathcal N)\to \mathcal D_{\geq 0}(A,\mathcal M)$ is fully faithful. The map $f$ is a steady localization if it is a localization and steady.
\end{definition}

\begin{exercise} Show that if a map $f: (A,\mathcal M)\to (B,\mathcal N)$ in $\AnRing$ is a localization, then the induced map $(B,\mathcal N)\to (B,\mathcal N)\otimes_{(A,\mathcal M)}(B,\mathcal N)$ is an isomorphism. (Hint: First establish that any base change of a localization is again a localization, by noting that $(B,\mathcal N)$ can be regarded as the completion of a non-complete analytic ring structure on $A$.) If $f$ is steady, prove the converse. (Hint: The condition of fully faithfulness is equivalent to the statement that for all $M\in \mathcal D_{\geq 0}(B,\mathcal N)$, the map $M\to M\otimes_{(A,\mathcal M)}(B,\mathcal N)$ is an equivalence. Rewrite the latter as $M\otimes_{(B,\mathcal N)} ((B,\mathcal N)\otimes_{(A,\mathcal M)}(B,\mathcal N))$.)
\end{exercise}

With this definition, we have the following descent of modules.

\begin{proposition}\label{prop:descent} Let $f_i: (A,\mathcal M)\to (A_i,\mathcal M_i)$, $i\in I$, be a finite family of steady localizations in $\AnRing$ such that the functor $\mathcal D(A,\mathcal M)\to \prod_i \mathcal D(A_i,\mathcal M_i)$ is conservative. Let $\mathcal C_I$ be the category of nonempty subsets of $I$; we get a functor $\mathcal C_I\to \AnRing$ sending any $J\subset I$ to $(A_J,\mathcal M_J) := \bigotimes_{i\in J,/(A,\mathcal M)} (A_i,\mathcal M_i)$.

Then for any $M\in \mathcal D(A,\mathcal M)$ the natural map
\[
M\to \lim_{J\in C_I} (M\otimes_{(A,\mathcal M)} (A_J,\mathcal M_J))
\]
is an isomorphism, and
\[
\mathcal D(A,\mathcal M)\to \lim_{J\in C_I} \mathcal D(A_J,\mathcal M_J)
\]
is an equivalence.
\end{proposition}

\begin{proof} This is formal, cf.~(proof of) \cite[Proposition 10.5]{Condensed}. The key point is that steady localizations commute with any base change. We also use that localizations commute with finite limits, for which we have to pass from the prestable $\infty$-category $\mathcal D_{\geq 0}$ to the stable $\infty$-category $\mathcal D$ (where they can be reinterpreted as finite colimits).
\end{proof}

In the next lecture, we discuss some examples. In particular, we will see that for the analytic rings corresponding to Huber pairs $(A,A^+)$, the condition of $(A,A^+)_\solid\to (B,B^+)_\solid$ being steady is closely related to the condition that $A\to B$ is adic in Huber's sense, i.e.~given compatible rings of definition $A_0\subset A$, $B_0\subset B$, $A_0\to B_0$, if $I\subset A_0$ is an ideal of definition, then $IB_0\subset B_0$ is an ideal of definition. Thus the present discussion mirrors a standard discussion on Huber pairs: In that situation, general pushouts do not exist, but they do exist when one of the maps is adic.

\newpage

\section*{Appendix to Lecture XII: Topological invariance of analytic ring structures}

In this appendix, we characterize analytic ring structures in terms of the full sub-$\infty$-category $\mathcal D_{\geq 0}(A,\mathcal M)\subset \mathcal D_{\geq 0}(A)$. As an application, we prove that analytic ring structures are invariant under nilpotent thickenings. First, we have the following result comparing subcategories of connective derived categories, and of abelian categories.

\begin{proposition}\label{prop:derivedvsabelian} Let $A$ be a condensed animated associative ring. The collection of full sub-$\infty$-categories $\mathcal D\subset \mathcal D_{\geq 0}(A)$ stable under all limits and colimits is in natural bijection with the collection of all full subcategories $\mathcal C\subset \pi_0 A\Mod$ stable under all limits, colimits, and extensions, via sending $\mathcal D$ to the intersection with $\pi_0 A\Mod$, and $\mathcal C$ to the full sub-$\infty$-category $\mathcal D$ of all $C\in \mathcal D_{\geq 0}(A)$ such that all $H_i(C)\in\mathcal C$ for $i\geq 0$.
\end{proposition}

\begin{proof} If $\mathcal D\subset \mathcal D_{\geq 0}(A)$ is stable under all limits and colimits, then in particular for any $C\in \mathcal D$, also $\tau_{\geq 1} C\in \mathcal D$ as the suspension of the loops of $C$, and thus $H_0(C)[0]\in \mathcal D$. It follows that defining $\mathcal C$ to be the intersection of $\mathcal D$ with $\pi_0 A\Mod$, one has $C\in\mathcal D$ if and only if all $H_i(C)\in \mathcal C$. Indeed, the forward direction follows by inductively applying this argument, and the converse follows as $\mathcal D$ is stable under extensions (as these can be written as cofibers) and (Postnikov) limits. Moreover, stability of $\mathcal D$ under limits and colimits implies stability of $\mathcal C$ under limits, colimits, and extensions.

Conversely, if $\mathcal C\subset \pi_0 A\Mod$ is stable under limits, colimits, and extensions, then defining $\mathcal D$ as in the statement of the proposition, one sees that $\mathcal D$ is stable under fibres and cofibres, as well as infinite direct sums or infinite direct products, and thus under all limits and colimits. The intersection of $\mathcal D$ with $\pi_0 A\Mod$ is by definition $\mathcal C$, giving the result.
\end{proof}

Now we can characterize analytic ring structures.

\begin{proposition}\label{prop:characterizeanalyticstructure} Let $A$ be a condensed animated associative ring. A full sub-$\infty$-category $\mathcal D\subset \mathcal D_{\geq 0}(A)$ is of the form $\mathcal D_{\geq 0}(A,\mathcal M)$ for a necessarily unique analytic ring structure $(A,\mathcal M)$ on $A$ if and only if satisfies the following conditions:
\begin{enumerate}
\item The sub-$\infty$-category $\mathcal D\subset \mathcal D_{\geq 0}(A)$ is stable under all limits and colimits.
\item The sub-$\infty$-category $\mathcal D\subset \mathcal D_{\geq 0}(A)$ is stable under $\intHom_{\mathcal D_{\geq 0}(\Cond(\Ab))}(\mathbb Z[S],-)$ for any extremally disconnected profinite set $S$.
\item The inclusion $\mathcal D\subset \mathcal D_{\geq 0}(A)$ admits a left adjoint.
\end{enumerate}
\end{proposition}

We remark that by the adjoint functor theorem, the existence of the left adjoint is automatic modulo set-theoretic issues.

\begin{proof} This is clear in the forward direction. Conversely, we can define $\mathcal M[S]$ as the image of $A[S]$ under the left adjoint (guaranteed by (3)). Then we formally know that for all $C\in \mathcal D$,
\[
\Hom_{\mathcal D_{\geq 0}(A)}(\mathcal M[S],C)\to \Hom_{\mathcal D_{\geq 0}(A)}(A[S],C)
\]
is an isomorphism of anima. By condition (2), we actually see that it is an isomorphism of condensed anima. As by (1), $\mathcal D$ contains all sifted colimits of $\mathcal M[S]$'s, we see that this defines an analytic ring structure.
\end{proof}

As one application, one can generalize the concept of induced analytic ring structures to the associative case.\footnote{We thank Longke Tang for very helpful discussions on this point.}

\begin{proposition} Let $(A,\mathcal M)$ be an analytic animated associative ring, and let $f: A\to B$ be a map of condensed animated associative rings. Let $\mathcal D\subset \mathcal D_{\geq 0}(B)$ be the full subcategory of all objects that lie in $\mathcal D_{\geq 0}(A,\mathcal M)\subset \mathcal D_{\geq 0}(A)$ when regarded as $A$-module via restriction of scalars along $f$. Then $\mathcal D$ defines an analytic ring structure $(B,\mathcal N)$ on $B$.
\end{proposition}

Unfortunately, the free complete modules $\mathcal N[S]$ are in general subtle to describe. We also warn the reader that if one tries to apply this proposition to $B=\mathcal M[\ast]$ to form the completion of $(A,\mathcal M)$, it is not clear that the forgetful functor $\mathcal D_{\geq 0}(B,\mathcal N)\to \mathcal D_{\geq 0}(A,\mathcal M)$ is an equivalence; we believe this is false in general. Thus, completion is not wellbehaved in the associative case.

\begin{proof} We have to check the conditions of Proposition~\ref{prop:characterizeanalyticstructure}. Parts (1) and (2) are clear, and for part (3) it remains to describe the completion functor. Let $N\in \mathcal D_{\geq 0}(B)$. One can forget to $\mathcal D_{\geq 0}(A)$ and form the completion there; base extending to $B$, we get a pushout diagram
\[
N\leftarrow B\otimes_A N\to B\otimes_A (N\otimes_A (A,\mathcal M)).
\]
Let $F(N)\in \mathcal D_{\geq 0}(B)$ denote the pushout, so $F$ is an endofunctor of $\mathcal D_{\geq 0}(B)$ with a map $\mathrm{id}\to F$. Moreover, for any $N'\in \mathcal D$, the induced map
\[
\mathrm{Hom}_{\mathcal D_{\geq 0}(B)}(F(N),N')\to \mathrm{Hom}_{\mathcal D_{\geq 0}(B)}(N,N')
\]
is an isomorphism. Also, in $A$-modules, the map $N\to F(N)$ factors over a complete module, namely $N\otimes_A (A,\mathcal M)$. It follows that $L(N)=\mathrm{colim}_n F^n(N)$ defines the desired completion functor.
\end{proof}

As another application, one can show that analytic ring structures are independent of the animated structure.

\begin{proposition}\label{prop:pi0invariance} Let $f: A\to B$ be map of condensed animated associative rings such that $\pi_0 A\to \pi_0 B$ is an isomorphism.

Then there is a bijective correspondence between analytic ring structures $(A,\mathcal M)$ on $A$ and analytic ring structures $(B,\mathcal N)$ on $B$, subject to the condition that they yield the same complete objects in the abelian heart $\pi_0 A\Mod\cong \pi_0 B\Mod$. In the forward direction, this is given by taking the induced analytic ring structure.
\end{proposition}

\begin{remark} Everything applies, with identical proofs, to general $\mathbb E_1$-algebras in condensed connective spectra. For example one can define a solid analytic ring structure on $\mathbb S$, and for $0<r<1$ an $r$-liquid analytic ring structure on $\mathbb S[T^{-1}]$ (whose completion is some condensed $\mathbb E_\infty$-ring $\mathbb S((T))_{>r}$ that seems subtle to define directly).
\end{remark}

\begin{proof} By 2-out-of-3, we can reduce to the case where $f$ is the projection $A\to \pi_0 A$.

The hearts of $\mathcal D_{\geq 0}(A)$ and $\mathcal D_{\geq 0}(\pi_0 A)$ agree. The corresponding abelian categories of complete modules are identified under taking induced analytic ring structures. To see that this defines the desired bijection, we need to see that for any analytic ring structure $(\pi_0 A,\mathcal M')$ on $\pi_0 A$ with corresponding heart $\mathcal C\subset \pi_0 A\Mod$, one can define an analytic ring structure $(A,\mathcal M)$ such that $C\in \mathcal D_{\geq 0}(A,\mathcal M)$ if and only if all $H_i(C)\in \mathcal C$. By Proposition~\ref{prop:derivedvsabelian}, this defines some full sub-$\infty$-category $\mathcal D\subset \mathcal D_{\geq 0}(A)$ stable under all limits and colimits. We need to check conditions (2) and (3) of Proposition~\ref{prop:characterizeanalyticstructure}. For condition (2), we can assume (by Postnikov limits and filtrations) that $C=X[j]$ for some $X\in \mathcal C$ and $j\geq 0$, in which case $C$ comes from a $\pi_0 A$-module, and the result follows from that case.

For condition (3), we need to see that the functor $M\mapsto M(S)$ on $\mathcal D$ is representable for all extremally disconnected profinite sets $S$. We prove by induction on $i$ that there is some $i$-truncated $\mathcal M_i[S]\in \mathcal D$ with a map $A[S]\to \mathcal M_i[S]$ such that
\[
\Hom_{\mathcal D}(\mathcal M_i[S],M)\to M(S)
\]
is an isomorphism of anima for all $i$-truncated $M\in \mathcal D$. For $i=0$, we can take $\mathcal M_0[S]=\pi_0 \mathcal M'[S]$. Given $\mathcal M_i[S]$, let $N_i$ be the cofiber of $A[S]\to \mathcal M_i[S]$ in $\mathcal D_{\geq 0}(A)$ and let
\[
N_i'=N_i\otimes_A (\pi_0 A,\mathcal M').
\]
As the forgetful functor $\mathcal D_{\geq 0}(\pi_0 A,\mathcal M')\to \mathcal D_{\geq 0}(A)$ lands in $\mathcal D$, we can regard $N_i'$ as an object of $\mathcal D$. Then we let $\mathcal M_{i+1}[S]$ be the fibre of $\mathcal M_i[S]\to N_i\to N_i'$, and we note that $A[S]\to \mathcal M_i[S]$ factors naturally over $\mathcal M_{i+1}[S]$.

To see that
\[
\Hom_{\mathcal D}(\mathcal M_{i+1}[S],M)\to M(S)
\]
is an isomorphism of anima for all $i+1$-truncated $M\in \mathcal D$, it suffices to check for $M=X[i+1]$ with $X\in \mathcal C$. Note that for such $M$, we have the fibre sequence
\[
\Hom_{\mathcal D_{\geq 0}(A)}(N_i,M)\to \Hom_{\mathcal D}(\mathcal M_i[S],M)\to M(S)
\]
where the right-most term is concentrated in degree $i+1$, while the middle term is isomorphic to $M(S)$ after applying $\tau_{\geq 1}$, by induction. It follows that $\Hom_{\mathcal D_{\geq 0}(A)}(N_i,M)$ sits in degree $0$. This can also be written as
\[
\Hom_{\mathcal D_{\geq 0}(\pi_0 A)}(N_i\otimes_A \pi_0 A,M)
\]
which again for all $M=X[i+1]$ with $X\in \mathcal C$ sits in degree $0$. It follows that $N_i'=N_i\otimes_A (\pi_0 A,\mathcal M')$ lies in $\mathcal D_{\geq i+1}(\pi_0 A,\mathcal M')$, and we still have
\[
\Hom_{\mathcal D_{\geq 0}(\pi_0 A)}(N_i',M) = \Hom_{\mathcal D_{\geq 0}(\pi_0 A)}(N_i\otimes_A \pi_0 A,M).
\]
On the other hand, the map
\[
\Hom_{\mathcal D_{\geq 0}(\pi_0 A)}(N_i',M)\to \Hom_{\mathcal D_{\geq 0}(A)}(N_i',M)
\]
is also an isomorphism, as $N_i'\in \mathcal D_{\geq i+1}$ and $M$ is $i+1$-truncated.

In summary, we see that the natural map
\[
\Hom_{\mathcal D_{\geq 0}(A)}(N_i,M)\leftarrow \Hom_{\mathcal D}(N_i',M)
\]
is an isomorphism, where both sides are concentrated in degree $0$. We defined $\mathcal M_{i+1}[S]$ as the fiber of $\mathcal M_i[S]\to N_i'$. Recall that $N_i'$ was in $\mathcal D_{\geq i+1}$, in particular its $\pi_0$ vanishes. This implies that we get a cofiber sequence
\[
\Hom_{\mathcal D}(N_i',M)\to \Hom_{\mathcal D}(\mathcal M_i[S],M)\to \Hom_{\mathcal D}(\mathcal M_{i+1}[S],M)
\]
and comparing with
\[
\Hom_{\mathcal D_{\geq 0}(A)}(N_i,M)\to \Hom_{\mathcal D}(\mathcal M_i[S],M)\to M(S)
\]
gives the result.
\end{proof}

One can moreover prove invariance under nilpotent thickenings.

\begin{proposition}\label{prop:nilinvariance} Let $f: A\to B$ be a map of condensed animated associative rings such that $\pi_0 A\to \pi_0 B$ is surjective with nilpotent kernel $I\subset A$. Then analytic ring structures $(A,\mathcal M)$ on $A$ are in bijection with analytic ring structures $(B,\mathcal N)$ on $B$, via taking induced analytic ring structures.
\end{proposition}

\begin{proof} By Proposition~\ref{prop:pi0invariance}, we can assume that $A$ and $B$ are $0$-truncated. We may also assume that $I^2=0$ by induction. In the forward direction, the functor is given by taking an induced analytic ring structure. Let us define the inverse functor, so take an analytic ring structure $(B,\mathcal N)$ on $B$, corresponding to some abelian category $\mathcal C_B\subset B\Mod$. Let $\mathcal C\subset A\Mod$ be the subcategory of all $M\in A\Mod$ such that $M/IM$ and $IM$, which lie naturally in $B\Mod$, are in $\mathcal C_B$. This contains $\mathcal C_B$ and is stable under all limits, colimits, and extensions; it can also be characterized as the category generated by $\mathcal C_B$ under extensions. By Proposition~\ref{prop:derivedvsabelian}, it corresponds to some full sub-$\infty$-category $\mathcal D\subset \mathcal D_{\geq 0}(A)$. We need to check the criteria of Proposition~\ref{prop:characterizeanalyticstructure}. Condition (2) can again be checked for $C=X[j]$ for some $X\in \mathcal C$ and $j\geq 0$, where it follows by filtering in terms of two objects from $\mathcal C_B$.

It remains to check condition (3). This follows from the argument in the proof of Proposition~\ref{prop:pi0invariance} with a slightly refined induction. Namely, besides the class of $i$-truncated objects of $\mathcal D$, we also consider the class of $i$-truncated objects $C\in \mathcal D$ such that $H_i(C)$ is killed by $I$. Then we start the induction with this restricted class of $0$-truncated objects, then go to all $0$-truncated objects, then to the restricted class of $1$-truncated objects, to all $1$-truncated objects, etc.~.
\end{proof}

\newpage

\section*{Appendix to Lecture XII: Frobenius}

Fix a prime $p$. Consider a condensed animated commutative $\mathbb F_p$-algebra $A$, equipped as an associative ring with an analytic ring structure $(A,\mathcal M)$. One may wonder whether the Frobenius $\phi: A\to A$ induces a map of analytic rings $(A,\mathcal M)\to (A,\mathcal M)$. This condition is stable under all colimits. The goal of this appendix is to show that it is true in some important cases. We denote by $C_p$ the cyclic group of order $p$.

\begin{proposition}\label{prop:frobexists} Extend the functor $S\mapsto \mathcal M[S]$ to all profinite sets by using simplicial resolutions by extremally disconnected sets. Assume the following condition:
\begin{assumption} For all profinite sets $S$ with $C_p$-action and fixed points $S_0 = S^{C_p}$, the natural map
\[
\mathcal M[S_0]^{tC_p}\to \mathcal M[S]^{tC_p}
\]
is an isomorphism.
\end{assumption}

Then $\phi$ induces a map of analytic rings $(A,\mathcal M)\to (A,\mathcal M)$.

Moreover, the assumption is satisfied in the following two situations:
\begin{enumerate}
\item if $(A,\mathcal M)$ is an algebra over $\mathbb F_{p,\solid}$;
\item if for all profinite sets $S$, $\mathcal M[S]$ is $m$-truncated for some $m$.
\end{enumerate}
\end{proposition}

By (the version for animated rings of) \cite[Proposition 7.14]{Condensed}, it is enough to construct functorial $\phi$-linear maps
\[
\phi_S: \pi_0 \mathcal M[S]\to \pi_0 \mathcal M[S]
\]
for all extremally disconnected sets $S$, subject to the condition that for any map $f: T\to A$ from an extremally disconnected profinite set, letting $\tilde{f}: \mathcal M[T]\to \mathcal M[T]$ be the induced extension of $t\mapsto f(t)[t]$, the diagram
\[\xymatrix{
\pi_0 \mathcal M[T]\ar[r]^{\phi_T}\ar[d]^{\pi_0 \tilde{f}} & \pi_0 \mathcal M[T]\ar[d]^{\pi_0 \tilde{f}}\\
\pi_0 \mathcal M[T]\ar[r]^{\phi_T} & \pi_0 \mathcal M[T]
}\]
commutes.

To do this, recall that for any abelian group $M$, there is a natural linear map
\[
M\to \check{H}^0(C_p,M\otimes\ldots\otimes M) = \mathrm{cofib}(\mathrm{Nm}: (M\otimes\ldots\otimes M)_{C_p}\to (M\otimes\ldots\otimes)^{C_p}),
\]
where $M\otimes\ldots\otimes M$ (with $p$ tensor factors) is equipped with the $C_p$-action permuting the tensor factors cyclically, and we take the $0$-th Tate cohomology of the $C_p$-action, i.e.~the cofiber of the norm map from the coinvariants to the invariants. This map is induced by the (non-additive) map
\[
M\to (M\otimes\ldots\otimes M)^{C_p}: m\mapsto m\otimes\ldots\otimes m,
\]
noting that it becomes additive after passing to the cofiber of the norm map.

Applying this to the condensed abelian group $\pi_0 \mathcal M[S]$, we get a natural map
\[
\pi_0 \mathcal M[S]\to \check{H}^0(C_p,\pi_0 \mathcal M[S]\otimes\ldots\otimes \pi_0 \mathcal M[S])
\]
that we can compose with the $C_p$-equivariant multiplication map
\[
\pi_0 \mathcal M[S]\otimes\ldots\otimes \pi_0 \mathcal M[S]\to \pi_0 \mathcal M[S^p]
\]
to get a functorial map
\[
\pi_0 \mathcal M[S]\to \check{H}^0(C_p,\pi_0 \mathcal M[S^p]).
\]

Now note that there is a map $\pi_0 \mathcal M[S]\to \pi_0 \mathcal M[S^p]$ induced by the diagonal inclusion $S\hookrightarrow S^p$. We would like to use that the map $\check{H}^0(C_p,\pi_0 \mathcal M[S])\to \check{H}^0(C_p,\pi_0 \mathcal M[S^p])$ is an isomorphism. As $C_p$ acts trivially on the $p$-torsion module $\pi_0\mathcal M[S]$, this would give the desired map
\[
\pi_0 \mathcal M[S]\to \check{H}^0(C_p,\pi_0 \mathcal M[S]) = \pi_0 \mathcal M[S].
\]

But we have passed to $\pi_0$ too early; we need to stay derived for a little longer to get the passage from $\mathcal M[S^p]$ to $\mathcal M[S]$. For this, recall that one can define Tate cohomology in general for a spectrum $X$ with $C_p$-action, as
\[
X^{tC_p} =\mathrm{cofib}(\mathrm{Nm}: X_{hC_p}\to X^{hC_p})
\]
the cofibre of the norm map from homotopy orbits to homotopy fixed points, cf.~e.g.~\cite[Chapter I]{NikolausScholze}. The map $M\to \check{H}^0(C_p,M\otimes\ldots\otimes M)$ generalizes to the following. Let $X$ be any spectrum in the sense of stable homotopy theory, and consider $X\otimes\ldots\otimes X$, the $p$-fold tensor product, with the cyclic $C_p$-action. Then there is a unique functorial lax symmetric monoidal map
\[
X\to (X\otimes\ldots\otimes X)^{tC_p},
\]
the so-called Tate diagonal, cf.~again \cite{NikolausScholze}. (This map exists only for spectra, not in $\mathcal D(\Ab)$.)

Now following the above, we get a functorial map
\[
\mathcal M[S]\to (\mathcal M[S]\otimes\ldots\otimes \mathcal M[S])^{tC_p}\to \mathcal M[S^p]^{tC_p}.
\]
Again, we have the map $\mathcal M[S]\to \mathcal M[S^p]$ induced by the diagonal $S\to S^p$, and now the assumption guarantees that this induces an isomorphism after applying $-^{tC_p}$.

In total, we have produced a functorial map $\mathcal M[S]\to \mathcal M[S]^{tC_p}$. Now is the time to pass to $\pi_0$: First, on the target, where we get a natural map to $(\pi_0 \mathcal M[S])^{tC_p}$. In the resulting map
\[
\mathcal M[S]\to (\pi_0 \mathcal M[S])^{tC_p}
\]
we take $\pi_0$, which gives the desired functorial map
\[
\pi_0 \mathcal M[S]\to \pi_0 \mathcal M[S].
\]
We need to see that this is linear over $\phi: \pi_0 A\to \pi_0 A$. But $\mathcal M[S]\to \mathcal M[S^p]^{tC_p}\cong \mathcal M[S]^{tC_p}$ is linear over the Tate-valued Frobenius $A\to A^{tC_p}$ by construction, and the Tate-valued Frobenius induces the usual Frobenius on $\pi_0$.

It remains to see that for any map $f: T\to A$ from some extremally disconnected $T$, inducing a map $\tilde{f}: \mathcal M[T]\to \mathcal M[T]$ via $t\mapsto f(t)[t]$, the diagram
\[\xymatrix{
\pi_0 \mathcal M[T]\ar[r]^{\phi_T}\ar[d]^{\pi_0 \tilde{f}} & \pi_0 \mathcal M[T]\ar[d]^{\pi_0 \tilde{f}}\\
\pi_0 \mathcal M[T]\ar[r]^{\phi_T} & \pi_0 \mathcal M[T]
}\]
commutes. But the whole construction is equivariant for such diagonal multiplications. More precisely, it suffices to see that before taking $\pi_0$, the diagram
\[\xymatrix{
\mathcal M[T]\ar[r]\ar[d]^{\tilde{f}} & \mathcal M[T]^{tC_p}\ar[d]^{\tilde{f}^{tC_p}}\\
\mathcal M[T]\ar[r] & \mathcal M[T]^{tC_p}
}\]
commutes. This in turn reduces to showing that the two diagrams
\[\xymatrix{
\mathcal M[T]\ar[r]\ar[d]^{\tilde{f}} & \mathcal M[T^p]^{tC_p}\ar[d]^{\tilde{f}_p^{tC_p}}\\
\mathcal M[T]\ar[r] & \mathcal M[T^p]^{tC_p}
}\]
and
\[\xymatrix{
\mathcal M[T]^{tC_p}\ar[r]\ar[d]^{\tilde{f}^{tC_p}} & \mathcal M[T^p]^{tC_p}\ar[d]^{\tilde{f}_p^{tC_p}}\\
\mathcal M[T]^{tC_p}\ar[r] & \mathcal M[T^p]^{tC_p}
}\]
commute, where $\tilde{f}_p: \mathcal M[T^p]\to \mathcal M[T^p]$ is induced by the map $T^p\to A$ that sends $(t_1,\ldots,t_p)$ to $\prod_{i=1}^p f(t_i)$. But the commutativity of the first diagram follows from functoriality of the Tate construction, while the commutativity of the second diagram is clear.

It remains to show that the assumption is satisfied if $A$ is an algebra over $\mathbb F_{p,\solid}$, or if all $\mathcal M[S]$ are truncated. For this, we note that we can write the cofiber of
\[
\mathcal M[S_0]^{tC_p}\to \mathcal M[S]^{tC_p}
\]
as the limit of
\[
\ldots\to \overline{\mathcal M}[\overline{S}][-2n]\xrightarrow{\beta}\ldots\xrightarrow{\beta} \overline{\mathcal M}[\overline{S}][-2]\xrightarrow{\beta}\overline{\mathcal M}[\overline{S}]
\]
along a certain map $\beta: \overline{\mathcal M}[\overline{S}][-2]\to \overline{\mathcal M}[\overline{S}]$, where $\overline{S}$ is the quotient of the pointed profinite set $S/S_0$ by the free (in pointed profinite sets) $C_p$-action, and $\overline{\mathcal M}$ is the extension of $\mathcal M$ to pointed profinite sets (taking $(T,\ast)$ to $\mathcal M[T]/\mathcal M[\ast]$). This comes from a general formula for $X^{tC_p}$ as a limit of $X_{hC_p}[-2n]$ when $X\in\mathcal D(\Ab)$.

Moreover, by functoriality of the whole situation, the map $\beta$ base changes. In the case of $\mathbb F_{p,\solid}$, it is equal to $0$ as then $\mathcal M[S]$ is projective for all profinite sets $S$; thus, it is also $0$ for any $(A,\mathcal M)$ over $\mathbb F_{p,\solid}$, and hence the limit above vanishes.

On the other hand, if $\overline{\mathcal M}[\overline{S}]$ is truncated, then the limit vanishes as the terms become more and more coconnective.
\newpage

\section*{Appendix to Lecture XII: Normalizations of analytic animated rings}

The goal of this appendix is to prove of the following result.

\begin{proposition} Let $(A,\mathcal M)$ be an analytic animated commutative ring. Consider the $\infty$-category of condensed animated commutative $A$-algebras $B$ such that $B\in\mathcal D_{\geq 0}(A,\mathcal M)\subset \mathcal D_{\geq 0}(A)$. This has an initial object, whose underlying object in $\mathcal D_{\geq 0}(A)$ is $\mathcal M[\ast]$.
\end{proposition}

We use that the $\infty$-category of condensed animated commutative $A$-algebras $B$ is monadic over the $\mathcal D_{\geq 0}(A)$, as follows from the $\infty$-categorical version of the Barr--Beck theorem \cite[Theorem 4.7.0.3]{LurieHA}. Here, the relevant left adjoint functor can be defined by animation. More precisely, consider the category of pairs of maps of commutative rings $A\to B$, with its forgetful functor to the category of pairs $(A,M)$ of a commutative ring $A$ and an $A$-module $M$. This admits a left adjoint, sending $M$ to the free commutative $A$-algebra
\[
M\mapsto A\oplus M\oplus \mathrm{Sym}^2 M\oplus \ldots \oplus \mathrm{Sym}^n M\oplus \ldots .
\]
Animating this adjunction defines an adjunction between maps of animated commutative rings $A\to B$ and the $\infty$-category of pairs of an animated commutative ring $A$ with an animated $A$-module $M$. (See \cite[Section 25.2.2, Construction 25.2.2.6]{LurieSAG} for a more careful discussion.) We can then also pass to condensed objects, and to the fibre over the given condensed animated commutative ring $A$. This shows that the relevant monad $T$ is given by a monad structure on the functor
\[
M\mapsto \mathrm{Sym}^\bullet M = A\oplus M\oplus \mathrm{Sym}^2 M\oplus \ldots \oplus \mathrm{Sym}^n M\oplus \ldots
\]
on $\mathcal D_{\geq 0}(A)$.

In particular, we note that if $A$ is a condensed animated commutative ring, then by animating $\mathrm{Sym}^n: (A,M)\mapsto (A,\mathrm{Sym}^n M)$, passing to condensed objects, and the fibre over the condensed animated commutative ring $A$, there is a natural functor $\mathrm{Sym}^n$ on $\mathcal D_{\geq 0}(A)$. We need the following lemma.

\begin{lemma} Let $(A,\mathcal M)$ be an analytic animated commutative ring. Let $M\to N$ be a map in $\mathcal D_{\geq 0}(A)$ that becomes an isomorphism after $-\otimes_A (A,\mathcal M)$. Then $\mathrm{Sym}^n M\to \mathrm{Sym}^n N$ becomes an isomorphism after $-\otimes_A (A,\mathcal M)$ for all $n\geq 0$.
\end{lemma}

\begin{proof} Let $Q$ be the cofibre of $M\to N$. By \cite[Construction 25.2.5.4]{LurieSAG}, $\mathrm{Sym}^n N$ admits a filtration whose graded pieces are
\[
\mathrm{Sym}^n M,\mathrm{Sym}^{n-1} M\otimes Q,\ldots,\mathrm{Sym}^{n-i} M\otimes \mathrm{Sym}^i Q,\ldots,\mathrm{Sym}^n Q.
\]
(This can be proved by animating the corresponding statement for usual (finite projective) modules.) Thus, it suffices to see that for all $Q$ in the kernel $K$ of $-\otimes_A (A,\mathcal M)$, also $\mathrm{Sym}^i Q$ lies in $K$ for all $i>0$.

We argue by induction on $i$. Then we see by induction, using the above filtration, that the composite functor $F: Q\mapsto \mathrm{Sym}^i Q\otimes_A (A,\mathcal M)$ from $K$ to $\mathcal D_{\geq 0}(A,\mathcal M)$ commutes with cofibres. In particular, $F(Q[n])\cong F(Q)[n]$ for all $n$, or
\[
\mathrm{Sym}^i Q\otimes_A (A,\mathcal M)\cong \mathrm{Sym}^i(Q[n])[-n]\otimes_A (A,\mathcal M)
\]
for all $n$. Passing to the colimit over $n$, we note that the functor $Q\mapsto \colim_n \mathrm{Sym}^i(Q[n])[-n]$ is an exact endofunctor of $\mathcal D_{\geq 0}(A)$. By Dold--Puppe, \cite[Korollar 12.7]{DoldPuppe}, this functor vanishes unless $i$ is power $p^m$ of a prime $p$, and in the latter case every homology group is killed by $p$. (This is the first ``Goodwillie derivative of $\mathrm{Sym}^i$''.) More precisely, consider first the case $A=\mathbb Z$. Then one gets a colimit-preserving endofunctor
\[
\mathcal D_{\geq 0}(\mathrm{Ab})\to \mathcal D_{\geq 0}(\mathrm{Ab});
\]
this is given by tensoring with an object $R\in \mathcal D_{\geq 0}(\mathbb Z\otimes_{\mathbb S}\mathbb Z)$. Any homotopy group of $R$ is killed by $p$, so by using the Postnikov filtration of $R$, one gets a complete filtration of the Goodwillie derivative of $\mathrm{Sym}^i$, $i=p^m$, where each graded piece is given by a functor of the form
\[
Q\mapsto Q/p\otimes_{\mathbb F_p} V_j[j]
\]
for some $\mathbb F_p$-vector space $V_j=\pi_j R$. By base change, the same statement is true more generally for polynomial algebras $A=\mathbb Z[T_1,\ldots,T_k]$, except that now the associated graded is really given by an $A$-bimodule, and we know that as a left $A$-module it is given by $A/p\otimes_{\mathbb F_p} V_j[j]$. To understand the right $A$-action, we note that for any (usual) $A$-module $M$, multiplication by $T_\ell: M\to M$ induces on $\mathrm{Sym}^i_A M$ the map that is multiplication by $T_\ell^i$, and hence the same is true after animation and passing to the Goodwillie derivative. This ensures that the right $A$-action is the composite of $\mathrm{Frob}_p^m: A\to A/p$ and the left $A$-action (which factors over $A/p$).\footnote{This argument also yields a proof of the result of Dold--Puppe used above.}

Animating this filtration (and passing to condensed objects) then yields a complete filtration of the Goodwillie derivative of $\mathrm{Sym}_A^{p^m}$ over any (condensed) animated commutative ring $A$, whose graded pieces are given by
\[
Q\mapsto (Q\otimes_{A,\mathrm{Frob}_p^m} A/p)\otimes_{\mathbb F_p} V_j[j].
\]
Now the result follows from the definition of analytic animated commutative rings.
\end{proof}

In particular, for any map $M\to N$ in $\mathcal D_{\geq 0}(A)$ that is sent to an isomorphism in $\mathcal D_{\geq 0}(A,\mathcal M)$ also the induced map $\mathrm{Sym}^\bullet M\to \mathrm{Sym}^\bullet N$ is sent to an isomorphism in $\mathcal D_{\geq 0}(A,\mathcal M)$. This implies that $M\mapsto \mathrm{Sym}^\bullet M$ descends to a monad on $\mathcal D_{\geq 0}(A,\mathcal M)$ whose algebras are those condensed animated commutative $A$-algebras whose underlying object lies in $\mathcal D_{\geq 0}(A,\mathcal M)$. Moreover, $-\otimes_A (A,\mathcal M)$ defines a left adjoint to the forgetful functor on algebras for these respective monads, which applied to the unit gives the desired result.

\newpage

\section{Lecture XIII: Analytic spaces}

Finally, we give one possible definition of analytic spaces. In this lecture, we use the term analytic ring to refer to objects of $\AnRing$, i.e.~complete analytic animated commutative rings $(A,\mathcal M)$.\footnote{One can consider a variant using analytic connective $\mathbb E_\infty$-rings instead.} As this theory is an abandoned stepping stone towards analytic stacks in \cite{AnalyticStacks}, we will be lax about the details; in particular, we will ignore set-theoretic issues.

Consider the $\infty$-category of functors from $\AnRing$ to anima, i.e.~the $\infty$-category of presheaves of anima on $\AnRing^{\mathrm{op}}$. By the Yoneda lemma, this admits $\AnRing^{\mathrm{op}}$ as a full subcategory, via the functor
\[
(A,\mathcal M)\mapsto \AnSpec(A,\mathcal M) : (B,\mathcal N)\mapsto \Hom_\AnRing((A,\mathcal M),(B,\mathcal N)).
\]
We endow $\AnRing^{\mathrm{op}}$ with the Grothendieck topology generated by the families of maps
\[
\{\AnSpec(A_i,\mathcal M_i)\to \AnSpec(A,\mathcal M)\}
\]
whenever $(A,M)\to (A_i,\mathcal M_i)$ is a finite family of steady localizations such that
\[
\mathcal D(A,\mathcal M)\to \prod_i \mathcal D(A_i,\mathcal M_i)
\]
is conservative.\footnote{Warning: This Grothendieck topology is incomparable to the $!$-topology used in \cite{AnalyticStacks}: On the one hand, $!$-covers are much more general, but on the other hand covers in the present sense are not guaranteed to be $!$-covers.} We note that pullbacks of such finite families of maps are of the same form.

\begin{proposition} For any analytic ring $(A,\mathcal M)$, the functor $\AnSpec(A,\mathcal M)$ is a sheaf on $\AnRing^{\mathrm{op}}$.
\end{proposition}

\begin{proof} This follows directly from Proposition~\ref{prop:descent}.
\end{proof}

If $f: \mathcal F\to \mathcal G$ is a map of functors from $\AnRing$ to anima, we write $\mathcal F\subset \mathcal G$ if for all analytic rings $(A,\mathcal M)$, all fibres of 
\[
\mathcal F(A,\mathcal M)\to \mathcal G(A,\mathcal M)
\]
are either empty or contractible.

\begin{definition} Let $(A,\mathcal M)$ be an analytic ring. A steady subspace of $\AnSpec(A,\mathcal M)$ is a subfunctor $U\subset \AnSpec(A,\mathcal M)$ such that the natural map
\[
\colim_{\AnSpec(A',\mathcal M')\subset U} \AnSpec(A',\mathcal M')\to U
\]
is an isomorphism, where the colimit runs over all steady localizations $(A,\mathcal M)\to (A',\mathcal M')$ for which $\AnSpec(A',\mathcal M')\subset \AnSpec(A,\mathcal M)$ factors over $U$.
\end{definition}

\begin{remark} As the colimit is taken over a non-filtered index category, one might worry that it may produce undesired outcomes. This is not so: Compute the colimit first as presheaves. Then for all analytic rings $(B,\mathcal N)$ with a fixed map to $U$, the fibre product computes the colimit of a point over all $\AnSpec(A',\mathcal M')$ contained in $U$ and such that $\AnSpec(B,\mathcal N)\to U$ factors over $\AnSpec(A',\mathcal M')$. If nonempty, the index category is cofiltered (as with any two steady localizations, it contains their tensor product), so the geometric realization of the index category is contractible, and hence this colimit is itself a point or empty. Now note that if $\mathcal G$ is a sheaf of anima and $\mathcal F\subset \mathcal G$ is a presheaf of anima (the inclusion meaning that all fibres are empty or contractible), then the sheafification $\mathcal F'$ of $\mathcal F$ still has the property that $\mathcal F'\to \mathcal G$ has empty or contractible fibres; it is the subsheaf of $\mathcal G$ of all sections that locally lift to $\mathcal F$. To see this, note that sheafification commutes with taking fibres, so it is enough to recall that the sheafification of a presheaf whose values are empty or a point has values which still are empty or a point.

This argument shows that $U\subset \AnSpec(A,\mathcal M)$ is a steady subspace if and only if there is some collection of steady localizations $(A,\mathcal M)\to (A_i,\mathcal M_i)_i$ such that $U$ is the image of
\[
\bigsqcup_i \AnSpec(A_i,\mathcal M_i)\to \AnSpec(A,\mathcal M).
\]
\end{remark}

One verifies immediately that pullbacks of steady subspaces are steady, and conversely $U\subset \AnSpec(A,\mathcal M)$ is a steady subspace if and only the pullbacks $U_i\subset \AnSpec(A_i,\mathcal M_i)$ are steady subspaces for all terms in a finite cover of $(A,\mathcal M)$ by steady localizations $(A_i,\mathcal M_i)$. Thus, the following definition has reasonable properties.

\begin{definition} An inclusion $\mathcal F\subset \mathcal G$ of sheaves of anima on $\AnRing$ is a steady subspace if for all analytic rings $(A,\mathcal M)$ the fibre product
\[
\mathcal F\times_{\mathcal G} \AnSpec(A,\mathcal M)\to \AnSpec(A,\mathcal M)
\]
is a steady subspace.
\end{definition}

Finally, we can define analytic spaces.

\begin{definition} An analytic space is a sheaf of anima $X$ on $\AnRing$ such that the map
\[
\colim_{\AnSpec(A,\mathcal M)\subset X} \AnSpec(A,\mathcal M)\to X
\]
is an isomorphism, where the colimit runs over all affine steady subspaces $\AnSpec(A,\mathcal M)\subset X$.
\end{definition}

Again, one might worry that the colimit does not behave well as the index category is not well-behaved. However, we have the following result, saying that we could equivalently define analytic spaces $X$ as those sheaves that can be covered by steady subspaces of the form $\AnSpec(A,\mathcal M)\subset X$.

\begin{proposition} If $X$ is a sheaf of anima on $\AnRing$ such that the map
\[
\colim_{\AnSpec(A,\mathcal M)\subset X} \AnSpec(A,\mathcal M)\to X
\]
is surjective, where as above the colimit runs over all affine steady subspaces $\AnSpec(A,\mathcal M)\subset X$, then $X$ is an analytic space, i.e.~the map is an isomorphism.

In particular, any steady subspace of an analytic space is again an analytic space.
\end{proposition}

\begin{proof} Assume first that $X$ admits an injection into some affine $\tilde{X} = \AnSpec(\widetilde{A},\widetilde{\mathcal M})$; let us call such $X$ quasi-affine for the rest of this proof. In that case, fibres products over $X$ can also be computed as fibre products over $\tilde{X}$. For any $(B,\mathcal N)$ mapping to $X$, the fibre of the displayed map is computed by the colimit over a point over the index category of all $\AnSpec(A,\mathcal M)\subset X$ over which $\AnSpec(B,\mathcal N)\to X$ factors, which will be cofiltered in this case. Thus, this colimit is either empty or a point, which gives the result in this case.

At this point, we already see that any steady subspace of some quasi-affine $X$ is again an analytic space (and clearly quasi-affine). Now let $X$ be general and consider the colimit
\[
\colim_{Y\subset X} Y\to X
\]
where the colimit runs over all quasi-affine analytic spaces $Y$ that are steady subspaces of $X$. By the same argument as in the first paragraph, this map is an isomorphism. Finally, the map
\[
\colim_{\AnSpec(A,\mathcal M)\subset X} \AnSpec(A,\mathcal M)\to \colim_{Y\subset X} Y
\]
is an isomorphism, by comparing both with
\[
\colim_{\AnSpec(A,\mathcal M)\subset Y\subset X} \AnSpec(A,\mathcal M):
\]
This is equal to the first, as the set of possible $Y$ is cofiltered, and it is equal to the second as each $Y$ can be written as $\colim_{\AnSpec(A,\mathcal M)\subset Y} \AnSpec(A,\mathcal M)$.
\end{proof}

By Proposition~\ref{prop:descent}, we can define an $\infty$-category of quasicoherent sheaves on $X$.

\begin{definition} Let $X$ be an analytic space. The $\infty$-category of quasicoherent sheaves on $X$ is
\[
\mathcal D(X) := \lim_{\AnSpec(A,\mathcal M)\subset X} \mathcal D(A,\mathcal M).
\]
\end{definition}

Equivalently, the association $(A,\mathcal M)\mapsto \mathcal D(A,\mathcal M)$ defines a sheaf of $\infty$-categories on $\AnRing^{\mathrm{op}}$, and $\mathcal D(X)$ is the value of this sheaf on $X$. In particular:

\begin{proposition} Let $X$ be an analytic space and let $U_i\subset X$, $i\in I$, be steady subspaces that cover $X$ (i.e.~$\bigsqcup_i U_i\to X$ is a surjection). Let $C_I$ be the category of finite nonempty subsets of $I$, and consider the functor taking $J\in C_I$ to $U_J=\bigcap_{i\in J} U_i$. Then
\[
\mathcal D(X)\to \lim_{J\in C_I} \mathcal D(U_J)
\]
is an equivalence of $\infty$-categories.

If $X=\AnSpec(A,\mathcal M)$ is affine, then $\mathcal D(X)\cong \mathcal D(A,\mathcal M)$.
\end{proposition}

We see that there is a well-defined $\infty$-category of analytic spaces, defined very analogously to schemes, and it comes with an $\infty$-category of quasicoherent sheaves, satisfying descent.

It is high time for some examples. First, we need some examples of steady localizations. All of them will be produced by the following criterion.

\begin{proposition}\label{prop:proringatinfinity} Let $f: (A,\mathcal M)\to (B,\mathcal N)$ be a map of analytic rings. Assume that for any extremally disconnected profinite set $S$ and any $M\in \mathcal D_{\geq 0}(A,\mathcal M)$, the natural map
\[
(\intHom_{\mathcal D(A)}(A[S],M)\otimes_{(A,\mathcal M)} (B,\mathcal N))(\ast)\to (M\otimes_{(A,\mathcal M)} (B,\mathcal N))(S)
\]
in $\mathcal D(\Ab)$ is an equivalence. Then $f$ is steady.
\end{proposition}

\begin{proof} Consider any analytic ring $(C,\mathcal M_C)$ over $(A,\mathcal M)$ and let $M\in \mathcal D_{\geq 0}(C,\mathcal M_C)$. Then
\[
(M\otimes_{(A,\mathcal M)} (B,\mathcal N))(S) = (\intHom_{\mathcal D(A)}(A[S],M)\otimes_{(A,\mathcal M)} (B,\mathcal N))(\ast) = (\intHom_{\mathcal D(C,\mathcal M_C)}(\mathcal M_C[S],M)\otimes_{(A,\mathcal M)} (B,\mathcal N))(\ast).
\]
Varying $S$, the right-hand side of this formula defines a functor from the $\infty$-category of $C$-modules isomorphic to some $\mathcal M_C[S]$ towards anima, commuting with finite products, and thus an object of $\mathcal D_{\geq 0}(C,\mathcal M_C)$, that necessarily agrees with $M\otimes_{(A,\mathcal M)} (B,\mathcal N)$.
\end{proof}

An interesting question is whether an induced analytic ring structure is steady. This leads to the following class of modules. Here and in the following, we write
\[
P^\vee = \intHom_{\mathcal D(A)}(P,A)\in \mathcal D(A)
\]
for $P\in \mathcal D(A)$.

\begin{definition}\label{def:nuclear} Let $(A,\mathcal M)$ be an analytic ring. An object $M\in \mathcal D(A,\mathcal M)$ is nuclear if for all extremally disconnected profinite sets $S$, the natural map
\[
(A[S]^\vee\otimes_{(A,\mathcal M)} M)(\ast)\to M(S)
\]
in $\mathcal D(\Ab)$ is an isomorphism.
\end{definition}

Let us characterize this class of modules.\footnote{The general yoga of nuclear objects and trace-class maps in symmetric monoidal categories appears now in many places in the literature: We discussed it in more detail in \cite{Complex}, and it has also been used in \cite{Efimov} and related works.} For the characterization, the following class of maps is critical.

\begin{definition}\label{def:traceclass} Let $(A,\mathcal M)$ be an analytic ring. A map $f: P\to Q$ between compact objects of $\mathcal D(A,\mathcal M)$ is trace-class if there is some map
\[
g: A\to P^\vee\otimes_{(A,\mathcal M)} Q
\]
such that $f$ is the composite
\[
P\xrightarrow{1\otimes g} P\otimes_{(A,\mathcal M)} P^\vee\otimes_{(A,\mathcal M)} Q\to Q
\]
where the second map contracts the first two factors.
\end{definition}

\begin{definition}\label{def:basicnuclear} Let $(A,\mathcal M)$ be an analytic ring. An object $M\in \mathcal D(A,\mathcal M)$ is basic nuclear if it can be written as a sequential colimit
\[
M=\colim(P_0\to P_1\to \ldots)
\]
of compact $P_i\in \mathcal D(A,\mathcal M)$ along trace-class maps.
\end{definition}

\begin{proposition}\label{prop:nuclear} Let $(A,\mathcal M)$ be an analytic ring and let $M\in \mathcal D(A,\mathcal M)$. Then $M$ is nuclear if and only if it can be written as a filtered colimit of basic nuclear objects. The class of basic nuclear objects is stable under all countable colimits, and the class of nuclear objects is stable under all colimits.
\end{proposition}

\begin{proof} Assume first that $M$ is basic nuclear, so $M$ is a sequential colimit of compact $P_i$ along trace-class maps. Choose maps $A\to P_i^\vee\otimes_A P_{i+1}$ giving rise to these trace-class maps $P_i\to P_{i+1}$. Let $Q=A[S]$. Then nuclearity means that $(Q^\vee\otimes_{(A,\mathcal M)} M)(\ast)\to \Hom_{\mathcal D(A)}(Q,M)$ is an isomorphism (in $\mathcal D(\Ab)$). Both sides commute with colimits in $M$, and there are natural backwards maps
\[
\Hom_{\mathcal D(A)}(Q,P_i)\to (\intHom_{\mathcal D(A)}(Q,P_i)\otimes_{(A,\mathcal M)} P_i^\vee\otimes_{(A,\mathcal M)} P_{i+1})(\ast)\to (Q^\vee\otimes_{(A,\mathcal M)} P_{i+1})(\ast),
\]
showing that $M$ is nuclear.

Next, we show that the class of basic nuclear $M$ is stable under finite colimits; equivalently, under passage to cones. Let $f: M\to M'$ be a map of basic nuclear objects, and write $M$ resp.~$M'$ as a filtered colimit of $P_i$ resp.~$P_i'$ as in the definition. As all $P_i$ are compact, the map $M\to M'$ can be realized via compatible maps $P_i\to P_i'$, up to reindexing the $P_i'$. Let $Q_i$ be the cone of $P_i\to P_i'$, so the cone of $f$ is the sequential colimit of the $Q_i$. We claim that $Q_i\to Q_{i+2}$ is trace-class for all $i$. To see this, we need to find a map $A\to Q_i^\vee\otimes_{(A,\mathcal M)} Q_{i+2}$ defining this map. Note that $Q_i^\vee\otimes_{(A,\mathcal M)} Q_{i+2}$ is the fibre of
\[
(P_i')^\vee\otimes_{(A,\mathcal M)} Q_{i+2}\to P_i^\vee\otimes_{(A,\mathcal M)} Q_{i+2}
\]
and to define a map there, it suffices to exhibit a section of $(P_i')^\vee\otimes_{(A,\mathcal M)} P_{i+2}'$ and a section of $P_i^\vee\otimes_{(A,\mathcal M)} P_{i+2}$ whose images in $P_i^\vee\otimes_{(A,\mathcal M)} P_{i+2}'$ agree. To define this, pick sections
\[
\alpha: A\to P_i^\vee\otimes_{(A,\mathcal M)} P_{i+1}\ ,\ \beta: A\to (P_{i+1}')^\vee\otimes_{(A,\mathcal M)} P_{i+2}'
\]
witnessing that $P_i\to P_{i+1}$ and $P_{i+1}'\to P_{i+2}'$ are trace-class, and take the induced sections of $P_i^\vee\otimes_{(A,\mathcal M)} P_{i+2}$ and $(P_i')^\vee\otimes_{(A,\mathcal M)} P_{i+2}'$.

To see stability of basic nuclear $M$ under all countable colimits, it remains to observe that it is stable under countable direct sums, which is clear from the definition.

It is clear from the definition that the class of all nuclear objects is stable under all colimits. Now let $M$ be a general nuclear object. We can look at the $\omega_1$-filtered diagram of all basic nuclear $M'$ mapping to $M$ (here we use that basic nuclear objects admit countable colimits), and replacing $M$ by the cone of the colimit of this diagram mapping to $M$, we can assume that there are no maps from basic nuclear $M'$'s to $M$. We need to see that then $M=0$. If not, there is some compact $P_0$ with a nonzero map $P_0\to M$. As by nuclearity,
\[
(P_0^\vee\otimes_A M)(\ast) = \Hom(P_0,M),
\]
and both sides commute with filtered colimits as functors of $M$, we can find some compact $P_1\to M$ and a section of $(P_0^\vee\otimes_A P_1)(\ast)$ so that $P_0\to M$ is the composite of $P_1\to M$ with the corresponding trace-class map $P_0\to P_1$. Continuing, we find a nonzero map from some basic nuclear $\colim(P_0\to P_1\to\ldots)$ to $M$, contradiction.
\end{proof}

The proposition implies that nuclear modules have stronger properties than a priori guaranteed by the definition:

\begin{proposition}\label{prop:nuclearsteady} Let $(A,\mathcal M)$ be an analytic ring and let $M\in \mathcal D(A,\mathcal M)$ be nuclear. Then for all $M'\in \mathcal D(A,\mathcal M)$ and extremally disconnected profinite sets $S$, the map
\[
(\intHom_{\mathcal D(A))}(A[S],M')\otimes_{(A,\mathcal M)} M)(\ast)\to (M'\otimes_{(A,\mathcal M)} M)(S)
\]
in $\mathcal D(\Ab)$ is an isomorphism.

In particular, if $A\to B$ is a map of condensed animated commutative rings such that $B\in \mathcal D(A,\mathcal M)$ is nuclear, and $(B,\mathcal N)$ is the induced analytic ring structure on $B$, then $(A,\mathcal M)\to (B,\mathcal N)$ satisfies the criterion from Proposition~\ref{prop:proringatinfinity}, and hence is steady.
\end{proposition}

\begin{proof} As both sides commute with colimits in $M$, we can assume that $M$ is basic nuclear, and write it as the sequential colimit of compact $P_i$ along trace-class maps $P_i\to P_{i+1}$. In that case, there are backwards maps
\[\begin{aligned}
(M'\otimes_{(A,\mathcal M)} P_i)(S)&=\intHom_{\mathcal D(A)}(A[S],M'\otimes_{(A,\mathcal M)} P_i)(\ast)\\
&\to (\intHom_{\mathcal D(A)}(A[S],M'\otimes_{(A,\mathcal M)} P_i)\otimes_{(A,\mathcal M)} P_i^\vee\otimes_{(A,\mathcal M)} P_{i+1})(\ast)\\
&\to (\intHom_{\mathcal D(A)}(A[S],M'\otimes_{(A,\mathcal M)} P_i\otimes_{(A,\mathcal M)} P_i^\vee)\otimes_{(A,\mathcal M)} P_{i+1})(\ast)\\
&\to (\intHom_{\mathcal D(A))}(A[S],M')\otimes_{(A,\mathcal M)} P_{i+1})(\ast)
\end{aligned}
\]
and so passing to the filtered colimit over the $P_i$ gives the desired equivalence.

The final sentence follows directly.
\end{proof}

We will now verify this condition in the examples listed two lectures ago.

\begin{example} The following are steady localizations of analytic rings.
\begin{enumerate}
\item If $A\to B=A[S^{-1}]$ is a usual localization of (discrete) rings, with trivial analytic ring structure, then $A\to B$ is a steady localization. Indeed, $B$, like any discrete module, is nuclear over $A$.
\item The map $(\mathbb Z[T],\mathbb Z)_\solid\to (\mathbb Z[T],\mathbb Z[T])_\solid$, corresponding to localization to the subset $\{|T|\leq 1\}$ of the adic space $\Spa(\mathbb Z[T],\mathbb Z)$. Recall that in this case (cf.~\cite[Lecture VII]{Condensed}),
\[
M\otimes_{(\mathbb Z[T],\mathbb Z)_\solid} (\mathbb Z[T],\mathbb Z[T])_\solid = \intHom_{\mathcal D((\mathbb Z[T],\mathbb Z)_\solid)}((\mathbb Z((T^{-1}))/\mathbb Z[T])[-1],M).
\]
Then
\[
(\intHom_{\mathcal D(\mathbb Z[T])}(\mathbb Z[T][S],M)\otimes_{(\mathbb Z[T],\mathbb Z)_\solid} (\mathbb Z[T],\mathbb Z[T])_\solid)(\ast) = \Hom_{\mathcal D(\mathbb Z[T])}((\mathbb Z((T^{-1}))/\mathbb Z[T])[-1]\otimes_{\mathbb Z}\mathbb Z_\solid[S],M)
\]
while also
\[
(M\otimes_{(\mathbb Z[T],\mathbb Z)_\solid} (\mathbb Z[T],\mathbb Z[T])_\solid)(S) = \Hom_{\mathcal D((\mathbb Z[T],\mathbb Z)_\solid)}((\mathbb Z((T^{-1}))/\mathbb Z[T])[-1]\otimes_{\mathbb Z}\mathbb Z_\solid[S],M),
\]
giving the desired result.
\item Any rational subset of an adic space is of the form $U=\{|f_i|\leq |g|\neq 0\}$ for finitely many elements $f_1,\ldots,f_n,g\in A$ (satisfying some condition). In the present situation of analytic rings, we can realize this localization by first inverting $g$, and then asking the condition $|\frac{f_i}g|\leq 1$. Inverting $g$ is handled by a base change of (1), while asking $|\frac{f_i}g|\leq 1$ is handled by a base change of (2).
\item If $D_1\subset D_2$ is an inclusion of closed discs in $\mathbb C$, then $\mathcal O(D_2)\to \mathcal O(D_1)$ (endowed with the analytic ring structure induced from the $p$-liquid structure on $\mathbb C$ for some chosen $0<p\leq 1$) is a steady localization. This follows from nuclearity of $\mathcal O(D_2)\to \mathcal O(D_1)$.
\end{enumerate}
\end{example}

Regarding (3), let us at least recall how general Huber pairs fit into the present framework.

\begin{proposition} There is a fully faithful functor $(A,A^+)\mapsto (A,A^+)_\solid$ from the category of Huber pairs $(A,A^+)$ to the $\infty$-category $\AnRing$, defined as follows. For a finitely generated $\mathbb Z$-algebra $R$, let $R_\solid$ be the analytic ring structure on $R$ given by
\[
R_\solid[S] = \varprojlim_i R[S_i]
\]
for a profinite set $S=\varprojlim_i S_i$. For a general discrete $\mathbb Z$-algebra $R$, let $R_\solid = \varinjlim_{R'\to R} R'_\solid$, where the colimit is over all finitely generated $\mathbb Z$-algebras $R'$ mapping to $R$. Finally,
\[
(A,A^+)_\solid = \underline{A}\otimes_{A_{\mathrm{disc}}^+} (A^+_{\mathrm{disc}})_\solid
\]
is an induced analytic ring structure.
\end{proposition}

\begin{remark} A priori, this induced analytic ring structure might be animated. However, Andreychev has proved that $(A,A^+)_\solid[S]$ is always $0$-truncated, \cite[Section 3.3]{Andreychev}.

The analytic ring $(A,A^+)_\solid$ is always complete; this follows from $\underline{A}$ being a module over $(A^+_{\mathrm{disc}})_\solid$, which can be deduced by writing it as the limit of $A/I^n$ where $I\subset A_0\subset A$ is an ideal of definition that is also an $R$-module, where $R\subset A^+$ is any chosen finitely generated subalgebra.
\end{remark}

\begin{remark} The second part of a Huber pair is an open and integrally closed subring $A^+\subset A$ consisting of powerbounded elements. One may wonder how these precise conditions arise. One answer is the following: Consider all possible complete analytic ring structures on $\underline{A}$ admitting a map from $\mathbb Z_\solid$ and that can be obtained via filtered colimits from pushouts of the localization $(\mathbb Z[T],\mathbb Z)_\solid\to (\mathbb Z[T],\mathbb Z[T])_\solid$. Then this class is canonically in bijection with the open and integrally closed subrings $A^+\subset A$ contained in the powerbounded elements. Given $A^+\subset A$, one takes the pushout of $(\mathbb Z[T],\mathbb Z)_\solid\to (\mathbb Z[T],\mathbb Z[T])_\solid$ along any map $\mathbb Z[T]\to A$ sending $T$ to an element of $A^+$.
\end{remark}

\begin{proof} That the functor is well-defined follows from the results of \cite{Condensed} and the result on induced analytic ring structures. For fully faithfulness, let $(A,A^+)$ and $(B,B^+)$ be Huber pairs. By definition of maps of analytic rings, maps from $(A,A^+)_\solid$ to $(B,B^+)_\solid$ embed into maps from $\underline{A}$ to $\underline{B}$. As $A$ and $B$ are metrizable, these agree with maps from $A$ to $B$. We need to see that a map $A\to B$ sends $A^+$ into $B^+$ if and only if it induces a map of analytic rings $(A,A^+)_\solid\to (B,B^+)_\solid$. This is clear in the forward direction.

Thus, assume that $f: A\to B$ is a map of Huber rings that induces a map of analytic rings $(A,A^+)_\solid\to (B,B^+)_\solid$, but does not send $A^+$ into $B^+$. Take any $g\in A^+$ such that $f(g)\not\in B^+$. We get a map $(\mathbb Z[T],\mathbb Z[T])\to (A,A^+)$ sending $T$ to $g$; precomposing with this map, we may assume that $(A,A^+)=(\mathbb Z[T],\mathbb Z[T])$. Note that the localization from $(\mathbb Z[T],\mathbb Z)_\solid$ to $(\mathbb Z[T],\mathbb Z[T])_\solid$ is obtained by killing the algebra object $\mathbb Z((T))^{-1}$. Thus, our assumption is that
\[
\mathbb Z((T^{-1}))\otimes_{(\mathbb Z[T],\mathbb Z)_\solid} (B,B^+)_\solid = 0.
\]
Asking whether an algebra is zero is just asking whether $1=0$ in this algebra. If this happens in a filtered colimit, it happens at a finite stage. Thus, we can then find some finitely generated subalgebra $R\subset B^+$ such that
\[
\mathbb Z((T^{-1}))\otimes_{(\mathbb Z[T],\mathbb Z)_\solid} (B,R)_\solid = 0.
\]
Using the resolution
\[
0\to \prod_{\mathbb N} \mathbb Z\otimes_{\mathbb Z} \mathbb Z[T]\xrightarrow{1-\mathrm{shift}\otimes T} \prod_{\mathbb N} \mathbb Z\otimes_{\mathbb Z} \mathbb Z[T]\to \mathbb Z((T^{-1}))\to 0,
\]
this means that the map
\[
\prod_{\mathbb N} R\otimes_{R_\solid} \underline{B}\xrightarrow{1-\mathrm{shift}\otimes g} \prod_{\mathbb N} R\otimes_{R_\solid} \underline{B}
\]
is an isomorphism. We need to see that this implies that $g$ is integral over $R$.

In other words, we now start with a finitely generated $\mathbb Z$-algebra $R$, some Huber $R$-algebra $B$, and an element $g\in B$. Choose a ring of definition $B_0\subset B$ containing $R$ and some ideal of definition $I\subset B_0$. Then in particular $\underline{B} = \lim_n B/I^n$ as condensed $R$-modules. We want to show that if
\[
\prod_{\mathbb N} R\otimes_{R_\solid} \underline{B}\xrightarrow{1-\mathrm{shift}\otimes g} \prod_{\mathbb N} R\otimes_{R_\solid} \underline{B}
\]
is an isomorphism, then $g$ is integral over $R+I$ (which is a subring of $B$ and contained in $B^+$). For any $n$,
\[
\prod_{\mathbb N} R\otimes_R B/I^n = \varinjlim_{M\subset B/I^n} \prod_{\mathbb N} R\otimes_R M = \varinjlim_{M\subset B/I^n} \prod_{\mathbb N} M\subset \prod_{\mathbb N} R/I^n
\]
where the colimit runs over finitely generated $R$-submodules $M$ of $B/I^n$. We see that the inverse limit over all $n$ maps to $\prod_{\mathbb N} B$, and the image is contained in sequences $(b_0,b_1,\ldots)$ such that for all $n$, the submodule of $B/I^n$ generated by all $b_i$ is finitely generated.

If the map is surjective, then in particular $1$ is in the image, which means that we can find such a sequence $(b_0,b_1,\ldots)$ in $B$ satisfying $b_0=1$, $b_1-gb_0=0$, $b_2-gb_1=0$, etc., i.e.~$b_i=g^i$. This sequence needs to have the property that the $R$-submodule of $B/I$ generated by all $g^i$ is finitely generated. But this means that there is some relation
\[
g^n + r_{n-1} g^{n-1} + \ldots + r_0 = 0\in B/I,
\]
which as desired shows that $g$ is integral over $R+I$.
\end{proof}

In the context of adic spaces, one can show that nuclearity is closely related to adic maps; for example, if $A\to B$ is a map of Tate-Huber rings (i.e.~Huber rings admitting a topologically nilpotent unit), then $B$ is always a nuclear $A$-module, for any choice of $A^+$. On the other hand, if $A\to B$ is not adic, then $B$ is never nuclear as $A$-module, and in fact one can show that for any choices of $A^+$ mapping to $B^+$, the map $(A,A^+)_\solid\to (B,B^+)_\solid$ will not be steady.

\newpage

\section{Lecture XIV: Varia}

In this final lecture, we briefly outline a few more ideas. We do not give proofs here; however, they are not especially difficult. (The hard work was in obtaining the examples of the solid analytic ring structure on $\mathbb Z$ last semester, and the liquid analytic ring structures on $\mathbb R$ this semester.)

First, one may wonder whether one can define analytic spaces more in the spirit of locally ringed topological spaces. This is possible. If $X$ is an analytic space, then we can consider the site $|X|$ of steady subspaces $U\subset X$ of $X$. This forms a locale. If $X=\AnSpec(A,\mathcal M)$ is affine, then a basis is given by the affine steady subspaces, which are quasicompact and quasiseparated. Thus, Deligne's theorem on the existence of points implies that, if there is only a set (instead of a proper class) of affine steady subspaces of $X$, then $|X|$ can be identified with the site of open subsets of a spectral space. For general (non-affine) $X$, it would then be the case that $|X|$ is a locally spectral space.

Moreover, $|X|$ carries a sheaf $\mathcal O_X$ of condensed animated commutative rings (obtained by sheafifying its value on affine subspaces), as well as a sheaf $\mathcal M_X[S]$ of condensed animated $\mathcal O_X$-modules, for every extremally disconnected profinite set $S$ (functorial in $S$, taking finite disjoint unions to finite direct sums). We want to define what ``locally ringed'' means, i.e.~we want an analogue of the condition that the stalks of $\mathcal O_X$ are local.

\begin{definition} A locally analytically ringed locale is a triple $(X,\mathcal O_X,\mathcal M_X)$ consisting of a locale $X$, a sheaf of condensed animated commutative rings $\mathcal O_X$ on $X$, and a functor $\mathcal M_X: S\mapsto \mathcal M_X[S]$ from extremally disconnected profinite sets $S$ to condensed animated $\mathcal O_X$-modules, taking finite disjoint unions to finite direct sums, subject to the following conditions:
\begin{enumerate}
\item (``analytically ringed'') there is a basis of subsets $U\subset X$ for which $(\mathcal O_X(U),\mathcal M_X(U))$ is an analytic ring and such that on this basis, the restriction maps are maps of analytic rings;
\item (``stalks are local'') for every $U\subset X$ with a map $(A,\mathcal M)\to (\mathcal O_X(U),\mathcal M_X(U))$ from some analytic ring $(A,\mathcal M)$, and every steady cover
\[
\bigsqcup_i \AnSpec(A_i,\mathcal M_i)\to \AnSpec(A,\mathcal M)\,
\]
there is a cover of $U$ by $U_i\subset U$ such that $(A,\mathcal M)\to (\mathcal O_X(U_i),\mathcal M_X(U_i))$ factors over $(A_i,\mathcal M_i)$.
\end{enumerate}

A map $f: (Y,\mathcal O_Y,\mathcal M_Y)\to (X,\mathcal O_X,\mathcal M_X)$ of locally analytically ringed locales is a map $f: Y\to X$ of locales together with a map $f^\ast\mathcal O_X\to \mathcal O_Y$ of sheaves of condensed animated commutative rings, with the following properties.
\begin{enumerate}
\item (``analytically ringed map'') If on basic affine opens $V\subset Y$ maps into $U\subset X$, then $(\mathcal O_X(U),\mathcal M_X(U))\to (\mathcal O_Y(V),\mathcal M_Y(V))$ is a map of analytic rings.
\item (``maps on local rings local'') Moreover, in the same situation, if there is some steady localization
\[
(A,\mathcal M)=(\mathcal O_X(U),\mathcal M_X(U))\to (A',\mathcal M')
\]
over which $(A,\mathcal M)\to (\mathcal O_Y(V),\mathcal M_Y(V))$ factors, then one can find $U'\subset U$ containing the image of $V$ such that $(A,\mathcal M)\to (\mathcal O_X(U'),\mathcal M_X(U'))$ factors over $(A',\mathcal M')$.
\end{enumerate}
\end{definition}

It is now not hard to prove the following proposition.

\begin{proposition} The functor $X\mapsto (|X|,\mathcal O_X,\mathcal M_X)$ from analytic spaces to locally analytically ringed locales is well-defined and fully faithful. An object lies in the essential image if and only if it is locally isomorphic to the locally analytically ringed locale associated to some analytic ring $(A,\mathcal M)$.
\end{proposition}

Using this language, we can explain how to embed complex-analytic spaces into analytic spaces over $\mathbb C$ equipped with its $p$-liquid analytic ring structure $(\mathbb C,\mathcal M_{<p})$ (for any fixed $0<p\leq 1$).

\begin{proposition} Let $(X,\mathcal O_X)$ be a complex-analytic space; so $X$ is a topological space and $\mathcal O_X$ is a sheaf of condensed $\mathbb C$-algebras, locally Zariski closed in an open $n$-dimensional complex ball. Consider the functor taking any analytic space $Y$ over $(\mathbb C,\mathcal M_{<p})$ to maps of locales $f: |Y|\to X$ together with a map of sheaves of condensed animated $\mathbb C$-algebras $f^\ast \mathcal O_X\to \mathcal O_Y$. This functor is representable by an analytic space $X_p^{\mathrm{an}}$ over $\AnSpec (\mathbb C,\mathcal M_{<p})$.
\end{proposition}

Let us give an example: Let $X=\{z\mid |z|<1\}$ be the open unit disc in $\mathbb C$. Let $D_r = \{z\mid |z|\leq r\}\subset X$ be the closed discs of radius $r<1$. Let $\mathcal O(D_r)$ be the condensed $\mathbb C$-algebra of overconvergent holomorphic functions on $D_r$. This is a nuclear algebra over $(\mathbb C,\mathcal M_{<p})$. Endow $\mathcal O(D_r)$ with the analytic ring structure induced from $(\mathbb C,\mathcal M_{<p})$. Then $X_p^{\mathrm{an}}$ is the increasing union of $\AnSpec \mathcal O(D_r)$ over all $r<1$.

In fact, similar results hold true for real-analytic, smooth, or topological manifolds: If $(X,\mathcal O_X)$ is one such, then the functor taking an analytic space $Y$ over $(\mathbb R,\mathcal M_{<p})$ to maps of locales $f: |Y|\to X$ together with a map of sheaves of condensed animated $\mathbb R$-algebras $f^\ast \mathcal O_X\to \mathcal O_Y$ is representable by an analytic space $X_p^{\mathrm{an}}$ over $\AnSpec(\mathbb R,\mathcal M_{<p})$. The explicit description of the functor is similar, using the affine analytic spaces associated with overconvergent real-analytic, smooth, or continuous functions on a closed ball.

\begin{warning} In the topological case, the functor does not land in steady analytic spaces over $\AnSpec(\mathbb R,\mathcal M_{<p})$, and the functor does not commute with products (it does in the other cases). The issue is that for example $C(S^1,\mathbb R)\otimes_{(\mathbb R,\mathcal M_{<p})}^L C(S^1,\mathbb R)$ is not given by $C((S^1)^2,\mathbb R)$ (and non-steadyness comes from non-nuclearity of $C(S^1,\mathbb R)$).
\end{warning}

We see that analytic spaces over $\AnSpec(\mathbb R,\mathcal M_{<p})$ are able to handle all flavours of real or complex manifolds simultaneously. In fact, nothing is holding one back from even mixing them, taking products of real-analytic manifolds with topological manifolds, etc.~. Even more, one can also mix with algebraic varieties:

\begin{proposition} For any analytic ring $(A,\mathcal M)$, there is a faithful functor $X\mapsto X^{\mathrm{an}}$ from schemes $X$ over the (animated) ring $A(\ast)$ to analytic spaces over $\AnSpec(A,\mathcal M)$. In fact, $X^{\mathrm{an}}$ represents the functor taking an analytic space $Y$ over $\AnSpec(A,\mathcal M)$ to maps $f: |Y|\to X$ together with a map $f^\ast \mathcal O_X\to \mathcal O_Y$ of sheaves of animated $A(\ast)$-algebras that induces local maps on local rings. Moreover, there is a fully faithful functor $\mathcal D_{\mathrm{qc}}(X)\hookrightarrow \mathcal D(X^{\mathrm{an}})$.
\end{proposition}

\begin{warning} In a previous version, we claimed that this functor is fully faithful. This may not be the case, for interesting reasons in fact.
\end{warning}

\begin{remark} The analytification of $\mathbb A^1_{\mathbb C}$ in this sense is $\AnSpec(\mathbb C[T],\mathcal M_{<p}[T])$, taking the usual ring of polynomials over $\mathbb C$ with the condensed structure and the analytic structure induced from $(\mathbb C,\mathcal M_{<p})$. This is not what is usually regarded as the analytification of $\mathbb A^1_{\mathbb C}$, which would be the increasing union of open balls of radius $r$ around the origin, letting $r\to \infty$. In the present situation, this is still a steady subspace of $\AnSpec(\mathbb C[T],\mathcal M_{<p}[T])$. The latter is however larger, and contains points ``infinitesimally close to infinity''.

We see that the present framework of analytic spaces makes it possible to have analytic spaces that are globally algebraic but locally analytic -- in the sense that the global functions of $\AnSpec(\mathbb C[T],\mathcal M_{<p}[T])$ are just the polynomials $\mathbb C[T]$, but locally one gets certain algebras of convergent functions.
\end{remark}

Let us now turn to adic spaces.

\begin{proposition} Let $(A,A^+)$ be a Huber pair and $X=\AnSpec((A,A^+)_\solid)$. The underlying locale $|X|$ maps to the topological space $\Spa(A,A^+)$ defined by Huber. Moreover, for any rational subspace $U\subset \Spa(A,A^+)$ with preimage $U'\subset X$, the map
\[
(A,A^+)_\solid\to (\mathcal O_X(U'),\mathcal M_X(U'))
\]
extends uniquely to an isomorphism
\[
(\mathcal O_{\Spa(A,A^+)}(U),\mathcal O_{\Spa(A,A^+)}^+(U))_\solid\cong (\mathcal O_X(U'),\mathcal M_X(U'))
\]
in the following cases:
\begin{enumerate}
\item if $A$ is discrete;
\item if $A$ admits a noetherian ring of definition $A_0$ over which it is finitely generated;
\item if $A$ is Tate and sheafy.
\end{enumerate}
\end{proposition}

\begin{remark} Let us quickly indicate the construction of the map $|X|\to \Spa(A,A^+)$. An element $a\in A$ is defined to be $\leq 1$ on some affine subspace $U\subset X$ if and only if the $(\mathcal O_X(U),\mathcal M_X(U))$-module $\mathbb Z((T^{-1}))\otimes_{\mathbb Z[T]} (\mathcal O_X(U),\mathcal M_X(U))$ vanishes, where $T$ is mapped to $a$. Extending this definition to local fractions $\frac ab$, this defines a valuation on $A$ locally on $X$. This valuation is however a priori not continuous. Let $\Spv(A,A^+)$ be the spectral space of all (not necessarily continuous) valuations $|\cdot|$ on $A$ with $|A^+|\leq 1$ and $|A^{\circ\circ}|<1$. Then $\Spa(A,A^+)\subset \Spv(A,A^+)$ and there is a retraction, cf.~\cite[Proposition 2.6, Theorem 3.1]{HuberContVal}. Now the construction above gives a natural map $|X|\to \Spv(A,A^+)$, functorial in $(A,A^+)$. Composing with the retraction $\Spv(A,A^+)\to \Spa(A,A^+)$ gives the desired map $|X|\to \Spa(A,A^+)$.

We warn the reader that the retraction $\Spv(A,A^+)\to \Spa(A,A^+)$ is not natural in $(A,A^+)$ for non-adic maps $(A,A^+)\to (B,B^+)$ (and hence neither is $|X|\to \Spa(A,A^+)$). For example, pulling back
\[
\Spa(\mathbb Z[T^{\pm 1}],\mathbb Z)\subset \Spa(\mathbb Z[T],\mathbb Z)
\]
to $\Spa(\mathbb Q_p\langle T\rangle,\mathbb Z_p\langle T\rangle)$, one gets a punctured unit disc, which is not quasicompact anymore. In the context of analytic spaces, the pullback would be given by an analytic ring structure on $\mathbb Q_p\langle T\rangle[T^{-1}]$, which is not a Huber ring anymore. Relatedly, there are points of the analytic space $\AnSpec((\mathbb Q_p\langle T\rangle,\mathbb Z_p\langle T\rangle)_\solid)$ that are ``infinitesimally close to the origin''. These will map in $\Spa(\mathbb Q_p\langle T\rangle,\mathbb Z_p\langle T\rangle)$ to the origin, while their image in $\AnSpec((\mathbb Z[T],\mathbb Z)_\solid)$ will not map to the origin of $\Spa(\mathbb Z[T],\mathbb Z)$.

A different way to understand the situation is to regard only $\Spv(A,A^+)$ and the map $|X|\to \Spv(A,A^+)$ as fundamental, and then observe that Huber was only able to define the structure on those rational subsets of $\Spv(A,A^+)$ that arise via pullback from $\Spa(A,A^+)$, as only on those rational subsets one (usually) gets analytic rings corresponding to Huber pairs again.
\end{remark}

The proposition says that in cases (1) -- (3), Huber has defined the ``correct'' structure sheaf. From the proposition, one sees that there is a functor from adic spaces glued out of $\Spa(A,A^+)$'s with $A$ satisfying one of the given conditions, to analytic spaces.

One can also use the results of these lectures to recover (and slightly generalize) the results on gluing of vector bundles (or coherent sheaves), cf.~e.g.~\cite{KedlayaLiu2}. For this, note that from Proposition~\ref{prop:descent} one formally gets descent for dualizable objects. Now we have the following result.

\begin{theorem}[{\cite{Andreychev}}] Let $(A,A^+)$ be a Huber pair. Then the dualizable objects in $\mathcal D((A,A^+)_\solid)$ are equivalent to the perfect complexes of $A(\ast)$-modules.

In general, for any analytic ring $(A,\mathcal M)$, the category of dualizable objects in $\mathcal D(A,\mathcal M)$ is generated (under cones and retracts) by objects of the form $\mathrm{cone}(1-f: P\to P)$ where $P\in \mathcal D(A,\mathcal M)$ is compact and $f: P\to P$ is a trace-class endomorphism.
\end{theorem}

For a similar result in complex geometry, see \cite[Section VI.1]{ScholzeRealLLC}. In general, one sees that dualizable objects come from ``Fredholm operators'': Namely, if $f:P\to P$ is a trace-class map, then $\mathrm{cone}(1-f: P\to P)$ is always a dualizable object. Here $1-f$ can be understood as a Fredholm operator, and hence the previous theorem is related to classical functional-analytic properties of Fredholm operators.

Let us end these lectures by giving one somewhat funny example of an analytic ring; it is essentially the Novikov ring appearing in symplectic geometry.

For any $0<r<1$, consider the ring
\[
\mathbb Z((T^{\mathbb R}))_r = \{\sum_{x\in \mathbb R} a_x T^x\mid \sum_x |a_x| r^x<\infty\},\footnote{In particular, at most countably many $a_x$ are nonzero.}
\]
equipped with the following condensed ring structure. Write
\[
\mathbb Z((T^{\mathbb R}))_r = \bigcup_{c>0} \mathbb Z((T^{\mathbb R}))_{r,\leq c}
\]
where
\[
\mathbb Z((T^{\mathbb R}))_{r,\leq c} = \{\sum_{x\in \mathbb R} a_x T^x\mid \sum_x |a_x| r^x\leq c\},
\]
which we give the structure of a compact Hausdorff space by writing it as a quotient of the compact Hausdorff space
\[
\{(x_1^+,x_1^-,x_2^+,x_2^-,\ldots)\mid x_i^+,x_i^-\in [(\log(c)-\log(i))/\log r,\infty], \sum_i (r^{x_i^+} + r^{x_i^-})\leq c\}
\]
via
\[
(x_1^+,x_1^-,x_2^+,x_2^-,\ldots)\mapsto \sum_i (T^{x_i^+} - T^{x_i^-}).
\]
In particular, the map $\mathbb R\cup \{\infty\}\to \mathbb Z((T^{\mathbb R}))_r: x\mapsto T^x$ defines a map of condensed sets, and hence $0=T^\infty$ is connected to $1=T^0$ in $\mathbb Z((T^{\mathbb R}))_r$. We note that $\mathbb Z((T^{\mathbb R}))_r$ has a rather peculiar condensed structure, in that a null-sequence $x_0,x_1,\ldots$ need not go to $0$ in the $\leq c$-sense: For example, the sequence
\[
1-T,1-T^{1/2},\ldots,1-T^{1/n},\ldots
\]
converges to $1-T^0 = 1-1=0$, so is a null-sequence, but none of the terms lies in $\mathbb Z((T^{\mathbb R}))_{r,\leq 1}$.

Again, we can define similarly for any finite set $S$ a subset
\[
\mathbb Z((T^{\mathbb R}))[S]_{r,\leq c}\subset \mathbb Z((T^{\mathbb R}))[S]_r,
\]
and then define for a profinite set $S=\varprojlim_i S_i$
\[
\mathcal M(S,\mathbb Z((T^{\mathbb R}))_r)_{\leq c} = \varprojlim_i \mathbb Z((T^{\mathbb R}))[S_i]_{r,\leq c},
\]
and
\[
\mathcal M(S,\mathbb Z((T^{\mathbb R}))_r) = \bigcup_{c>0} \mathcal M(S,\mathbb Z((T^{\mathbb R}))_r)_{\leq c}.
\]
Finally, we also define
\[
\mathbb Z((T^{\mathbb R}))_{>r} = \bigcup_{r'>r} \mathbb Z((T^{\mathbb R}))_{r'}
\]
and the same for a profinite set $S$.

Note that there is an action of the condensed semigroup $\mathbb R_{\geq 1}$ on $\mathbb Z((T^{\mathbb R}))_{>r}$ via $T\mapsto T^t$ for $t\in \mathbb R_{\geq 1}$. Restricted to a prime $p\in \mathbb R_{\geq 1}$, this action defines a Frobenius lift, so $\mathbb Z((T^{\mathbb R}))_{>r}$ is a novel kind of $\lambda$-ring where the individual Frobenius lifts combine into an action of $\mathbb R_{\geq 1}$. Note that the $\mathbb R_{\geq 1}$-action actually also induces canonical isomorphisms between $\mathbb Z((T^{\mathbb R}))_{>r}$ for different values of $r$, so this ring is actually independent of the choice of $r$ (up to canonical isomorphism).

\begin{theorem} This construction defines an analytic ring structure on $\mathbb Z((T^{\mathbb R}))_{>r}$. In fact, it is induced from the analytic ring structure on $\mathbb Z((T))_{>r}$ via base change along $\mathbb Z[T]\to \mathbb Z[T^{\mathbb R}]$.
\end{theorem}

Analytic geometry over $\mathbb Z((T^{\mathbb R}))_{>r}$ is then some funny crossover between arithmetic geometry and complex or real-analytic geometry, allowing actions of connected groups like $\mathbb R_{\geq 1}$ in characteristic $p$; we believe it is unlike anything that has been studied.

What does the corresponding analytic space look like? It maps at least to the Berkovich spectrum of $\mathbb Z$:

\begin{proposition} There is a natural map from $|\AnSpec(\mathbb Z((T))_{>r})|$ to the Berkovich spectrum $M(\mathbb Z)$ of $\mathbb Z$, of all multiplicative norms $|\cdot|: \mathbb Z\to \mathbb R_{\geq 0}$. Restricting to $|\AnSpec(\mathbb Z((T^{\mathbb R}))_{>r})|$, the map
\[
|\AnSpec(\mathbb Z((T^{\mathbb R}))_{>r})|\to M(\mathbb Z)
\]
is equivariant for the action of $\mathbb R_{\geq 1}$ (acting via rescaling the norm on $M(\mathbb Z)$).
\end{proposition}

In this picture, the various theories $(\mathbb R,\mathcal M_{<p})$ map to the point $p\in (0,1]$ on the half-line corresponding to real valuations. In particular, taking the limit for $p\to 0$ makes one enter the arithmetic part of the Berkovich space $M(\mathbb Z)$, so we see that the analytic ring structures on $\mathbb R$ are inextricably linked to arithmetic, which arises in the limit $p\to 0$.

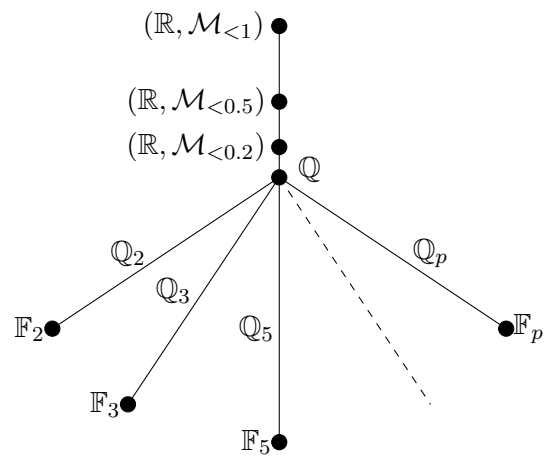
\begin{figure}
\centering

\begin{tikzpicture}[auto]
\draw (0.4,0.1) node {$\mathbb Q$};
\filldraw (0,0) circle [radius=0.1];
\draw (-3.3,-2) node {$\mathbb F_2$};
\filldraw (-3,-2) circle [radius=0.1];
\draw (-2.3,-3) node {$\mathbb F_3$};
\filldraw (-2,-3) circle [radius=0.1];
\draw (-0.3,-3.5) node {$\mathbb F_5$};
\filldraw (0,-3.5) circle [radius=0.1];
\draw (3.3,-2) node {$\mathbb F_p$};
\filldraw (3,-2) circle [radius=0.1];
\draw (0,0) -- (0,2);
\draw (0,0) -- (-3,-2);
\draw (0,0) -- (-2,-3);
\draw (0,0) -- (0,-3.5);
\draw[dashed] (0,0) -- (2,-3);
\draw (0,0) -- (3,-2);
\draw (-2,-1) node {$\mathbb Q_2$};
\draw (-1.4,-1.5) node {$\mathbb Q_3$};
\draw (-0.3,-2) node {$\mathbb Q_5$};
\draw (2,-1) node {$\mathbb Q_p$};
\draw (-1.1,1) node {$(\mathbb R,\mathcal M_{<0.5})$};
\filldraw (0,1) circle [radius=0.1];
\draw (-1,2) node {$(\mathbb R,\mathcal M_{<1})$};
\filldraw (0,2) circle [radius=0.1];
\draw (-1.1,0.4) node {$(\mathbb R,\mathcal M_{<0.2})$};
\filldraw (0,0.4) circle [radius=0.1];
\end{tikzpicture}

\caption{The Berkovich space $M(\mathbb Z)$}
\end{figure}

\bibliographystyle{amsalpha}

\bibliography{Analytic}

\end{document}